\documentclass[reqno,12pt,oneside]{amsart}

\usepackage{amsmath,amsfonts,amscd,flafter,epsf,amssymb,epsfig,subfigure,bm,url, mathrsfs, amsthm}
\usepackage{enumitem}
\usepackage[margin=1in]{geometry}
\newgeometry{left=3.18cm, right=3.18cm, top=2.54cm, bottom=2.54cm}

\usepackage{pgfplots,tikz}
\usepackage{xkeyval,tkz-base,tikz-3dplot}
\usetikzlibrary{arrows,calc,shapes,automata,backgrounds,petri,positioning,fadings,decorations.pathreplacing, spy, intersections}
\tikzset{>=stealth}
\usetikzlibrary{external}
\usetikzlibrary{decorations.markings}
\usetikzlibrary{lindenmayersystems}
\tikzset{
external/system call={
xelatex \tikzexternalcheckshellescape
-halt-on-error -interaction=batchmode --shell-escape --enable-write18
-jobname "\image" "\texsource"}
}

\def\nextAngle{0}
\tikzset{
next angle/.style={
in=#1+180,
out=\nextAngle,
prefix after command= {\pgfextra{\def\nextAngle{#1}}}
},
start angle/.style={
out=#1,
nangle=#1,
},
nangle/.code={
\def\nextAngle{#1}
}
}

\usepackage{etoolbox}
\AtBeginEnvironment{tikzcd}{\tikzexternaldisable}
\AtEndEnvironment{tikzcd}{\tikzexternalenable}

\theoremstyle{plain}

\newtheorem{que}{Question}
\newtheorem{thm}{Theorem}[section]
\newtheorem{lemma}[thm]{Lemma}
\newtheorem{cor}[thm]{Corollary}
\newtheorem{prop}[thm]{Proposition}
\newtheorem{assump}[thm]{Assumption}

\theoremstyle{definition}
\newtheorem{example}[thm]{Example}
\newtheorem{mydef}[thm]{Definition}
\newtheorem{algo}[thm]{Algorithm}

\theoremstyle{remark}
\newtheorem{remark}[thm]{Remark}

\usepackage{verbatim}

\newcommand{\Xm}{\overline{X}}
\newcommand{\Ym}{\overline{Y}}
\newcommand{\bbz}{\mathbb{Z}}

\newcommand{\abs}[1]{\left| #1 \right|}
\newcommand{\set}[1]{\left\{ #1 \right\}}
\newcommand{\bracket}[1]{\left( #1 \right)}
\newcommand{\class}[1]{\langle #1 \rangle}

\usepackage{hyperref}
\hypersetup{colorlinks = true, allcolors = blue}

\begin{document}

\title{Knot Floer homology and the fundamental group of $ (1,1) $ knots}

\author{Matthew Hedden}
\address{Mathematics Department, Michigan State University, East Lansing, MI 48824, USA }
\email{mhedden@math.msu.edu}

\author{Jiajun Wang}
\address{LMAM, School of Mathematics Sciences, Peking University, Beijing, 100871, P. R. China}
\email{wjiajun@pku.edu.cn}

\author{Xiliu Yang}
\address{LMAM, School of Mathematics Sciences, Peking University, Beijing, 100871, P. R. China}
\email{yangxiliu@pku.edu.cn}

\date{\today}

\maketitle

\begin{abstract}
We give an algorithm for computing the knot Floer homology of a $ (1,1) $ knot from a particular presentation of its fundamental group.
\end{abstract}


\section{Introduction}

Heegaard Floer homology (introduced by Peter Ozsv\'ath and Zolt\'an Szab\'o \cite{OS01a,OS01b}) provides various topological invariants for three- and four-manifolds. A  null homologous knot in a three-manifold induces a filtration on the Heegaard Floer chain complex and its homology, called the knot Floer homology, which was introduced by Peter Ozsv\'ath and Zolt\'an Szab\'o \cite{OS04a} and Jacob Rasmussen \cite{Ras03}  independently.

The fundamental group is an important invariant for three-manifolds. The geometrization theorem (proposed by William Thurston \cite{Thu82} and proved by Grisha Perelman \cite{Per02, Per03a, Per03b}) implies that the fundamental group completely determines a closed, orientable, irreducible, three-manifold up to orientation  except for lens spaces (see \cite{AFW15} for details). Hence the (hat version) Heegaard Floer homology can be completely determined by the fundamental group of the three-manifold. 

Heegaard Floer homology has a close relationship with the fundamental group. For an integer homology three-sphere $Y$, the Euler characteristic of $HF^+_{\rm red}(Y)$ minus half its \emph{correction term} equals Casson's invariant $\lambda(Y)$ under certain normalization \cite{OS03}, and Casson's invariant is the algebraic counting of $SU(2)$ representations of $\pi_1(Y)$. On the homology level, the instanton Floer homology is a categorification of Casson's invariant \cite{Flo88a} and its generators are the $SU(2)$ representations of the fundamental group \cite{Flo88a, KM11}. Seiberg-Witten Floer homology is isomorphic to Heegaard Floer homology (by the work of Hutchings \cite{Hut02}, Hutchings-Taubes \cite{HT07,HT09}, Taubes \cite{Tau10a,Tau10b,Tau10c,Tau10d,Tau10e}, and Kutluhan-Lee-Taubes \cite{KLT10a,KLT10b,KLT10c,KLT11,KLT12} or Colin-Ghiggini-Honda \cite{CGH12b, CGH12c,CGH12a}), and Witten \cite{Wit94} conjectured a relation between the Seiberg-Witten and Donaldson invariants.

It is interesting to find a concrete connection between the fundamental group and the Heegaard Floer homology. In \cite{OS04c}, Ozsv\'ath and Szab\'o asked the following two closely related questions:

\begin{que}\cite[Question 7]{OS04c}\label{question:OS_knots}
	Let $ K $ be a knot in $ S^3 $. Is there an explicit relationship between the fundamental group of the knot complement $ S^3 \setminus K $ and the knot Floer homology $ \widehat{HFK}(S^3, K) $?
\end{que}

\begin{que}\cite[Question 8]{OS04c}\label{question:OS_three_manifold}
Is there an explicit relationship between the Heegaard Floer homology and the fundamental group of a three-manifold?
\end{que}

We study Question \ref{question:OS_knots} for $ (1,1) $ knots. $ (1,1) $ knots are those knots which can be placed in one-bridge position with respect to a genus one Heegaard splitting of the three-sphere. $(1,1)$ knots form a large family of knots: torus knots and 2-bridge knots are all $(1,1)$ knots. Fujii \cite{Fuj96} showed that the Alexander polynomial of any knot can be realized as the Alexander polynomial of some $ (1,1) $ knot.

From the perspective of knot Floer homology, $ (1,1) $ knots are particularly appealing. It was observed by Goda, Matsuda, and Morifuji in \cite{GMM05} that $ (1,1) $ knots are exactly those knots which can be presented by a genus one doubly-pointed Heegaard diagram and their knot Floer homology can be computed combinatorially (see also \cite{OS04a}). The diagrammatic characterization of the $ (1, 1) $ L-space knots and the $ (1, 1) $ almost L-space knots is given in \cite{GLV18,BZ23}. Nie \cite{Nie19} constructed infinitely many $ (1, 1) $ knots which are topologically slice, but not smoothly slice, Himino \cite{Him24} constructed infinitely many mutually non-concordant hyperbolic $ (1,1) $ knots whose Upsilon invariants are convex. Antonio, Celoria, and Stipsicz \cite{ACS21} proved that a family of the alternating torus knots are linearly independent in the concordance group and a family of twist knots are linearly independent in the concordance group.

The fundamental group of a $(1,1) $ knot have a presentation with two generators and one relator.  Our main result is the following:

\begin{thm}\label{thm:main}
Let $K$ be a $(1,1)$ knot in $ S^3 $. Given a two-generator one-relator presentation $\pi_1(S^3 \setminus K)=\langle X,Y\,|\,R(X, Y)\rangle $ of its fundamental group coming from a genus one doubly-pointed Heegaard diagram, $\widehat{HFK}(S^3, K)$ can be computed directly from the relator $R(X,Y)$.
\end{thm}

The computation is provided by Algorithm \ref{algo:main}. The proof of Theorem \ref{thm:main} relies on the fact that the presentation from a genus one Heegaard diagram contains enough information about the universal cover of the diagram, which was employed to compute $\widehat{HFK}(S^3,K)$ for certain knots in \cite[Section 6]{OS04a} and for (1,1) knots in \cite{GMM05}. 
Algorithm \ref{algo:main} can be generalized for $(1,1)$ knots in lens spaces (Section \ref{sec:lens}).


It would be interesting to generalize Theorem \ref{thm:main} to knots with Heegaard diagrams of higher genus, or remove that the presentation of the fundamental group arise from a genus one Heegaard diagram. Algorithm \ref{algo:main} actually applies to a slightly more general group presentations of $(1,1)$ knots (Section \ref{sec:alexander}), for which Algorithm \ref{algo:main} does compute the Alexander polynomial (Corollary \ref{cor:algorithm_alex_polynomial_pseudo_geometric}) but may not compute $\widehat{HFK}(S^3, K)$ (Example \ref{example:pseudo_but_not_homology}).

Algorithm \ref{algo:main} does not only apply to the fundamental group of $(1,1)$ knots, it also applies to any \emph{pseudo-geometric} two-generator one-relator group presentations. Though we do not know whether the homology yields an invariant of the group, and what properties it captures.

This paper is organized as follows. In Section \ref{sec:pre}, we briefly introduce the knot Floer homology and $ (1,1) $ knots. In Section \ref{sec:bigon}, we give a method to find all (primitive) bigons and the basepoints contained in them from the presentation. In Section \ref{sec:alg}, we describe the algorithm from the special presentation to the knot Floer homology, and give some examples. 
In Section \ref{sec:lens}, we extend the algorithm for $ (1,1) $ knots in lens spaces. In Section \ref{sec:alexander}, we discuss our algorithm on general group presentations.

\subsection*{Acknowledgement} We thank Cheng Chang, Xuezhi Zhao, Shengyu Zou and Yanqing Zou for helpful discussions. We thank Cheng Chang for computer programming.

The second author is partially supported by NSFC grant 12131009 and National Key R\&D Program of China 2020YFA0712800.

\clearpage
\section{Preliminaries on Heegaard Floer homology and $(1,1)$ knots}\label{sec:pre}

We recall some facts about the knot Floer homology and $(1,1)$ knots. See \cite{OS04a, Ras03, OS11} for details on knot Floer homology. 

\subsection{Knot Floer homology}

A null-homologous knot $ K $ in a closed oriented 3-manifold $ Y $ can be represented by a \emph{doubly-pointed Heegaard diagram} $ \mathcal{H} := (\Sigma, \bm{\alpha}, \bm{\beta}, w, z) $, where
\begin{itemize}
	\item $ \Sigma $ is a closed, oriented surface of genus $ g $ in $Y$ which splits $Y$ into two handlebodies, denoted by $ V_{\alpha} $ and $ V_{\beta} $;
	\item $ \bm{\alpha}: = \{ \alpha_1, \cdots, \alpha_g \} $ (resp. $ \bm{\beta}:= \{ \beta_1, \cdots, \beta_g \} $) is a set of attaching circles for handlebody $ V_{\alpha} $ (resp. $ V_{\beta} $) that are homologically independent in $H_1(\Sigma)$. Denote by $ D_{\alpha_i} $ and $ D_{\beta_i} $ their attaching disks;
	\item $ w$ and $z$ are two basepoints in $ \Sigma - \bm{\alpha} - \bm{\beta} $;
	\item the knot $ K $ is specified by a proper arc joining $ z $ to $ w $ in $ V_{\alpha} - \bigcup_i D_{\alpha_i} $, and a proper arc joining $ w $ to $ z $ in $ V_{\beta} - \bigcup_i D_{\beta_i} $.
\end{itemize}

The $g$-th symmetric product $ {\rm Sym}^g(\Sigma) := \Sigma^{\times g}/S_g $ of $\Sigma$, where $ S_g $ is the symmetric group on $ g $ elements, has a symplectic structure induced from a complex structure on $\Sigma$, so that the two $ g $-dimensional tori $ \mathbb{T}_{\alpha} = \alpha_1 \times \cdots \times \alpha_g $ and $ \mathbb{T}_{\beta} = \beta_1 \times \cdots \times \beta_g $ are Lagrangian submanifolds (by Perutz \cite{Per08}). A intersection point $ {\bf x} \in \mathbb{T}_{\alpha} \cap \mathbb{T}_{\beta} $ is a $ g $-tuples $ \{x_1, \cdots, x_g\} $ such that each $ x_i \in \alpha_i \cap \beta_{\sigma(i)} $ for some permutation $ \sigma \in S_g $. Given two intersection points $ {\bf x}, {\bf y} \in \mathbb{T}_{\alpha} \cap \mathbb{T}_{\beta} $, let $ \pi_2 ({\bf x}, {\bf y}) $ be the set of homotopy classes of \emph{Whitney disks} connecting $ {\bf x} $ and $ {\bf y} $:
$$ \{ u: \mathbb{D} \rightarrow {\rm Sym}^g (\Sigma) \,|\, u(-i) = {\bf x}, u(i) = {\bf y}, u(a) \subset \mathbb{T}_{\alpha}, u(b) \subset \mathbb{T}_{\beta} \}, $$
where $ \mathbb{D} $ is the unit disk in $\mathbb{C} $ whose boundary consists of two arcs $ a = \{ z \in \partial \mathbb{D}\,|\,{\rm Re}(z) \geq 0\} $ and $ b = \{ z \in \partial \mathbb{D}\,|\,{\rm Re}(z) \leq 0\} $. The \emph{multiplicity} $ n_w (\phi) $ of $ \phi \in \pi_2({\bf x}, {\bf y})$ at $ w \in\Sigma $ is defined to be the algebraic intersection number of $ \phi $ with $ \{w\} \times {\rm Sym}^{g-1}(\Sigma) $. The moduli space $ {\mathcal{M}}(\phi) $ of pseudo-holomorphic representatives of $ \phi $ has a natural $\mathbb{R}$ action, and let $ \widehat{\mathcal{M}}(\phi) $ be the unparametrized moduli space $ \mathcal{M}(\phi) /\mathbb{R} $. The expected dimension of $ {\mathcal{M}}(\phi) $ is determined by the \emph{Maslov index} $ \mu (\phi) $ of $\phi$.

The chain complex $ \widehat{CFK}(\mathcal{H}) $ is a free Abelian group generated by the intersection points $ {\bf x} \in \mathbb{T}_{\alpha} \cap \mathbb{T}_{\beta} $ with the differential defined by
\begin{equation}\label{eqn:knot_floer_differential}
\widehat{\partial}_K {\bf x} = \sum_{{\bf y} \in \mathbb{T}_{\alpha} \cap \mathbb{T}_{\beta}}\sum_{\{\phi \in \pi_2({\bf x}, {\bf y}) \,|\, \mu(\phi) = 1, n_w(\phi) = n_z(\phi) = 0 \}} \# \widehat{\mathcal{M}}(\phi) \cdot {\bf y}.
\end{equation}
$ (\widehat{CFK}(\mathcal{H}), \widehat{\partial}_K) $ is a chain complex whose homology $\widehat{HFK}(Y,K)$ is an invariant of the knot $ K $ in $ Y $, called the \emph{knot Floer homology} of $K$ \cite{OS04a, Ras03}.

There are two gradings on $\widehat{CFK}(S^3,K)$. The \emph{Alexander grading} is the unique function $ F: \mathbb{T}_{\alpha} \cap \mathbb{T}_{\beta} \rightarrow \mathbb{Z} $ satisfying
\begin{equation}\label{eqn:alex_grading_relative}
F({\bf x}) - F({\bf y}) = n_z(\phi) - n_w(\phi),\quad \forall\, \phi \in \pi_2({\bf x}, {\bf y})
\end{equation}
and the additional symmetry
\begin{equation}\label{eqn:alex_grading_symmetry}
\# \{ {\bf x}\,|\,F({\bf x}) = i \} \equiv \# \{ {\bf x}\,|\,F({\bf x}) = -i \} \pmod{2},\quad \forall\,i \in \mathbb{Z}.
\end{equation}
For $\bf x,\bf y\in \mathbb{T}_\alpha\cap\mathbb{T}_\beta$, the \emph{relative Maslov grading} or the \emph{homological grading} satisfies
$${\rm gr}({\bf x}, {\bf y}) = \mu(\phi) - 2n_w(\phi), \quad \forall\, \phi \in \pi_2({\bf x}, {\bf y}). $$
A Heegaard diagram of $Y$ can be obtained by removing the basepoint $z$ from a doubly-pointed Heegaard diagram of $ (Y,K)$, and $\widehat{HF}(Y)$ can be obtained from $\widehat{CFK}(Y, K)$ with additional differentials. When $Y=S^3$, 
we have $ \widehat{HF}(S^3) \cong \mathbb{Z}$, and by defining this homology to be supported in Maslov grading 0, we can define an \emph{absolute Maslov grading} $M:\mathbb{T}_\alpha\cap\mathbb{T}_\beta\to\mathbb{Z}$ with 
$${\rm gr}({\bf x}, {\bf y}) = M({\bf x})- M({\bf y}).$$

It is evident that the differential \eqref{eqn:knot_floer_differential} preserves the Alexander grading and decreases the Maslov grading by one. Let $ \widehat{CFK}_m(S^3, K; s) $ be the subgroup of $\widehat{CFK}(S^3, K)$ generated by those $ {\bf x}\in  \mathbb{T}_{\alpha} \cap \mathbb{T}_{\beta}$ with $ F({\bf x}) = s $ and $ M({\bf x}) = m $, then $ \widehat{HFK}(S^3, K) $ can be decomposed as
$$ \widehat{HFK}(S^3, K) = \bigoplus_{m, s \in \bbz} \widehat{HFK}_m(S^3, K; s). $$
The knot Floer homology is a categorification of the Alexander polynomial.

\begin{thm}{\rm (Ozsv\'ath-Szab\'o \cite{OS04a}, Rasmussen \cite{Ras03})}
	Let $ K $ be a knot in $ S^3 $ and $ \Delta_K(T) $ its symmetrized Alexander polynomial, then
	$$ \sum_{m,s \in \bbz} (-1)^m \cdot {\rm rank} \widehat{HFK}_m (S^3, K; s) \cdot T^s = \Delta_K(T). $$
\end{thm}

\subsection{$(1,1)$ knots and their fundamental groups}\label{subsec:11knot}

\begin{mydef}[{\cite{Dol92}}]\label{def_gn_knot}
	A proper embedded arc $ \gamma $ in a handlebody $ H $ is called \emph{trivial} if there exists an embedded disk $ D $ in $ H $ such that $ \gamma \subseteq \partial D $ and $ \partial D \cap \partial H = \partial D \setminus {\rm Int} \gamma $.  A link $ L $ in a 3-manifold $ M $ is called a \emph{ $ (g, n) $ link} if there exists a genus $ g $ Heegaard splitting $ M = H_1 \cup H_2 $ such that $ L \cap H_i,  (i = 1, 2) $ is the union of $ n $ mutually disjoint properly embedded trivial arcs. The decomposition $ (H_1, L \cap H_1) \cup (H_2, L \cap H_2) $ is called a \emph{$ (g,n) $ decomposition} of $ (M, L) $.
\end{mydef}

We focus on $(1,1)$ knots in $S^3$. $(1,1)$ knots are precisely those knots which admit genus one doubly-pointed Heegaard diagrams, or \emph{$ (1,1) $ Heegaard diagram} for brevity. For more detailed discussions of Heegaard diagrams of $ (1,1) $ knots see \cite{GMM05, Ord06, Ord13}.



A $ (1,1) $ Heegaard diagram $ \mathcal{H} = (T^2,\alpha,\beta,w,z) $ for a $(1,1)$ knot $K$ gives a two-generator one-relator presentation of $ \pi_1(S^3 \setminus K) $ as follows. Orient $ \alpha $ and $ \beta $ so that their intersection number $ [\alpha]\cdot[\beta] = +1 $ and let $ t_{\alpha} $ be an oriented arc connecting $ w $ to $ z $ in $ T^2\setminus \alpha $. Travel along $\beta$ for a full round and record its intersection with $\alpha$ and $t_\alpha$: write $X$/$X^{-1}$ for a positive/negative intersection with $\alpha$, and $Y$/$Y^{-1}$ for a positive/negative intersection with $t_\alpha$. Let $R(X,Y)$ be the resulting word, then we have
\begin{equation}\label{eqn:fund_group_presentation}
	\pi_1(S^3 \setminus K) \cong \langle X,Y\,|\,R(X,Y) \rangle.
\end{equation}

To see that this is indeed a presentation of $ \pi_1(S^3 \setminus K) $, we stabilize the Heegaard splitting as follows. Let $V_\alpha$ ($V_\beta$ respectively) be the solid torus in the Heegaard splitting defined by $\mathcal{H}$ for which $\alpha$ ($\beta$ respectively) bounds a disk. We push the interior of $t_\alpha$ into $ V_\alpha$ and denote the new arc by $ \tilde{t}_\alpha$, so that $t_\alpha$ and $\tilde{t}_\alpha$ co-bounds a disk $D_1$. Take a tubular neighborhood $N(\tilde{t}_{\alpha})$ of $ \tilde{t}_{\alpha}$ in $ V_\alpha$ so that $N(\tilde{t}_\alpha)\cap \beta=\emptyset$ and $D_1\setminus N(\tilde{t}_{\alpha})$ is a disk, denoted by $D_2$. Let $V^\prime_\alpha=V_\alpha\setminus N(\tilde{t}_{\alpha})$ and $V^\prime_\beta$ be the closure of $V_\alpha\cup N(\tilde{t}_{\alpha})$. Then we have a genus two Heegaard splitting
\begin{equation}\label{eqn:stablization_splitting}
S^3 = V_\alpha' \bigcup_{\Sigma_2} V_\beta',
\end{equation}
with Heegaard surface $ \Sigma_2 = \partial V_{\alpha}' = \partial V_{\beta}'$. Let $\alpha_2=\partial D_2$. Let $\beta_2$ be a belt circle of $N(\tilde{t}_{\alpha})$ which is a meridian of $K$. The Heegaard diagram $ (\Sigma_2, \{ \alpha, \alpha_2\}, \{\beta, \beta_2\}) $ specifies the Heegaard splitting \eqref{eqn:stablization_splitting} and 
$$ \mathcal{H}' = (\Sigma_2, \{\alpha,\alpha_2\}, \{\beta\}) $$
is a Heegaard diagram for the knot complement $ S^3 \setminus K $, i.e., $ S^3 \setminus K $ can be obtained from a genus two surface $ \Sigma_2 $ by first attaching two-handles $ C_\alpha, C_{\alpha_2}, C_\beta $ so that the boundaries of the meridian disks of the two-handles are identified with $\alpha, \alpha_2, \beta $ respectively, and then attaching a three-ball to $ \partial (\Sigma_2 \cup C_\alpha \cup C_{\alpha_2}) = S^2 $. In this language, it is clear that the presentation \eqref{eqn:fund_group_presentation} specifies $ \pi_1(S^3 \setminus K) $: attaching $ C_\alpha $, $ C_{\alpha_2} $, and the three-ball to $ \Sigma_2 $ yields a genus two handlebody with fundamental group free on two generators. Note that $ \beta \cap \alpha_2 = \beta \cap t_{\alpha} $. According to Van Kampen's theorem, the relator obtained by attaching two-handle $ C_\beta $ is $ R(X, Y) $ described as above. 

\begin{example}
	\begin{figure}[htbp]
		\centering
		\begin{minipage}{0.4\linewidth}
			\def \globalscale {7}
			\begin{tikzpicture}[y=0.80pt, x=0.80pt, yscale=-\globalscale, xscale=\globalscale, inner sep=0pt, outer sep=0pt]
			\begin{scope}
			\path[draw=black,very thick] (26.4583,36.3802) ellipse (0.4909cm and 0.3547cm);
			
			\path[draw=black,very thick] (20.4504,36.1338) .. controls
			(20.4504,36.1338) and (21.9965,37.7932) .. (23.1051,38.3646) .. controls
			(25.6718,39.6875) and (27.5103,39.6283) .. (29.9141,38.3646) .. controls
			(31.5223,37.5191) and (32.6259,36.0380) .. (32.6259,36.0380);
			
			\path[draw=black,very thick] (21.1667,36.8048) .. controls
			(21.1667,36.8048) and (22.3669,35.6039) .. (23.3787,34.8746) .. controls
			(24.1786,34.2980) and (25.4723,33.9278) .. (26.4583,33.9278) .. controls
			(27.4444,33.9278) and (29.0116,34.4690) .. (29.8115,35.0456) .. controls
			(30.8234,35.7749) and (31.9161,36.8601) .. (31.9161,36.8601);
			
			\path[draw=black,very thick] (26.7343,39.3017) .. controls
			(27.7760,39.3202) and (28.7355,41.9548) .. (28.7540,43.7339) .. controls
			(28.7771,45.9614) and (27.8335,48.8452) .. (26.8759,48.9479);
			
			\path[draw=black,very thick, dashed] (26.7511,39.2850) .. controls (26.2752,39.2974) and
			(24.9695,41.0798) .. (24.8976,43.7716) .. controls (24.8344,46.1358) and
			(25.7247,48.8727) .. (26.7864,48.9479);
			
			\path[fill=black] (26.7,38.5) node{$\alpha$};
			
			\path[draw=green, thick, dashed] (27.2888,42.0978) .. controls (9.7258,37.9379) and
			(10.3381,31.1523) .. (23.3064,30.3461) .. controls (36.2746,29.5399) and
			(46.4467,43.5783) .. (29.6951,45.6714);
			\path[fill=black] (39.2,43) node{$ \tilde{t}_\alpha $};

			\path[draw=red,thick] (37.1224,26.6162) ..
			controls (41.6970,30.5445) and (38.9846,39.0747) .. (31.4732,41.3857) ..
			controls (29.8446,41.8867) and (27.3229,42.1047) .. (27.3229,42.1047);

			\path[draw=red,thick, dashed] (36.7916,26.4643) .. controls (36.7916,26.4643) and
			(34.4147,25.2100) .. (32.0606,26.4991) .. controls (30.3807,27.4190) and
			(28.6815,29.7609) .. (27.9787,31.1286) .. controls (27.2760,32.4964) and
			(26.8569,33.7772) .. (26.8569,33.7772);

			\path[draw=red,thick] (26.7508,33.8425) ..
			controls (26.2768,33.5415) and (22.5353,32.3683) .. (20.7495,33.2884) ..
			controls (19.2510,34.0605) and (18.3790,35.9822) .. (18.2051,37.6589) ..
			controls (17.9816,39.8124) and (18.6266,42.3919) .. (20.2509,43.8234) ..
			controls (22.6452,45.9336) and (26.7038,46.1450) .. (29.6449,45.6735);

			\path[fill=black] (41,36) node{$ t_\beta $};

			\path[draw=blue, very thick] (32.8757,38.8441) ..
			controls (29.6435,40.6467) and (25.9240,40.1546) .. (25.8224,41.7732) ..
			controls (25.7208,43.3919) and (27.0755,43.7998) .. (28.7364,43.8626) ..
			controls (30.3972,43.9254) and (31.4212,44.6690) .. (31.4034,45.9138) ..
			controls (31.3856,47.1586) and (29.5702,47.7222) .. (27.9674,47.7239) ..
			controls (7.9052,47.7447) and (5.2971,25.7744) .. (26.3253,26.5041) ..
			controls (35.9513,26.8381) and (37.8497,36.0702) .. (32.8757,38.8441) --
			cycle;
			\path[fill=black]  (16.5,43)  node{$ \beta $};

			\draw (27.2888,42.0978) circle(0.3);
			\draw[fill=black] (27.2888,43) node{$ w $};
			\draw[fill=black] (29.6951,45.6714) circle(0.3);
			\draw[fill=black] (29.6951,46.5) node{$ z $};
			\end{scope}
			\end{tikzpicture}
		\end{minipage}
		\qquad \qquad \qquad
		\begin{minipage}{0.4\linewidth}
			\begin{tikzpicture}
			\def\y{2.3}
			\draw[very thick, decoration={markings, mark=at position 1 with {\arrow[scale=1.5]{>}}}, postaction={decorate},name path=alpha] (0,0) -- (2*\y,0);
			\draw[very thick] (0,2*\y) -- (2*\y,2*\y);
			\draw[thick,dashed] (0,0) -- (0,2*\y)
			(2*\y, 0) -- (2*\y,2*\y);
			\draw (1.85*\y, 0.1*\y) node {$ \alpha $};
			
			\draw[thick, green, decoration={markings, mark=at position 0.4 with \arrow{>}}, postaction={decorate}] (1.35*\y,0.3*\y) -- (0.65*\y,1.7*\y);
			\draw (1.35*\y,0.3*\y)circle(0.07) node[below] {$ w $};
			\draw[fill=black] (0.65*\y,1.7*\y) circle(0.07) node[above] {$ z $};
			\draw[green] (\y, 0.8*\y) node[below] {$ t_{\alpha} $};
			
			\draw[very thick,blue, decoration={markings, mark=at position 0.3 with \arrow{>}}, postaction={decorate},name path=beta] (0.3*\y, 0)
			to[start angle=90,next angle=90] (1.7*\y, 2*\y)
			(1.7*\y, 0) -- (1.7*\y, 0.3*\y)
			to[start angle=90,next angle=180] (1.35*\y, 0.65*\y)
			to[next angle=270] (\y, 0.3*\y) -- (\y, 0)
			(\y, 2*\y) -- (\y, 1.7*\y)
			to[start angle=-90,next angle=-180] (0.65*\y, 1.35*\y)
			to[next angle=-270] (0.3*\y, 1.7*\y) -- (0.3*\y, 2*\y);
			\draw[blue] (1.45*\y, 1.25*\y) node {$ \beta $};
			
			\fill [name intersections={of=beta and alpha, name=i, total=\t}]
			[red, every node/.style={below right, black, opacity=1}]
			\foreach \s in {1,...,\t}{(i-\s) circle (0.07) node[below] {\small $x_{\s}$}};
			\end{tikzpicture}
		\end{minipage}
		\caption{A Heegaard diagram for the right hand trefoil $ T_{2,3}$.}
		\label{fig:RHT}
	\end{figure}
	
The right hand trefoil $ T_{2,3} $ has a Heegaard diamgram $(T^2,\alpha,\beta,w,z)$ as illustrated in Figure \ref{fig:RHT}. The $T_{2,3}$ is given by $\tilde{t}_{\alpha} \cup t_{\beta}$, where the arc $ \tilde{t}_{\alpha} $ in the solid torus is described as above, and $t_{\beta} $ is an arc in $T^2\setminus\beta$ connecting $z$ and $w$. For convenience, we represent the torus as a rectangle with opposite sides glued together, and push $ \tilde{t}_{\alpha} $ to $ t_{\alpha} $ in the torus, as shown in the figure on the right. By choosing orientations on each curve, we obtain a presentation of $ \pi_1(S^3 \setminus T_{2,3}) $ ($ X $-letters are labelled in order):
	$$ \pi_1(S^3 \setminus T_{2,3}) \cong \langle X,Y \,|\, R(X,Y) = X_1 \Ym X_2 Y \Xm_3 Y \rangle.  $$
\end{example}

Throughout the paper, we use $ \Xm $ for $ X^{-1} $ and $ \Ym $ for $ Y^{-1} $ for brevity. 

\begin{assump}\label{assump:relator}
Let $ K $ be a $ (1,1) $ knot in $ S^3 $ and $\langle X, Y\,|\,R(X, Y)\rangle $ be a presentation of $ \pi_1(S^3 \setminus K) $ obtained as above. We make the following assumptions:
\begin{enumerate}
\item\label{assump:relator_1} The curves $ \alpha $ and $ \beta $ are oriented so that $ [\alpha]\cdot[\beta] = +1 $, that is
$$ \# \{X \,|\, X \in R(X,Y)\} - \#\{\Xm \,|\, \Xm \in R(X,Y)\} = +1;$$
\item\label{assump:relator_2} $ \# \{X \,|\, X \in R(X,Y)\} \geq 2 $, or equivalently, $K$ is knotted;
\item\label{assump:relator_3} The relator $R(X,Y)$ is \emph{cyclically reduced}, that is, there are no subwords of the form $ X\Xm, \Xm X, Y\Ym$ or $\Ym Y $ up to cyclic permutations of the relator.
\end{enumerate}
\end{assump}

Assumption \ref{assump:relator} makes no actual restrictions. \eqref{assump:relator_1} is obtained by proper choice of orientations of $\alpha$ and $\beta$. \eqref{assump:relator_2} is obtained for nontrivial knots. Given a (1,1) Heegaard diagram, we can isotope the beta curve so that every bigon contains $z$ and/or $w$ (see \cite{GLV18}), and the resulting presentation will satisfy \eqref{assump:relator_3}. 

We write the relator as:
$$ R(X,Y) = R_1R_2\cdots R_i\cdots R_m, $$
where $ R_i \in \{X, \Xm, Y, \Ym \} $, and denote by $ R_i^j $ the proper subword $ R_iR_{i+1} \cdots R_j $, with the convention that $ R_{m+k} = R_k $. We are interested in the subword $ R_i^j $ with $\set{R_i, R_j}=\set{X, \Xm} $. We label the $X$-letters in $R_i^j$ form $1$ to $n$, and the $i$-th $X$-letter in $R_i^j$ is either $X_i$ or $\Xm_i$. Denote by $W_a^b$ ($1\leqslant a<b\leqslant n$) the subword of $R_i^j$ from $X_a$/$\overline{X}_a$ to $X_b$/$\overline{X}_b$ (hence $ W_1^n = R_i^j $). We use the capital $X$ and $Y$ to denote the letters in the relator $R$ and the lowercase $x$ and $y$ to denote the intersections points in the Heegaard diagram $\mathcal{H}$.

\newpage
\section{Bigons and disk words for $(1,1)$ knots}\label{sec:bigon}

Let $K$ be a $(1,1)$ knot in $S^3$ and $ \class{X, Y\,|\,R(X, Y)}$ be a presentation of $\pi_1(S^3\setminus K)$ obtained from a $ (1,1) $ Heegaard diagram $ \mathcal{H} = (T^2, \alpha, \beta, z, w) $ of $(S^3,K)$, satisfying Assumption \ref{assump:relator}. To compute $ \widehat{HFK}(S^3, K) $ from the presentation, we need to find all the information of the chain complex $ \widehat{CFK}(S^3, K) $. The generators are the intersection points of the curves $ \alpha $ and $ \beta $ in the torus $ T^2 $, which correspond to $ X $-letters ($ X $ or $ \Xm $) in the relator $ R $. On the other hand, the differential is determined by bigons in Heegaard diagram $ \mathcal{H} $. Since the relator $R$ is cyclically reduced, all bigons contain at least one basepoint, and it follows that the differential is identically zero. Therefore it suffices to decide the Alexander and Maslov grading of $ \widehat{CFK}(S^3, K) $ from the relator $ R $. In order to do so, one may wish to find all bigons from $ R $, and decide the number of basepoints contained in each bigon. However, there are difficulties to achieve that (see Example \ref{knot5_2}). Instead, we will work on a special class of bigons (Definition \ref{def:primitive_bigon}) that can be determined from $R$, and are sufficient to determine the gradings of $ \widehat{CFK}(S^3, K) $.

\begin{mydef}\label{def:primitive_bigon}
Let $ D $ be a Whitney disk in $T^2$ connecting $x_1$ and $x_n$. Let $ \widetilde{x}_1 $ and $ \widetilde{x}_n $ be lifts of $x_1$ and $x_n$ respectively  in $ \mathbb{C} $, and $ \widetilde{D} $ be a lift of $ D $ connecting $ \widetilde{x}_1 $ and $ \widetilde{x}_n $. $ \partial \widetilde{D} $ consists of an $ \alpha $ arc $ a $ and a $ \beta $ arc $ b $. Suppose $ \widetilde{\alpha} $ is the lift of $ \alpha $ containing $ \widetilde{x}_1 $. The bigon $ \widetilde{D} $ is called \emph{primitive} if $ \widetilde{\alpha} \cap b = \{ \widetilde{x}_1, \widetilde{x}_n \} $, and a bigon $ D $ is called \emph{primitive} if it has a primitive lift.
\end{mydef}

To describe our algorithm, we use the universal cover $\mathbb{C}$ of the torus $ T^2$. We also use $ \alpha $, $ \beta $, $t_{\alpha}$, $ z $ and $ w $ to denote their lifts in $ \mathbb{C} $ when there is no confusion. The keypoint is that bigons can lift to embedded bigons in the universal cover and it is possible to count the number of lifted basepoints inside primitive bigons from the relator $R$.

Primitive bigons are enough to determine the gradings of all generators. If $D$ is a bigon connecting two points $x_1$ and $x_n$ in $\mathbb{C}$, suppose $ \alpha \cap b = \{ x_1, x_2, \cdots, x_n \} $, labeled by the order of them on $b$. Then $ D $ is a combination of primitive bigons $ D_i, (i=1, \cdots, n-1) $ that connecting two adjacent $ x_i $ and $ x_{i+1} $, which can be used to compute the grading difference between $x_1$ and $x_n$.

Suppose $ D $ is a primitive bigon. Let $ n_z(D) $ (resp. $ n_w(D) $) be the number of basepoints $ z $ (resp. $ w $) contained in $ D $ and write
\begin{equation}\label{eqn:PD}
P(D) = (n_z(D), n_w(D)).
\end{equation}
In this paper we always consider the grading shift from point $ x_1 $ to $ x_n $, so the numbers $ n_z(D) $ and $ n_w(D) $ may be negative. In fact, the sign of $ n_z(D) $ and $ n_w (D) $ depends on the orientation of $ D $, as defined below:

\begin{mydef}
	Let $ D $ be a bigon in $ \mathbb{C} $ whose boundary consists of an $ \alpha $ arc $ a $ and a $ \beta $ arc $b$ ($b$ is oriented). The \emph{orientation} of $D$ is \emph{positive} if $D$ is on the right side of the arc $ b $, and \emph{negative} if otherwise. 
\end{mydef}

\begin{mydef}\label{def:disk_word}
	Let $ \varphi(W_1^n) = \# \{ X \,|\, X \in W_1^n\} - \# \{ \Xm \,|\, \Xm \in W_1^n\} $. A subword $ W_1^n $ of the relator $ R(X,Y) $ is called a \emph{disk word} if its two ends have opposite signs and $ \varphi(W_1^n) = 0 $. A disk word $ W_1^n $ is called \emph{primitive} if $ \varphi(W_1^k) \neq 0 $ for all $ 1 < k < n $.
\end{mydef}

\begin{mydef}\label{def:height}
	Let $ W_1^n $ be a disk word, let $\hat{X}_i$ denote $X_i$ or $\overline{X}_i$. Define the \emph{height} (relative to the endpoints) of $\hat{X}_i$ as
	$$ h(\hat{X}_i) = \varphi (W_1^i) - \epsilon(\hat{X}_i), \ (1 \leq i \leq n), $$
	where $ \epsilon(\hat{X}_i) = 1 $ if $ \hat{X}_i $ has the same sign as $ \hat{X}_1 $, and zero otherwise.
	If a subword $W_i^j$ of $W_1^n$ is primitive and $ h(\hat{X}_i) = h(\hat{X}_j)=s$, we say that the primitive disk word $W_i^j$ has \emph{height $s$} in $ W_1^n $. 
	A primitive disk word $ W_1^n $ is \emph{elementary} if $ n = 2 $, \emph{upward} if $\hat{X_1}=X_1$ and \emph{downward} if $\hat{X_1}=\overline{X}_1$.
\end{mydef}

For a primitive disk word $W_1^n$ with $n>2$, all $ h(\hat{X}_i), (1 < i < n) $ have the same sign. And we have $ h(\hat{X}_i) > 0 $ ($1 < i < n$) if $W_1^n$ is upward, and $ h(\hat{X}_i) < 0 $ ($1 < i < n$) if $W_1^n$ is downward.

It is clear that the corresponding word of a bigon $ D $ in $ \mathbb{C} $ is a disk word, since the intersection points of each lifting of $ \alpha $ with the $ \beta $ part of $ \partial D $ occur in pairs of opposite signs. We call a disk word $ W_1^n $ a \emph{real bigon} if it corresponds to a bigon in $ \mathbb{C} $. An example of a disk word that does not correspond to a bigon is given in the following

\begin{example}\label{knot5_2}
	Figure \ref{fig:knot5_2} shows a lifting of a Heegaard diagram compatible with the knot $ 5_2 $ in $ S^3 $. The curve $ \beta $ gives the relator
	$$ R(X, Y) = X_1 Y \Xm_2 Y X_3 \Ym \Xm_4 \Ym X_5 Y \Xm_6 Y X_7. $$
	The two disk words $ W_1^4: X_1 Y \Xm_2 Y X_3 \Ym \Xm_4 $ and $ W_3^6: X_3 \Ym \Xm_4 \Ym X_5 Y \Xm_6 $ are the same if we replace $ Y $ by $ \Ym $. $ W_1^4 $ corresponds to a real bigon, but $ W_3^6 $ does not.
	\begin{figure}[htbp]
		\centering
		\begin{tikzpicture}
			\def\x{2.5}
			\def\y{2.3}
			\draw[very thick, decoration={markings, mark=at position 1 with {\arrow[scale=1.5]{>}}}, postaction={decorate},name path = alpha] (-0.2*\x,0) -- (3.5*\x,0) node[right] {$ \alpha $};
			\draw[thick, green, decoration={markings, mark=at position 0.5 with \arrow{>}}, postaction={decorate}] (1*\x,0.5*\y) -- (0.5*\x, \y) node[above right] {$ t_{\alpha} $} -- (0*\x,1.5*\y);
			\draw[thick, green, decoration={markings, mark=at position 0.5 with \arrow{>}}, postaction={decorate}] (3*\x,0.5*\y) -- (2*\x,1.5*\y);
			\draw (1*\x,0.5*\y)circle(0.07) node[below] {$ w $};
			\draw (3*\x,0.5*\y)circle(0.07);
			\draw[fill=black] (0*\x,1.5*\y) circle(0.07) node[below] {$ z $};
			\draw[fill=black] (2*\x,1.5*\y) circle(0.07);
			\draw[thick, green] (0.5*\x,-\y) -- (0*\x,-0.5*\y);
			\draw[thick, green] (2.5*\x,-\y) -- (2*\x,-0.5*\y);
			\draw[fill=black] (0*\x,-0.5*\y) circle(0.07);
			\draw[fill=black] (2*\x,-0.5*\y) circle(0.07);
			\draw[very thick,blue, decoration={markings, mark=at position 0.05 with {\arrow[scale=1.5]{>}}}, postaction={decorate},name path = beta] (1.5*\x, -\y)
			to[start angle=130, next angle=90] (1.3*\x, -0.5*\y) node[above left] {$ \beta $}
			to[next angle=90] (1.3*\x, 0) -- (1.3*\x, 0.3*\y)
			to[next angle=180] (\x, 0.7*\y)
			to[next angle=-100] (0.7*\x, 0.3*\y)
			to[next angle=180] (0, -0.6*\y)
			to[next angle=90] (-0.1*\x, -0.5*\y)
			to[next angle=70] (0.5*\x,0.3*\y)
			to[next angle=0] (\x, 0.8*\y)
			to[next angle=-90] (1.5*\x, 0.3*\y) -- (1.5*\x, -0.3*\y)
			to[next angle=0] (2*\x, -0.85*\y) -- (2.3*\x, -0.85*\y)
			to[next angle=90] (3.1*\x, -0.2*\y) -- (3.1*\x, 0.4*\y)
			to[next angle=180] (3*\x, 0.6*\y)
			to[next angle=-90] (2.9*\x, 0.4*\y) -- (2.9*\x, -0.2*\y)
			to[next angle=180] (2.2*\x, -0.7*\y)
			to[next angle=140] (1.8*\x, -0.6*\y)
			to[next angle=90] (1.7*\x, 0) -- (1.7*\x, 0.4*\y)
			to[next angle=130] (1.5*\x, \y)
			to[next angle=90] (1.3*\x, 1.5*\y);
			
			\fill [name intersections={of=beta and alpha, name=i, total=\t}]
			[red, every node/.style={below right, black, opacity=1}]
			\foreach \s in {1,...,\t}{(i-\s) circle (0.07) node {\footnotesize $x_{\s}$}};
		\end{tikzpicture}
		\caption{A Heegaard diagram of the knot $ 5_2 $ (lifting to $ \mathbb{C} $).}
		\label{fig:knot5_2}
	\end{figure}
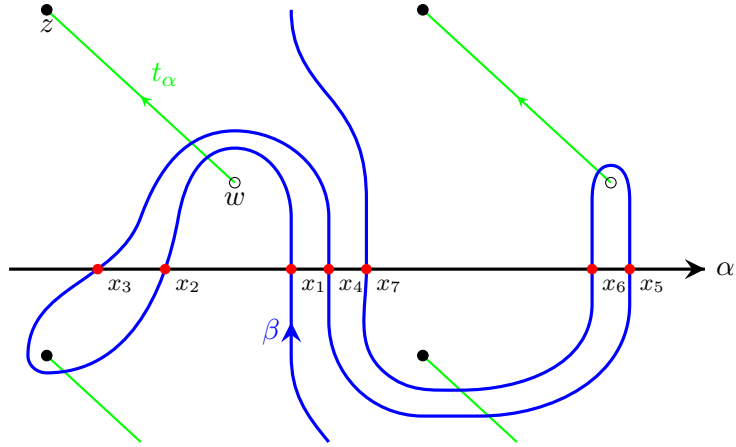
\end{example}

\begin{lemma}\label{lemma:primitive_word_are_real}
	A primitive disk word corresponds to a primitive bigon in $\mathbb{C}$ and vice versa.
\end{lemma}

\begin{proof}
	Let $ W_1^n $ be a primitive disk word. Suppose it corresponds to the subarc $ b $ of $ \beta $ whose two endpoints $ x_1, x_n $ are in the lifts $ \alpha $ and $ \alpha' $, respectively. Note that two sets $ \{k\,|\,x_k \in \alpha \cap b, 1<k \} $ and $ \{k\,|\,x_k \in \alpha' \cap b, k<n \} $ cannot both be empty sets. Assuming without loss of generality that the former is not empty and that $ i $ is its minimum, then the subarc of $ b $ from $ x_1 $ to $ x_i $ and $ \alpha $ bounds a primitive bigon, so that $ W_1^i $ is a disk word. Since $ W_1^n $ is assumed to be primitive, it follows that $ i = n $ so that $ W_1^n $ actually corresponds to a primitive bigon in $\mathbb{C}$.

	Suppose $ D $ is a primitive bigon. If its corresponding word is not primitive, then $ \varphi(W_1^k) = 0 $ for some $ 1<k<n $. Let $ j = \min \{k \,|\, \varphi(W_1^n) = 0, 1<k<n \} $, then the subword $ W_1^j $ is primitive by definition. As proved above, $ W_1^j $ corresponds to a primitive bigon, so that its endpoint $ x_j \in \alpha $, which cannot occur.
\end{proof}

Since the primitive disk word and the primitive bigon are in one-to-one correspondence, we will use these two terms alternatively by abuse of notation. All primitive bigons can be enumerated since the length of the relator $ R $ is finite. What we need to do is to find the number of basepoints contained in each of them directly from $R(X,Y)$.

\subsection{Elementary bigons}

In the beginning, we consider what the local diagram that corresponds to the simplest primitive disk word $ XY\Xm $ should look like. As illustrated in Figure \ref{fig:2choices}, it could have two choices:
\begin{figure}[htbp]
	\centering
	\begin{tikzpicture}
	\def\x{1}
	\draw[very thick, blue, decoration={markings, mark=at position 0.5 with \arrow{>}}, postaction={decorate}] (2.4*\x,-0.2*\x) -- (2.4*\x,0.3*\x)
	to[start angle=90, next angle=180](2*\x,0.7*\x)
	to[next angle=270](1.6*\x,0.3*\x) --(1.6*\x,-0.2*\x);
	\draw[very thick, blue, decoration={markings, mark=at position 0.5 with \arrow{>}}, postaction={decorate}] (0.1*\x,-0.2*\x) -- (0.1*\x,1.3*\x) node[below left] {$\beta$}
	to[start angle=90, next angle=0] (0.5*\x,1.7*\x)
	to[next angle=-90] (0.9*\x,1.3*\x) -- (0.9*\x,-0.2*\x);
	\draw[very thick, decoration={markings, mark=at position 1 with \arrow{>}}, postaction={decorate}] (-0.5*\x,0) -- (3*\x,0) node[right] {$ \alpha $};
	\draw[thick, green, decoration={markings, mark=at position 0.5 with \arrow{>}}, postaction={decorate}] (2*\x,0.5*\x) -- (0.5*\x,1.5*\x);
	\draw (2*\x,0.5*\x)circle(0.07) node[below] {$w$};
	\draw[fill=black] (0.5*\x,1.5*\x) circle(0.07) node[below] {$z$};
	\draw[green] (1.5*\x, 1*\x) node[above] {$t_{\alpha}$};
	\end{tikzpicture}
	\caption{Two choices of an arc whose word is $ XY\Xm $.}
	\label{fig:2choices}
\end{figure}
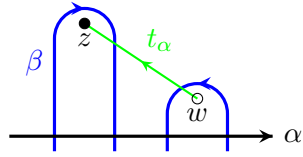

\noindent Note that the two cases cannot occur simultaneously in a diagram. If both are present, consider the word that follows the letter $\Xm $ to the next $X$-letter, it can only be $ Y\Xm $ or $ \Xm $ since the relator is reduced. And then continue the discussion with the next letter $ \Xm $, it can be seen that none of the $ \Xm $ is followed by $ X $, which cannot happen. When we choose $ XY\Xm $ to correspond to the one containing the basepoint $ z $ (resp. $ w $), the word $ \Xm YX $ must correspond to a bigon that contains the basepoint $ w $ (resp. $ z $), see Figure \ref{fig:switch_zw}:
\begin{figure}[htbp]
	\centering
	\begin{tikzpicture}
	\def\x{1}
	\begin{scope}[shift={(0,0)}]
	\draw[very thick, blue, decoration={markings, mark=at position 0.5 with \arrow{>}}, postaction={decorate}] (1.6*\x,2.2*\x) -- (1.6*\x,0.7*\x)
	to[start angle=-90, next angle=0] (2*\x,0.3*\x)
	to[next angle=90] (2.4*\x,0.7*\x) -- (2.4*\x,2.2*\x);
	\draw[very thick, blue, decoration={markings, mark=at position 0.5 with \arrow{>}}, postaction={decorate}] (0.1*\x,-0.2*\x) --(0.1*\x,1.3*\x) node[below left] {$ \beta $}
	to[start angle=90, next angle=0](0.5*\x,1.7*\x)
	to[next angle=270](0.9*\x,1.3*\x) -- (0.9*\x,-0.2*\x);
	\draw[very thick, decoration={markings, mark=at position 1 with \arrow{>}}, postaction={decorate}] (-0.5*\x,0) -- (3*\x,0) node[right] {$\alpha_0$};	\draw[very thick, decoration={markings, mark=at position 1 with \arrow{>}}, postaction={decorate}] (-0.5*\x,2*\x) -- (3*\x,2*\x) node[right] {$\alpha_1$};
	\draw[thick, green, decoration={markings, mark=at position 0.5 with \arrow{>}}, postaction={decorate}] (2*\x,0.5*\x) -- (0.5*\x,1.5*\x);
	\draw (2*\x,0.5*\x)circle(0.07) node[above] {$w$};
	\draw[fill=black] (0.5*\x,1.5*\x) circle(0.07) node[below] {$z$};
	\draw[green]  (1.4*\x, 1*\x)  node[above] {$t_{\alpha}$};
	\end{scope}

	\begin{scope}[shift={(5*\x,0)}]
	\draw[very thick, blue, decoration={markings, mark=at position 0.5 with \arrow{>}}, postaction={decorate}] (1.6*\x,2.2*\x) -- (1.6*\x,1.6*\x)
	to[start angle=-90, next angle=0](2*\x,1.2*\x)
	to[next angle=90](2.4*\x,1.6*\x) -- (2.4*\x,2.2*\x);
	\draw[very thick, blue, decoration={markings, mark=at position 0.5 with \arrow{>}}, postaction={decorate}] (0.1*\x,-0.2*\x) -- (0.1*\x,0.4*\x) node[below left] {$\beta$}
	to[start angle=90,next angle=0](0.5*\x,0.8*\x)
	to[next angle=270](0.9*\x,0.4*\x) -- (0.9*\x,-0.2*\x);
	\draw[very thick, decoration={markings, mark=at position 1 with \arrow{>}}, postaction={decorate}] (-0.5*\x,0) -- (3*\x,0) node[right] {$\alpha_0$};
	\draw[very thick, decoration={markings, mark=at position 1 with \arrow{>}}, postaction={decorate}] (-0.5*\x,2*\x) -- (3*\x,2*\x) node[right] {$\alpha_1$};
	\draw[thick, green, decoration={markings, mark=at position 0.5 with \arrow{>}}, postaction={decorate}] (2*\x,1.5*\x) -- (0.5*\x,0.5*\x);
	\draw (2*\x,1.5*\x)circle(0.07) node[above] {$z$};
	\draw[fill=black]  (0.5*\x,0.5*\x)  circle(0.07) node[below] {$w$};
	\draw[green]  (1.4*\x, 1*\x)  node[below] {$t_{\alpha}$};
	\end{scope}
	\end{tikzpicture}
	\caption{The first figure shows that $ XY\Xm $ contains the basepoint $ z $. The second figure is obtained by switching two basepoints $ z $ and $ w $.}
	\label{fig:switch_zw}
\end{figure}
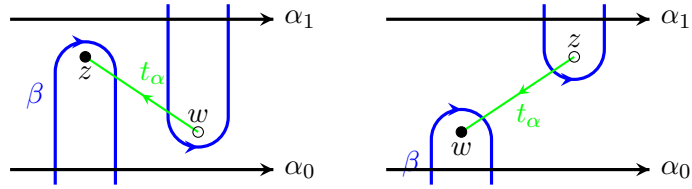

\noindent Thus, after exchanging two basepoints $z$ and $w$\protect\footnote{This procedure will reverse the orientation of the arc $ t_{\alpha} $. To obtain the same relator $ R $, we can reverse the orientation of $ T^2 $ and $ \alpha $. The resulting diagram $ (-T^2, -\alpha, \beta, w, z) $ is compatible with the mirror image $ K^* $ of the knot $ K $. Although the knot Floer homology can distinguish chirality, the fundamental group cannot. Our algorithm yields the knot Floer homology of $ K $ or $ K^* $.}
, we can assume that the disk word $ XY\Xm $ corresponds to the bigon containing the basepoint $ w $. 
So that the bigons corresponding to the four elementary disk words (or say \emph{elementary bigons}) can be drawn as in Figure \ref{fig:elementary_bigons}. Moreover, the orientation of each elementary bigon, as well as the number of basepoints it contains can be found, as shown in Table \ref{table:elementary_bigons}.

\begin{figure}[htbp]
	\centering
	\begin{tikzpicture}
	\def\x{0.8}
	\begin{scope}[shift={(0,0)}]
	\draw[very thick, blue, decoration={markings, mark=at position 0.4 with \arrow{>}}, postaction={decorate}] (2*\x,-0.3*\x) -- (2*\x,0.5*\x)
	to[start angle=90, next angle=180](1.05*\x,1.8*\x) -- (0.95*\x,1.8*\x)
	to[next angle=270](0,0.5*\x) -- (0,-0.3*\x);
	\draw[very thick, decoration={markings, mark=at position 0.5 with \arrow{>}}, postaction={decorate}] (-0.5*\x,0) -- (2.5*\x,0);

	\draw[thick, green, decoration={markings, mark=at position 0.4 with \arrow{>}}, postaction={decorate}] (\x,0.9*\x) -- (\x,2.5*\x);
	
	\draw (\x, 0.9*\x)circle(0.07) node[right] {$w$};
	\draw[fill=black] (\x, 2.5*\x)circle(0.07) node[right] {$z$};
	\draw[fill=black] (0,0)circle(0.05);
	\draw[fill=black] (2*\x, 0)circle(0.05);
	\draw (\x, -.6*\x) node[below] {$XY\Xm $};
	\draw (\x, -0.1) node[below] {$\alpha$};
	\draw[blue]   (-0.2*\x, 1*\x)node {$\beta$};
	\end{scope}

	\begin{scope}[shift={(4*\x,0)}]
	\draw[very thick, blue, decoration={markings, mark=at position 0.4 with \arrow{>}}, postaction={decorate}] (0,-0.3*\x) -- (0*\x,0.5*\x)
	to[start angle=90,next angle=0](0.95*\x,1.8*\x) -- (1.05*\x,1.8*\x)
	to[next angle=270](2*\x,0.5*\x) -- (2*\x,-0.3*\x);
	\draw[very thick, decoration={markings, mark=at position 0.5 with \arrow{>}}, postaction={decorate}] (-0.5*\x,0) -- (2.5*\x,0);

	\draw[thick, green, decoration={markings, mark=at position 0.4 with \arrow{>}}, postaction={decorate}] (\x,0.9*\x) -- (\x,2.5*\x);
	
	\draw (\x, 0.9*\x)circle(0.07) node[right] {$w$};
	\draw[fill=black] (\x, 2.5*\x)circle(0.07) node[right] {$z$};
	\draw[fill=black] (0,0)circle(0.05);
	\draw[fill=black] (2*\x, 0)circle(0.05);
	\draw (\x, -.6*\x) node[below] {$X\overline{YX}$};
	\draw (\x, -0.1) node[below] {$\alpha$};
	\draw[blue]   (-0.2*\x, 1*\x)node {$\beta$};
	\end{scope}

	\begin{scope}[shift={(8*\x,0)}]
	\draw[very thick, blue, decoration={markings, mark=at position 0.4 with \arrow{>}}, postaction={decorate}] (2*\x,2.3*\x) --  (2*\x,1.5*\x)
	to[start angle=-90,next angle=-180](1.05*\x,0.2*\x) -- (0.95*\x,0.2*\x)
	to[next angle=-270](0,1.5*\x) -- (0,2.3*\x);
	\draw[very thick, decoration={markings, mark=at position 0.5 with \arrow{>}}, postaction={decorate}] (-0.5*\x,2*\x) -- (2.5*\x,2*\x);

	\draw[thick, green, decoration={markings, mark=at position 0.4 with \arrow{>}}, postaction={decorate}] (\x,-0.5*\x) -- (\x,1.1*\x);
	
	\draw (\x,-0.5*\x)circle(0.07) node[right] {$w$};
	\draw[fill=black] (\x,1.1*\x)circle(0.07) node[right] {$z$};
	\draw[fill=black] (0,2*\x)circle(0.05);
	\draw[fill=black] (2*\x, 2*\x)circle(0.05);
	\draw[black] (\x, -.6*\x) node[below] {$\Xm YX$};
	\draw[black]    (\x, 2.1*\x) node[above] {$\alpha$};
	\draw[blue]   (-0.2*\x, 1*\x)node {$\beta$};
	\end{scope}

	\begin{scope}[shift={(12*\x,0)}]
	\draw[very thick, blue, decoration={markings, mark=at position 0.4 with \arrow{>}}, postaction={decorate}] (0,2.3*\x) -- (0,1.5*\x)
	to[start angle=-90,next angle=0] (0.95*\x,0.2*\x) -- (1.05*\x,0.2*\x)
	to[next angle=90] (2*\x,1.5*\x) -- (2*\x,2.3*\x);
	\draw[very thick, decoration={markings, mark=at position 0.5 with \arrow{>}}, postaction={decorate}] (-0.5*\x,2*\x) -- (2.5*\x,2*\x);

	\draw[thick, green, decoration={markings, mark=at position 0.4 with \arrow{>}}, postaction={decorate}] (\x,-0.5*\x) -- (\x,1.1*\x);
	
	\draw (\x,-0.5*\x)circle(0.07) node[right] {$w$};
	\draw[fill=black] (\x,1.1*\x)circle(0.07) node[right] {$z$};
	\draw[fill=black] (0,2*\x)circle(0.05);
	\draw[fill=black] (2*\x, 2*\x)circle(0.05);
	\draw[black] (\x, -.6*\x) node[below] {$\overline{XY}X$};
	\draw[black]    (\x, 2.1*\x) node[above] {$\alpha$};
	\draw[blue]   (-0.2*\x, 1*\x)node {$\beta$};
	\end{scope}
	\end{tikzpicture}
	\caption{Four elementary disk words and their corresponding bigons.}
	\label{fig:elementary_bigons}
\end{figure}
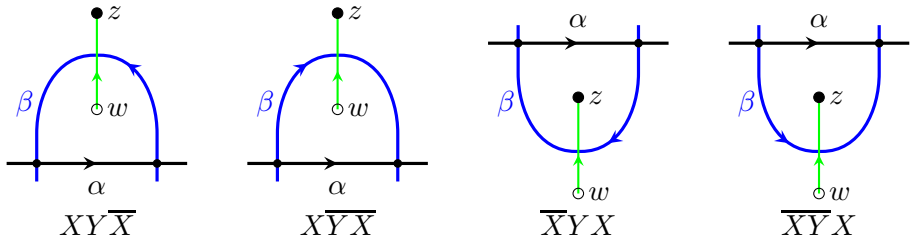

\renewcommand\arraystretch{1.5}
\begin{table}[h]
	\centering
	\begin{tabular}{|c|c|c|c|c|}
	\hline
	elementary disk word & $ XY\Xm $ & $ X\Ym \Xm $ & $ \Xm YX $ & $ \Xm \Ym X $ \\
	\hline
	orientation & negative & positive & positive & negative  \\
	\hline
	$ (n_z, n_w) $ & $ (0, -1) $ & $ (0 ,1) $ & $ (1, 0) $ & $ (-1, 0) $ \\
	\hline
	\end{tabular}
	\vskip\baselineskip
	\caption{Elementary bigons.}
	\label{table:elementary_bigons}
\end{table}

\begin{remark}
By Assumption \ref{assump:relator}, there are always two adjacent $X$-letters in $R$ with opposite signs, and there exist bigons of the form $ X Y^k \Xm $ or $\Xm Y^k X$. We claim that $ k $ must be $ \pm 1 $. If $\abs{k}>1$ on the contrary, then there exists a covering transformation $ \Gamma $ such that $ \beta \cap \Gamma(\beta) \neq \emptyset $, as illustrated in Figure \ref{fig:abs_k_is_1}.
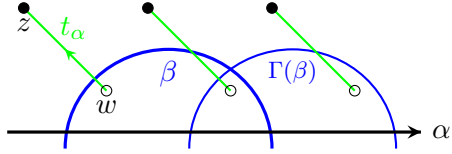
\begin{figure}[htbp]
	\centering
	\begin{tikzpicture}
		\def\x{1.1}
		\draw[very thick, blue, name path=beta] (\x, -0.2*\x)
		to[start angle=90, next angle=0] (2.25*\x,\x) node[below] {$\beta$}
		to[next angle=-90](3.5*\x,-0.2*\x);
		\draw[shift={(1.5*\x, 0)}, thick, blue, name path=beta2] (\x, -0.2*\x)
		to[start angle=90, next angle=0] (2.25*\x,\x) node[below] {\footnotesize $\Gamma(\beta)$}
		to[next angle=-90] (3.5*\x,-0.2*\x);
		\draw[very thick, decoration={markings, mark=at position 1 with \arrow{>}}, postaction={decorate}] (0.3*\x,0) -- (5.3*\x,0) node[right] {$\alpha$};
		\draw[thick, green, decoration={markings, mark=at position 0.5 with \arrow{>}}, postaction={decorate}] (1.5*\x,0.5*\x) -- (0.5*\x,1.5*\x);
		\draw[thick, green] (3*\x,0.5*\x) -- (2*\x,1.5*\x)
		(4.5*\x,0.5*\x) -- (3.5*\x,1.5*\x);
		
		\draw (1.5*\x,0.5*\x)circle(0.07) node[below] {$w$};
		\draw[fill=black] (0.5*\x,1.5*\x) circle(0.07) node[below] {$z$};
		\draw (3*\x,0.5*\x)circle(0.07);
		\draw[fill=black] (2*\x,1.5*\x) circle(0.07);
		\draw (4.5*\x,0.5*\x)circle(0.07);
		\draw[fill=black] (3.5*\x,1.5*\x) circle(0.07);
		\draw[green] (1.1*\x, \x) node[above] {$t_{\alpha}$};
	\end{tikzpicture}
	\caption{\label{fig:abs_k_is_1}
		No subword of the form $ XY^k \Xm $ with $ |k| > 1 $.}
\end{figure}
The argument is similar to Lemma \ref{lemma:point_smaller1}. It follows that there always exist elementary bigons containing basepoint $ z $ and $w$ respectively.
\end{remark}

\subsection{Bigons with height}

Now we consider general primitive bigons. Firstly, the orientation of a primitive bigon can be found from its corresponding word.

\begin{lemma}\label{lemma:orientation}
	If $ D $ is a real bigon, then the number of positive and negative elementary words that it contains differs by one. Moreover, $ D $ is positive if and only if
	$$ \#\{\text{positive elementary bigons in $ D $}\} - \#\{\text{negative elementary bigons in $ D $}\} = 1. $$
\end{lemma}

\begin{proof}
	Suppose $ D $ is a real bigon whose boundary consists of an $ \alpha $ arc $ a $ and a $ \beta $ arc $ b $. Perturb $ \beta $ so that it is perpendicular to $ \alpha $ at all intersection points. Then the total curvature of the arc $ b $ is $ \pi $ (resp. $ -\pi $), counted counterclockwise, if the bigon is on the left (resp. right) when walking along $ b $. That is, $ D $ is positive if and only if the total curvature is $ -\pi $.
	
	By assuming that $\beta $ is perpendicular to $\alpha $ at all intersections, the arc $ b $ will have a non-zero contribution to the total curvature only when it passes near the basepoints. As can be seen from the picture of the elementary bigons (Figure \ref{fig:elementary_bigons}), the arc $b$ on a negative (resp. positive) elementary bigon will contribute $ \pi $ (resp. $ -\pi $). Hence the total curvature is
	$$ \pi \cdot \left(\#\{\text{negative elementary bigons}\} - \#\{\text{positive elementary bigons}\}\right). $$ 
	Thus, $ D $ is positive if and only if $ D $ contains one more positive elementary bigons than negative ones.
\end{proof}

In the rest of the section, let $ W_1^n $ be an upward positive disk word with $ n > 2 $. Let $D$ be a corresponding bigon of $W_1^n$ in $\mathbb{C}$. Suppose that $\partial D=a\cup b$ with $a\subseteq\widetilde{\alpha}$ and $b\subseteq\widetilde{\beta}$. It is clear that $h(x_2)=h(x_{n-1})=1$. Let $ b_1 $ be the subarc from $ x_1 $ to $ x_2 $, and $ b_2 $ be the subarc from $ x_{n-1} $ to $ x_n $. The corresponding subwords for $b_1$ and $b_2$ are $ W_1^2 = X Y^l X $ and $ W_{n-1}^n = \Xm \Ym^{l'} \Xm $ for some integers $ l $ and $ l' $. Write the two subwords as a pair $ (W_1^2, W_{n-1}^n) $ or $ (X Y^l X, \Xm \Ym^{l'} \Xm) $. Denote by $ S $ the square domain bounded by the subarcs $ b_1 $, $ b_2 $ and two lifts of the $ \alpha $ curve. The key to 
calculating $ P(D)$ is to determine the number of basepoints contained in $ S $.

\begin{lemma}\label{lemma:point_smaller1}
	$ \max \set{ \abs{n_z(S)}, \abs{n_{w} (S)} } \leq 1 $.
\end{lemma}
\begin{proof}
	Suppose the square domain $ S $ contains two lifts $ w_1, w_2 $ of the basepoint $ w $. As in Figure \ref{fig:no_two_lifts_of_basepoints}, 
	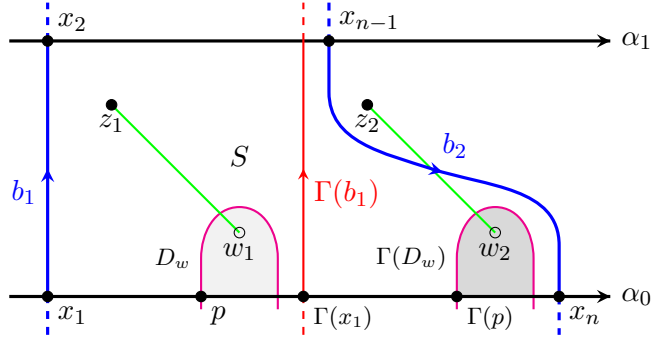
\begin{figure}[htbp]
		\centering
		\begin{tikzpicture}
		\def\x{1.7}
		\draw[thick, magenta, fill = gray!10] (1.2*\x,0*\x) -- (1.2*\x,0.3*\x) 
		to[start angle=90, next angle=0](1.5*\x,0.7*\x)
		to[next angle=-90](1.8*\x,0.3*\x) -- (1.8*\x,0*\x);
		\draw (1.2*\x,0.3*\x) node[left] {\footnotesize $ D_w $};
		\draw[thick, magenta] (1.2*\x,-0.1*\x) -- (1.2*\x,0*\x) (1.8*\x,-0.1*\x) -- (1.8*\x,0*\x);
		\draw[fill=black] (1.2*\x, 0) circle(0.07) node[below right] {$ p $};
		\draw[shift={(2*\x,0)}, thick, magenta, fill = gray!30] (1.2*\x,0*\x) -- (1.2*\x,0.3*\x)
		to[start angle=90, next angle=0](1.5*\x,0.7*\x)
		to[next angle=-90](1.8*\x,0.3*\x) --(1.8*\x,0*\x);
		\draw[thick, magenta] (3.2*\x,-0.1*\x) -- (3.2*\x,0*\x)
		(3.8*\x,-0.1*\x) -- (3.8*\x,0*\x);
		\draw (3.2*\x,0.3*\x) node[left] {\footnotesize $ \Gamma(D_w) $};
		\draw[fill=black] (3.2*\x, 0) circle(0.07) node[below right] {\footnotesize $ \Gamma(p) $};
		\draw[very thick, decoration={markings, mark=at position 1 with \arrow{>}}, postaction={decorate}] (-0.3*\x,0) -- (4.4*\x,0) node[right] {$ \alpha_0 $};
		\draw[very thick, decoration={markings, mark=at position 1 with \arrow{>}}, postaction={decorate}] (-0.3*\x,2*\x) -- (4.4*\x,2*\x) node[right] {$ \alpha_1 $};
		\draw[thick, green] (1.5*\x,0.5*\x) node[black, below]{$ w_1 $} -- (0.5*\x,1.5*\x);
		\draw[thick, green] (3.5*\x,0.5*\x) node[black, below]{$ w_2 $} -- (2.5*\x,1.5*\x);
		\draw (1.5*\x,0.5*\x) circle(0.07);
		\draw (3.5*\x,0.5*\x) circle(0.07);
		\draw[fill=black] (0.5*\x,1.5*\x) circle(0.07) node[below] {$ z_1 $};
		\draw[fill=black] (2.5*\x,1.5*\x) circle(0.07) node[below] {$ z_2 $};
		\draw[very thick, blue, dashed] (0,-0.3*\x) -- (0,0)
		(0,2*\x) -- (0,2.3*\x);
		\draw[very thick, blue, decoration={markings, mark=at position 0.5 with \arrow{>}}, postaction={decorate}] (0,0) -- (0, \x) node[below left] {$ b_1 $} -- (0,2*\x);
		\draw[fill=black] (0, 0) circle(0.07) node[below right] {$ x_1 $};
		\draw[fill=black] (0, 2*\x) circle(0.07) node[above right] {$ x_2 $};
		\draw[very thick, blue, dashed] (2.2*\x,2.3*\x) -- (2.2*\x, 2*\x)
		(4*\x,0*\x) -- (4*\x, -0.3*\x); 
		\draw[very thick, blue, decoration={markings, mark=at position 0.5 with \arrow{>}}, postaction={decorate}] (2.2*\x,2*\x) -- (2.2*\x,1.6*\x)
		to[start angle=-90, next angle=-18] (3*\x,\x) node[above right] {$ b_2 $}
		to[next angle=-90](4*\x,0.4*\x) -- (4*\x,0);
		\draw[fill=black] (4*\x, 0) circle(0.07) node[below right] {$ x_n $};
		\draw[fill=black] (2.2*\x, 2*\x) circle(0.07) node[above right] {$ x_{n-1} $};
		\draw[shift = {(2*\x,0)}, thick, red, decoration={markings, mark=at position 0.5 with \arrow{>}}, postaction={decorate}] (0,0) -- (0, \x) node[below right] {$\footnotesize \Gamma(b_1) $} -- (0,2*\x);
		\draw[shift = {(2*\x,0)}, thick, red, dashed] (0,-0.3*\x) -- (0,0)
		(0,2*\x) -- (0,2.3*\x);
		\draw[fill=black] (2*\x, 0)circle(0.07) node[below right] {\footnotesize $ \Gamma(x_1) $};
		\draw (1.5*\x, 1.1*\x) node {$ S $};
		\end{tikzpicture}
		\caption{\label{fig:no_two_lifts_of_basepoints}
		Domain $ S $ contains at least two lifts of the basepoint $ w $.}
	\end{figure}
we assume that $ w_2 $ is on the right-hand side of $ w_1 $. Note that there exists an elementary bigon containing $ w_1 $, denoted by $ D_w $. Then $ D_w \subseteq S $. Let $ p $ be an endpoint of $ D_w $. Consider the covering transformation $ \Gamma $ which maps $ w_1 $ to $ w_2 $, then $ w_2 \in \Gamma (D_w) $, and $ \Gamma (D_w) \subset S $. The orientation of the lift curve $ \widetilde{\alpha} $ gives an order
	$$ x_1 < p < \Gamma(p) < x_n. $$
	Therefore $ x_1 < \Gamma(x_1) < \Gamma(p) < x_n $. It implies that $ \Gamma (b_1)\subset S$ and $ \Gamma (D)\cap D\neq\emptyset$. It follows that $\widetilde{\beta}\cap\Gamma(\widetilde{\beta})\neq\emptyset$ and we get a contradiction.

The case that $ S $ contains two lifts $ z_1, z_2 $ of $ z $ is similar: the image of $ x_2 $ under the covering transformation $\Gamma'$ that maps $z_1$ to $z_2$ satisfies $x_2 < \Gamma'(x_2) < x_{n-1} $, so $\Gamma'(D) $ and $D$ are overlapped.
\end{proof}

As a consequence, we have $ |l - l'| \leq 1 $. Since $ D $ is upward and positive, we have $ x_1 < x_n $, i.e., the subarc $ b_2 $ is on the right side of $ b_1 $. When $ |l-l'| = 1 $, all possible pairs $ (W_1^2, W_{n-1}^n) $ are shown in Figure \ref{fig:b_diff_word},
\begin{figure}[htbp]
	\centering
	\begin{tikzpicture}
	\def\x{1.25}
	\begin{scope}[shift = {(0,0)}]
	\draw[fill=gray!10] (3.4*\x,0) -- (3.4*\x,0.3*\x)
	to[start angle=90, next angle=180] (1.8*\x,0.9*\x)
	to[next angle=90] (0.2*\x,1.5*\x) -- (0.2*\x,2*\x) -- (1.2*\x,2*\x) -- (1.2*\x,1.8*\x)
	to[start angle=-90,next angle=-5] (2.6*\x,1.2*\x)
	to[next angle=-90]	(4.2*\x,0.6*\x) -- (4.2*\x,0) -- (3.4*\x,0);
	\draw[very thick, blue, decoration={markings, mark=at position 0.5 with \arrow{>}}, postaction={decorate}] (3.4*\x,0) -- (3.4*\x,0.3*\x)
	to[start angle=90,next angle=180] (1.8*\x,0.9*\x) node[below right] {$ b_1 $}
	to[next angle=90] (0.2*\x,1.5*\x) -- (0.2*\x,2*\x);
	\draw[very thick, blue, dashed] (3.4*\x,-0.3*\x) -- (3.4*\x,0) (0.2*\x,2*\x) -- (0.2*\x,2.3*\x);
	\draw[very thick, blue, decoration={markings, mark=at position 0.5 with \arrow{>}}, postaction={decorate}] (1.2*\x,2*\x) -- (1.2*\x,1.8*\x)
	to[start angle=-90,next angle=-5] (2.6*\x,1.2*\x) node[above right] {$ b_2 $}
	to[next angle=-90] (4.2*\x,0.6*\x) -- (4.2*\x,0);
	\draw[very thick, blue, dashed] (1.2*\x,2.3*\x) -- (1.2*\x,2*\x)
	(4.2*\x,0) -- (4.2*\x,-0.3*\x);
	\draw[very thick, decoration={markings, mark=at position 1 with \arrow{>}}, postaction={decorate}] (-0.5*\x,0) -- (4.5*\x,0) node[right] {$\alpha_0 $};
	\draw[very thick, decoration={markings, mark=at position 1 with \arrow{>}}, postaction={decorate}] (-0.5*\x,2*\x) -- (4.5*\x,2*\x) node[right] {$ \alpha_1 $};
	\draw[thick, green] (1.5*\x,0.5*\x) -- (0.5*\x,1.5*\x);
	\draw[thick, green] (3*\x,0.5*\x) -- (2*\x,1.5*\x);
	\draw (1.5*\x,0.5*\x)  circle(0.07);
	\draw (3*\x,0.5*\x)  circle(0.07);
	\draw[fill=black] (0.5*\x,1.5*\x) circle(0.07) node[right] {$ z $};
	\draw[fill=black] (2*\x,1.5*\x) circle(0.07);
	\draw[fill=black] (3.4*\x,0) circle(0.07) node[below left] {$ x_1 $};
	\draw[fill=black] (4.2*\x,0) circle(0.07) node[below right] {$ x_n $};
	\draw (2*\x, -0.5*\x) node[below] {$(XY^lX, \Xm \Ym^{l-1} \Xm)$};
	\end{scope}

	\begin{scope}[shift = {(7*\x, 0)}]
	\draw[fill = gray!10] (2*\x,0*\x) -- (2*\x,0.3*\x)
	to[start angle=90,next angle=175] (0.6*\x,0.9*\x)
	to[next angle=90] (-0.7*\x,1.5*\x) -- (-0.7*\x,2*\x) -- (0*\x,2*\x) -- (0*\x,1.8*\x)
	to[start angle=-90,next angle=-5] (1.8*\x,1.2*\x)
	to[next angle=-90] (3.4*\x,0.5*\x) -- (3.4*\x,0) -- (2*\x, 0);
	\draw[very thick, blue, decoration={markings, mark=at position 0.5 with \arrow{>}}, postaction={decorate}] (2*\x,0) -- (2*\x,0.3*\x)
	to[start angle=90,next angle=175] (0.6*\x,0.9*\x) node[below left] {$ b_1 $}
	to[next angle=90] (-0.7*\x,1.5*\x) -- (-0.7*\x,2*\x);
	\draw[very thick, blue, dashed] (2*\x,-0.3*\x) -- (2*\x,0)
	(-0.7*\x,2*\x) -- (-0.7*\x,2.3*\x);
	\draw[very thick, blue, decoration={markings, mark=at position 0.5 with \arrow{>}}, postaction={decorate}] (0*\x,2*\x) -- (0*\x,1.8*\x)
	to[start angle=-90,next angle=-5] (1.8*\x,1.2*\x) node[above left] {$ b_2 $}
	to[next angle=-90] (3.4*\x,0.5*\x) -- (3.4*\x,0);
	\draw[very thick, blue, dashed] (0,2.3*\x) -- (0,2*\x)
	(3.4*\x,0) -- (3.4*\x,-0.3*\x);
	\draw[very thick, decoration={markings, mark=at position 1 with \arrow{>}}, postaction={decorate}] (-1*\x,0) -- (4*\x,0) node[right] {$ \alpha_0 $};
	\draw[very thick, decoration={markings, mark=at position 1 with \arrow{>}}, postaction={decorate}] (-1*\x,2*\x) -- (4*\x,2*\x) node[right] {$ \alpha_1 $};
	\draw[thick, green] (1.5*\x,0.5*\x) -- (0.5*\x,1.5*\x);
	\draw[thick, green] (3*\x,0.5*\x) -- (2*\x,1.5*\x);
	\draw (1.5*\x,0.5*\x)  circle(0.07);
	\draw (3*\x,0.5*\x)  circle(0.07) node[below] {$ w $};
	\draw[fill=black] (0.5*\x,1.5*\x) circle(0.07);
	\draw[fill=black] (2*\x,1.5*\x) circle(0.07);
	\draw[fill=black] (2*\x,0) circle(0.07) node[below left] {$ x_1 $};
	\draw[fill=black] (3.4*\x,0) circle(0.07) node[below right] {$ x_n $};
	\draw (2*\x, -0.5*\x) node[below] {$(XY^{l-1}X, \Xm \Ym^l \Xm)$};
	\end{scope}

	\begin{scope}[shift = {(0, -4*\x)}]
	\draw[fill = gray!10] (0.7*\x,0) -- (0.7*\x,0.3*\x)
	to[start angle=90,next angle=22] (1.6*\x,1*\x)
	to[next angle=90] (2.5*\x,1.7*\x) -- (2.5*\x,2*\x) -- (3.4*\x,2*\x) -- (3.4*\x,1.6*\x)
	to[start angle=-90,next angle=-150] (2.75*\x,0.9*\x)
	to[next angle=-90] (2.1*\x,0.2*\x) -- (2.1*\x,0) -- (0.7*\x, 0);
	\draw[very thick, blue, decoration={markings, mark=at position 0.5 with \arrow{>}}, postaction={decorate}] (0.7*\x,0) -- (0.7*\x,0.3*\x)
	to[start angle=90,next angle=22] (1.6*\x,1*\x) node[above left] {$ b_1 $}
	to[next angle=90] (2.5*\x,1.7*\x) -- (2.5*\x,2*\x);
	\draw[very thick, blue, dashed] (0.7*\x,-0.3*\x) -- (0.7*\x,0)
	(2.5*\x,2*\x) -- (2.5*\x,2.3*\x);
	\draw[very thick, blue, decoration={markings, mark=at position 0.5 with \arrow{>}}, postaction={decorate}] (3.4*\x,2*\x) -- (3.4*\x,1.6*\x) node[below right] {$ b_2 $}
	to[start angle=-90,next angle=-150] (2.75*\x,0.9*\x)
	to[next angle=-90] (2.1*\x,0.2*\x) -- (2.1*\x,0);
	\draw[very thick, blue, dashed] (3.4*\x,2.3*\x) -- (3.4*\x,2*\x)
	(2.1*\x,0) -- (2.1*\x,-0.3*\x);
	\draw[very thick, decoration={markings, mark=at position 1 with \arrow{>}}, postaction={decorate}] (-0.5*\x,0) -- (4.5*\x,0) node[right] {$ \alpha_0 $};
	\draw[very thick, decoration={markings, mark=at position 1 with \arrow{>}}, postaction={decorate}] (-0.5*\x,2*\x) -- (4.5*\x,2*\x) node[right] {$ \alpha_1 $};
	\draw[thick, green] (1.5*\x,0.5*\x) -- (0.5*\x,1.5*\x);
	\draw[thick, green] (3*\x,0.5*\x) -- (2*\x,1.5*\x);
	\draw (1.5*\x,0.5*\x)  circle(0.07) node[below] {$ w $};
	\draw (3*\x,0.5*\x)  circle(0.07);
	\draw[fill=black] (0.5*\x,1.5*\x) circle(0.07);
	\draw[fill=black] (2*\x,1.5*\x) circle(0.07);
	\draw[fill=black] (0.7*\x,0) circle(0.07) node[below left] {$ x_1 $};
	\draw[fill=black] (2.1*\x,0) circle(0.07) node[below right] {$ x_n $};
	\draw (2*\x, -0.5*\x) node[below] {$(X \Ym^l X, \Xm Y^{l-1} \Xm)$};
	\end{scope}

	\begin{scope}[shift = {(7*\x, -4*\x)}]
	\draw[fill = gray!10] (-0.2*\x,0) -- (-0.2*\x,0.4*\x)
	to[start angle=90,next angle=25] (0.6*\x,1.1*\x)
	to[next angle=90] (1.4*\x,1.8*\x) -- (1.4*\x,2*\x) -- (2.7*\x,2*\x) -- (2.7*\x,1.7*\x)
	to[start angle=-90,next angle=-160]	(1.8*\x,1*\x)
	to[next angle=-90] (0.9*\x,0.3*\x) -- (0.9*\x,0) -- (-0.2*\x, 0);
	\draw[very thick, blue, decoration={markings, mark=at position 0.5 with \arrow{>}}, postaction={decorate}] (-0.2*\x,0) -- (-0.2*\x,0.4*\x)
	to[start angle=90,next angle=25] (0.6*\x,1.1*\x) node[above left] {$ b_1 $}
	to[next angle=90] (1.4*\x,1.8*\x) -- (1.4*\x,2*\x);
	\draw[very thick, blue, dashed] (-0.2*\x,-0.3*\x) -- (-0.2*\x,0)
	(1.4*\x,2*\x) -- (1.4*\x,2.3*\x);
	\draw[very thick, blue, decoration={markings, mark=at position 0.5 with \arrow{>}}, postaction={decorate}] (2.7*\x,2*\x) -- (2.7*\x,1.7*\x)
	to[start angle=-90,next angle=-160] (1.8*\x,1*\x) node[below right] {$ b_2 $}
	to[next angle=-90] (0.9*\x,0.3*\x) -- (0.9*\x,0);
	\draw[very thick, blue, dashed] (2.7*\x,2.3*\x) -- (2.7*\x,2*\x)
	(0.9*\x,0) -- (0.9*\x,-0.3*\x);
	\draw[very thick, decoration={markings, mark=at position 1 with \arrow{>}}, postaction={decorate}] (-1*\x,0) -- (4*\x,0) node[right] {$ \alpha_0 $};
	\draw[very thick, decoration={markings, mark=at position 1 with \arrow{>}}, postaction={decorate}] (-1*\x,2*\x) -- (4*\x,2*\x) node[right] {$ \alpha_1 $};
	\draw[thick, green] (1.5*\x,0.5*\x) -- (0.5*\x,1.5*\x);
	\draw[thick, green] (3*\x,0.5*\x) -- (2*\x,1.5*\x);
	\draw (1.5*\x,0.5*\x)  circle(0.07);
	\draw (3*\x,0.5*\x)  circle(0.07);
	\draw[fill=black] (0.5*\x,1.5*\x) circle(0.07);
	\draw[fill=black] (2*\x,1.5*\x) circle(0.07) node[above] {$ z $};
	\draw[fill=black] (-0.2*\x,0) circle(0.07) node[below left] {$ x_1 $};
	\draw[fill=black] (0.9*\x,0) circle(0.07) node[below right] {$ x_n $};
	\draw (2*\x, -0.5*\x) node[below] {$(X \Ym^{l-1} X, \Xm Y^l \Xm)$};
	\end{scope}
	\end{tikzpicture}
	\caption{Four cases of the pair $ (W_1^2, W_{n-1}^n) $, and a basepoint contained in $S$.}
	\label{fig:b_diff_word}
\end{figure}
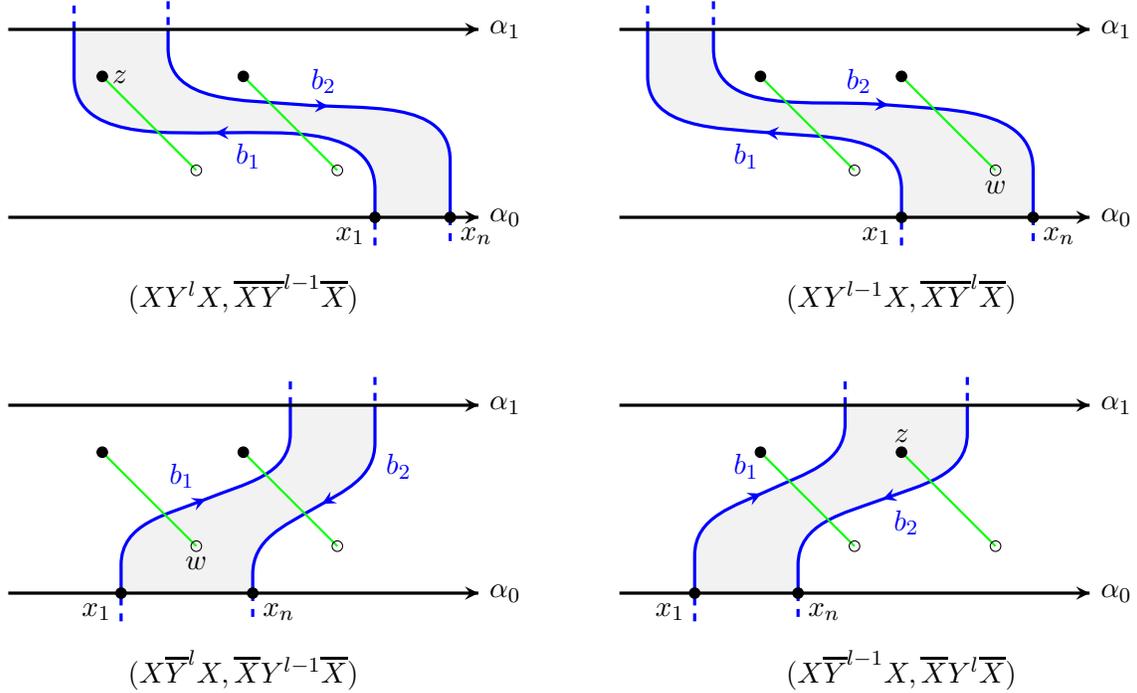
and we have the following:

\begin{cor}\label{cor:b_diff_word}
	If the two subarcs $ b_1 $ and $ b_2 $ correspond to the words $ XY^l X $ and $ \Xm \Ym^{l'} \Xm $ with $ |l - l'| = 1 $, then the pair $ (W_1^2, W_{n-1}^n) $ has four cases and the basepoint contained in $ S $ is shown in Table \ref{table:b_diff_word}. 
\begin{table}[htbp]
\centering
\setlength{\tabcolsep}{2pt} 
\renewcommand\arraystretch{1.5}
\begin{tabular}{|c|c|c|c|c|}
		\hline
		$ (W_1^2, W_{n-1}^n) $ & $ (XY^lX, \Xm \Ym^{l-1} \Xm) $ & $ (XY^{l-1}X, \Xm \Ym^l \Xm) $ & $ (X \Ym^l X, \Xm Y^{l-1} \Xm) $ & $ (X \Ym^{l-1} X, \Xm Y^l \Xm) $ \\
		\hline
		$ P(S) $ & $ (1, 0) $ & $ (0, 1) $ & $ (0, 1) $ & $ (1, 0) $ \\
		\hline
		\end{tabular}
		\vskip\baselineskip
		\caption{Four cases of the pair $ (W_1^2, W_{n-1}^n) $ when $ |l-l'|=1 $.}
		\label{table:b_diff_word}
\end{table}
\end{cor}

It is evident that $ n_z(S) = n_w(S) $ if $ l = l' $. We now consider the height one primitive bigons contained in $ D $. The simplest case is that there is exactly one such bigon, namely the one corresponding to $ W_2^{n-1} $.

\begin{lemma}\label{lemma:one_height_one}
Suppose the two subarcs $ b_1 $ and $ b_2 $ correspond to the words $ XY^lX $ and $ \Xm \Ym^l \Xm $. If $ W_2^{n-1} $ is a primitive disk word, then
$$P(S)=(0,0).$$
\end{lemma}

\begin{proof}
Without loss of generality, suppose that $ S $ contains a basepoint. Since $ n_z(S) = n_w(S)$, $S$ contains both $ z $ and $ w $ exactly once by Lemma \ref{lemma:point_smaller1}. As illustrated in Figure \ref{fig:1height1}, 
	\begin{figure}[htbp]
		\centering
		\begin{tikzpicture}
			\def\x{1.1}
			\draw[fill = gray!10] (2*\x,-\x) -- (2*\x,-1.3*\x)
			to[start angle=-90, next angle=-180] (1.7*\x,-1.7*\x) -- (1.6*\x,-1.7*\x)
			to[next angle=90](1.3*\x,-1.3*\x) --(1.3*\x,-\x);
			\draw[shift = {(1.5*\x, 3*\x)}, fill = gray!30] (2*\x,-\x) -- (2*\x,-1.3*\x)
			to[start angle=-90, next angle=-180] (1.7*\x,-1.7*\x) -- (1.6*\x,-1.7*\x)
			to[next angle=90](1.3*\x,-1.3*\x) --(1.3*\x,-\x);
			
			\draw[very thick, decoration={markings, mark=at position 1 with  {\arrow[scale=1.5]{>}}}, postaction={decorate}] (-1*\x,0) -- (5.5*\x,0) node[right] {$ \alpha_0 $};
			\draw[very thick, decoration={markings, mark=at position 1 with  {\arrow[scale=1.5]{>}}}, postaction={decorate}] (-1*\x,2*\x) -- (5.5*\x,2*\x) node[right] {$ \alpha_1 $};
			\draw[thick, decoration={markings, mark=at position 1 with {\arrow[scale=1.5]{>}}}, postaction={decorate},dashed] (-1*\x,-1*\x) -- (5.5*\x,-1*\x) node[right] {$ \alpha' $};
			\draw[ thick, decoration={markings, mark=at position 1 with {\arrow[scale=1.5]{>}}}, postaction={decorate},dashed] (-1*\x,3*\x) -- (5.5*\x,3*\x) node[right] {$ \Gamma(\alpha') $};
			\draw[thick, green] (1*\x,0.5*\x) -- (0*\x,1.5*\x);
			\draw[thick, green] (2.5*\x,0.5*\x) -- (1.5*\x,1.5*\x);
			\draw[thick, green] (4*\x,0.5*\x) -- (3*\x,1.5*\x);
			\draw (1*\x,0.5*\x)  circle(0.07);
			\draw (2.5*\x,0.5*\x)  circle(0.07);
			\draw (4*\x,0.5*\x)  circle(0.07);
			\draw[fill=black] (0*\x,1.5*\x) circle(0.07);
			\draw[fill=black] (1.5*\x,1.5*\x) circle(0.07);
			\draw[fill=black] (3*\x,1.5*\x) circle(0.07);
			\draw[very thick, blue, decoration={markings, mark=at position 0.5 with {\arrow[scale=1.5]{>}}}, postaction={decorate}] (0*\x,-0.2*\x) -- (0*\x,0.4*\x)
			to[start angle=90, next angle=20](1*\x,1.1*\x) node[above left] {$ b_1 $}
			to[next angle=90] (2*\x,1.8*\x) -- (2*\x,2*\x);
			\draw[very thick, blue, decoration={markings, mark=at position 0.5 with {\arrow[scale=1.5]{>}}}, postaction={decorate}, dashed] (2*\x,2*\x) -- (2*\x,3*\x)
			to[start angle=90, next angle=0] (2.9*\x,5*\x) -- (3.1*\x,5*\x)
			to[next angle=-90](4*\x,3*\x) node[above right] {$ \beta $} -- (4*\x, 2*\x);
			\draw[very thick, blue, decoration={markings, mark=at position 0.5 with {\arrow[scale=1.5]{>}}}, postaction={decorate}] (4*\x,2*\x) -- (4*\x,1.6*\x)
			to[start angle=-90, next angle=-160] (3*\x,0.9*\x) node[below right] {$ b_2 $}
			to[next angle=-90](2*\x,0.2*\x) -- (2*\x,0);
			\draw[thick, blue, decoration={markings, mark=at position 0.3 with \arrow{>}}, postaction={decorate}] (2*\x,0) -- (2*\x,-1.3*\x)
			to[start angle=-90, next angle=180](1.7*\x,-1.7*\x) -- (1.6*\x,-1.7*\x)
			to[next angle=90](1.3*\x,-1.3*\x) -- (1.3*\x,-0.9*\x);
			\draw[thick, blue, dashed] (1.3*\x,-0.9*\x) -- (1.3*\x,-0.8*\x)
			to [start angle=-90, next angle=180](1.2*\x,-0.5*\x)
			to[next angle=270](1.1*\x,-0.8*\x) -- (1.1*\x,-1.3*\x) -- (1.1*\x,-1.4*\x)
			to[next angle=0] (1.8*\x, -2.3*\x) -- (2.4*\x, -2.3*\x)
			to[next angle=90] (3.1*\x,-1.3*\x) -- (3.1*\x,0.3*\x);
			
			\draw[shift={(1.5*\x,3*\x)},thick, red] (2*\x,0.2*\x) -- (2*\x,-1.3*\x) 
			to[start angle=-90, next angle=180](1.7*\x,-1.7*\x) -- (1.6*\x,-1.7*\x)
			to[next angle=90] (1.3*\x,-1.3*\x) --(1.3*\x,-0.9*\x);
			
			\draw[shift = {(0, -3*\x)}, fill=black] (1.5*\x,1.5*\x) circle(0.07);
			\draw[shift = {(0, -3*\x)}, thick, green] (2*\x,\x) -- (1.5*\x,1.5*\x);
			
			\draw[fill=black] (0, 0) circle(0.07) node[below right] {$ x_1 $};
			\draw[fill=black] (2*\x, 0) circle(0.07) node[below right] {$ x_n $};
			\draw[fill=black] (2*\x, -\x) circle(0.07) node[below right] {\tiny $ x_N $};
			\draw[fill=black] (3.1*\x, 0) circle(0.07) node[below right] {$ x' $};
			\draw[red, fill=red] (3.5*\x, 3*\x) circle(0.07) node[above left] {\tiny $ \Gamma(x_n) $};
			
			\draw (3*\x, 4*\x) node {$ D_1 $};
			\draw (2.5*\x, -1.6*\x)  node {$ D' $};
			\draw (1.7*\x, -1.2*\x)  node {\tiny $ D'' $};
			\draw[red] (3.1*\x, 2.25*\x)  node {\tiny $ \Gamma(D'') $};
			\draw[black] (1.6*\x,\x)  node {$ S $};
		\end{tikzpicture}
		\caption{The case where the subword $ W_2^{n-1} $ is primitive. The red arc is a transformation of arc $ s $.}
		\label{fig:1height1}
	\end{figure}
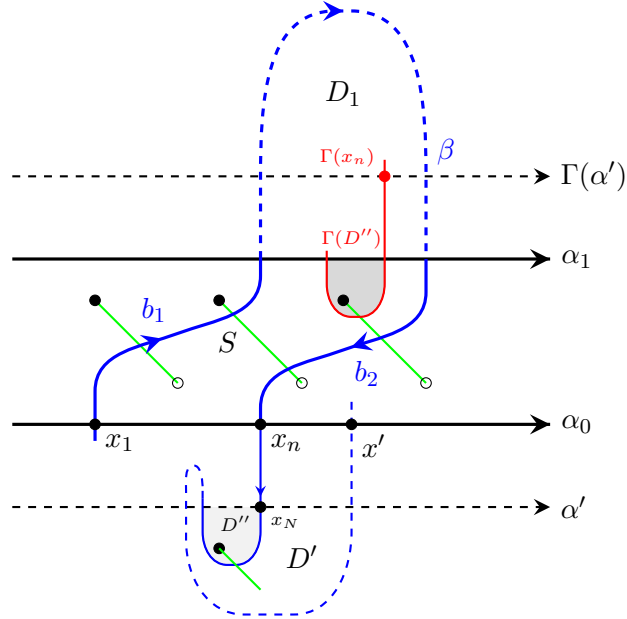
since the algebraic intersection number of $ \alpha_0 $ and $ \beta $ is one, there will be other intersection points. Without loss of generality, suppose $ x' \neq x_1 $ is adjacent to $ x_n $, then the arcs bounded by $ x_n $ and $ x' $ on $ \alpha_0 $ and $ \beta $ bound a primitive bigon $ D' $, and which is on the other side of $ \alpha_0 $ with $ D $, i.e., $ D' $ is downward. Therefore, it must contain a subword of the form $ \Xm Y^{\pm 1} X $, whose corresponding bigon contains the basepoint $ z $. In fact, the first elementary bigon after the letter $ \Xm_n $ is of this form, assumed to be $ \Xm_N Y^{\pm 1} X_{N+1}, (N \geq n) $, and denoted by $ D''$. By the assumption that $ S $ contains the basepoint $ z $, there is a covering transformation $ \Gamma $ such that $ \Gamma (D'') $ is contained in $ D $. Furthermore, let $ s $ be the subarc of $ \beta $ that from the point $ x_n $ to $ x_{N+1} $, then $ \Gamma(s) \subseteq S $. In particular, the covering transformation $ \Gamma $ maps $ x_n $ to a point in the interior of $ D $. Since $ \beta \cap \Gamma(\beta) = \emptyset $, this implies that $ \Gamma (D) $ is contained in $ D $. We get a contradiction since $ \Gamma $ is nothing but a translation.
\end{proof}

\begin{cor}
	$ P(X^k Y \Xm^k) = P(X Y \Xm) $. This equation holds for the other three types as well. We also call all of them are elementary bigons.
\end{cor}

\begin{cor}\label{cor:no_whole_arc} 
	A lift of a real bigon does not contain a whole lift of $ t_\alpha $.
\end{cor}
\begin{proof}
	Assume $ D $ is a lift of a real bigon that contains a whole lift $ \widetilde{t}_{\alpha} $ of $ t_{\alpha} $, then $ \widetilde{t}_{\alpha} $ is contained in a primitive bigon whether or not $ D $ is primitive. However, it is obvious from the proof of Lemma \ref{lemma:one_height_one} that this is impossible.
\end{proof}

Let $ x_{k_1}, \cdots, x_{k_{d+1}} $ ($ k_1 = 2, k_{d+1} = n-1 $) be all points of height one. Suppose $ D $ has more than one height one primitive bigons ($ d \geq 2 $), then the first and the last height one points $ x_2 $ and $ x_{n-1} $ are connected by a series of primitive bigons $ D_i $ corresponding to the subword $ W_{k_i}^{k_{i+1}}, \ i = 1, \cdots, d $. We call $ D_1 $ and $ D_d $ the first and the last height one primitive bigons in $ W_1^n $, respectively.

\begin{lemma}\label{lemma:b_same_word}
	Suppose the two subarcs $ b_1 $ and $ b_2 $ correspond to the words $ XY^lX $ and $ \Xm \Ym^l \Xm $. If $ D $ contains at least two primitive bigons of height one, say $ D_1, \cdots, D_d $, ($ d \geq 2 $). Then $ P(S) = (1, 1) $ if and only if the first and last height one primitive bigons $ D_1 $ and $ D_{d} $ have the same orientation as $ D $; otherwise, $ P(S) = (0, 0) $.
\end{lemma}
\begin{proof}
	Note that $ P(S) = (1, 1) $ or $ (0, 0) $. Suppose $ P(S) = (1, 1) $. As shown in Figure \ref{fig:b_same_word1},
	\begin{figure}[htbp]
		\centering
		\begin{tikzpicture}
			\def\x{1.4}
			\draw[very thick, decoration={markings, mark=at position 1 with {\arrow[scale=1.5]{>}}}, postaction={decorate}] (-1*\x,0) -- (5*\x,0) node[right] {$ \alpha_0 $};
			\draw[very thick, decoration={markings, mark=at position 1 with {\arrow[scale=1.5]{>}}}, postaction={decorate}] (-1*\x,2*\x) -- (5*\x,2*\x) node[right] {$ \alpha_1 $};
			\draw[thick, green] (1*\x,0.5*\x) -- (0*\x,1.5*\x);
			\draw[thick, green] (3*\x,0.5*\x) -- (2*\x,1.5*\x);
			\draw (1*\x,0.5*\x)  circle(0.07) node[below] {$ w $};
			\draw (3*\x,0.5*\x)  circle(0.07);
			\draw[fill=black] (0*\x,1.5*\x) circle(0.07);
			\draw[fill=black] (2*\x,1.5*\x) circle(0.07) node[below] {$ z $};
			\draw[very thick, blue, decoration={markings, mark=at position 0.5 with {\arrow[scale=1.5]{>}}}, postaction={decorate}] (0.5*\x,-0.2*\x)
			to[start angle=90, next angle=90] (0.5*\x,0.3*\x)
			to[next angle=50](0.9*\x,1*\x) node[above left] {$ b_1 $}
			to[next angle=90](1.3*\x,1.7*\x) 
			to[next angle=90](1.3*\x,2.3*\x);
			\draw[fill=black] (0.5*\x, 0)  circle(0.07) node[below left] {$ x_1 $};
			\draw[fill=black] (1.3*\x, 2*\x)  circle(0.07) node[below right] {\small $ x_2 $};
			\draw[very thick, blue, decoration={markings, mark=at position 0.5 with {\arrow[scale=1.5]{>}}}, postaction={decorate}] (2.4*\x,2.3*\x)
			to[start angle=-90, next angle=-90] (2.4*\x,1.7*\x)
			to[next angle=-130](2*\x,1*\x) node[below right] {$ b_2 $}
			to[next angle=-90](1.6*\x,0.3*\x)
			to[next angle=-90](1.6*\x,-0.2*\x);
			\draw[fill=black] (1.6*\x, 0)  circle(0.07) node[below right] {$ x_n $};
			\draw[fill=black] (2.4*\x, 2*\x)  circle(0.07) node[below left] {\small $ x_{n-1} $};
			\draw (1.3*\x, 0.8*\x)  node {$ S $};
			\draw (0.2*\x, 0.8*\x)  node {$ S' $};
			
			\draw[shift ={(2*\x, 0)}, thick, red, decoration={markings, mark=at position 0.5 with {\arrow[scale=1.5]{>}}}, postaction={decorate}] (0.5*\x,-0.2*\x)
			to[start angle=90, next angle=90] (0.5*\x,0.3*\x)
			to[next angle=50](0.9*\x,1*\x) node[right] {\footnotesize $ \Gamma_1 (b_1) $}
			to[next angle=90](1.3*\x,1.7*\x) 
			to[next angle=90](1.3*\x,2.3*\x);
			\draw[red,fill=red] (3.3*\x, 2*\x)  circle(0.07) node[below right] {\footnotesize $ \Gamma_1(x_2) $};
			\draw[shift ={(-2*\x, 0)}, thick, magenta, decoration={markings, mark=at position 0.5 with {\arrow[scale=1.5]{>}}}, postaction={decorate}] (2.4*\x,2.3*\x)
			to[start angle=-90, next angle=-90] (2.4*\x,1.7*\x)
			to[next angle=-130](2*\x,1*\x) node[left] {\footnotesize $ \Gamma^{-1}_1(b_2) $}
			to[next angle=-90](1.6*\x,0.3*\x)
			to[next angle=-90](1.6*\x,-0.2*\x);
			\draw[magenta,fill=magenta] (0.4*\x, 2*\x)  circle(0.07) node[below left] {\footnotesize $ \Gamma^{-1}_1(x_{n-1}) $};
			
			\draw[shift ={(0.3*\x, 2*\x)}, thick, red, decoration={markings, mark=at position 0.5 with {\arrow[scale=1.5]{>}}}, postaction={decorate}] (0.5*\x,-0.2*\x)
			to[start angle=90, next angle=90] (0.5*\x,0.3*\x)
			to[next angle=50](0.9*\x,1*\x) node[above left] {\footnotesize $ \Gamma_2(b_1) $}
			to[next angle=90](1.3*\x,1.7*\x);
			\draw[red,fill=red] (0.8*\x, 2*\x)  circle(0.07) node[above left] {\small $ x'_1 $};
			\draw[shift ={(1.3*\x, 2*\x)}, thick, magenta, decoration={markings, mark=at position 0.5 with {\arrow[scale=1.5]{>}}}, postaction={decorate}] (2.4*\x,1.7*\x) to[start angle=-90,next angle=-130](2*\x,1*\x) node[below right] {\footnotesize $ \Gamma_3(b_2) $}
			to[next angle=-90](1.6*\x,0.3*\x)
			to[next angle=-90](1.6*\x,-0.2*\x);
			\draw[magenta,fill=magenta] (2.9*\x, 2*\x)  circle(0.07) node[above right] {\small $ x'_n $};
			
			\draw[thick, blue, decoration={markings, mark=at position 1 with \arrow{>}}, postaction={decorate}, dashed] (1.3*\x,2.3*\x)
			to[start angle=90, next angle=180](1.15*\x,2.6*\x)
			to[next angle=270](1*\x,2.3*\x)
			to[next angle=270](1*\x,1.7*\x);
			\draw[blue, fill=blue] (\x, 2*\x) circle(0.07);
			\draw[thick, blue, decoration={markings, mark=at position 1 with \arrow{>}}, postaction={decorate}, dashed] (1.3*\x,2.3*\x)
			to[start angle=90, next angle=90] (1.3*\x,2.4*\x)
			to[next angle=50](1.7*\x,3*\x)
			to[next angle=90](2.1*\x,3.7*\x) 
			to[next angle=90](2.1*\x,4*\x)
			to[next angle=180](1.3*\x,5*\x)
			to[next angle=270](0.5*\x,4*\x);
			
			\draw[shift= {(1.4*\x, 0)}, thick, blue, decoration={markings, mark=at position 0.1 with \arrow{>}}, postaction={decorate}, dashed] (1.3*\x,1.7*\x)
			to[start angle=90, next angle=90](1.3*\x,2.3*\x)
			to[next angle=180](1.15*\x,2.6*\x)
			to[next angle=270] (1*\x,2.3*\x);
			\draw[shift= {(1.4*\x, 0)}, blue, fill=blue] (1.3*\x, 2*\x) circle(0.07);
			\draw[thick, blue, decoration={markings, mark=at position 0.1 with \arrow{>}}, postaction={decorate}, dashed] (4.8*\x,4*\x)
			to[start angle=90, next angle=180] (4*\x,5*\x)
			to[next angle=-90](3.2*\x,4*\x)
			to[next angle=-130](2.8*\x,3.1*\x)
			to[next angle=-90](2.4*\x,2.4*\x)
			to[next angle=-90](2.4*\x,2.3*\x);
		\end{tikzpicture}
		\caption{The case where the subword $ W_2^{n-1} $ is not primitive and $ S $ contains basepoints. The two red arcs are translations of the arc $ b_1 $, and magenta arcs are translations of $ b_2 $. The two dashed arcs on the left are two choices of $ D_1 $ with opposite orientation as $ D $; the two dashed arcs on the right are two choices of $ D_d $ with opposite orientation as $ D $, all of them cannot occur.}
		\label{fig:b_same_word1}
	\end{figure}
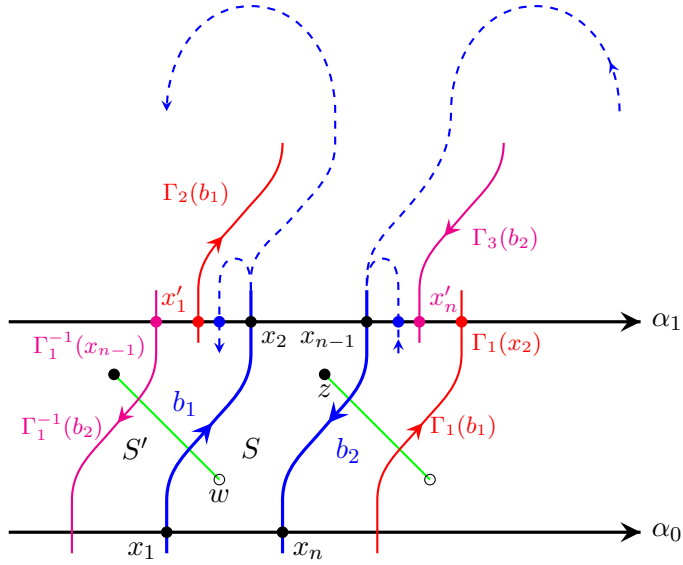
	consider the square domain $ S' $ bounded by the two lifts of the curve $ \alpha $, the arcs $ b_1 $ and $ \Gamma_{1}^{-1} (b_2) $, where $ \Gamma_{1} $ is a horizontal translation such that $ S' $ contains no basepoints. Let $ x'_1 $ be a copy of $ x_1 $ at the height one $ \alpha $-lifting and at the point on the left closest to $ x_2 $. There is a unique transformation that maps $ x_1 $ to $ x'_1 $, say $ \Gamma_2 $. One can find that $ x'_1 $ is on the right side of $ \Gamma^{-1}_1 (x_{n-1}) $, i.e.,
	$$ \Gamma^{-1}_1(x_{n-1}) < x'_1 < x_2. $$
	Because otherwise, the bigons $ \Gamma_2(D) $ and $ \Gamma^{-1}_1 (D) $ would overlap, which cannot happen.

	For the first height one primitive bigon $ D_1 $, consider the position of its other endpoint $ x_{k_2} $. Since any two different lifts of the curve $ \beta $ are disjoint, if $ x_{k_2} < x'_1 $, it forces that $ \Gamma_2 (D) $ contained in $ D_1 $, which cannot occur; on the other hand, $ x_{k_2} $ cannot be a point between $ x'_1 $ and $ x_2 $ since there is no basepoint contained in domain $ S' $. Thus $ x_{k_2} $ must be on the right side of $ x_2 $, i.e., $ D_1 $ has same orientation as $ D $. For the last height one primitive bigon $ D_d $, consider the copy of $ x_n $ at height one $ \alpha $-lifting and right closest to $ x_2 $, say $ x'_n $, there is a transformation $ \Gamma_3 $ that maps $ x_n $ to $ x'_n $. By the same argument, one can find that $ x_{k_{d-1}} $ be to the left side of $ x_{k_d} = x_{n-1} $, so that $ D_d $ has the same orientation as $ D $.

	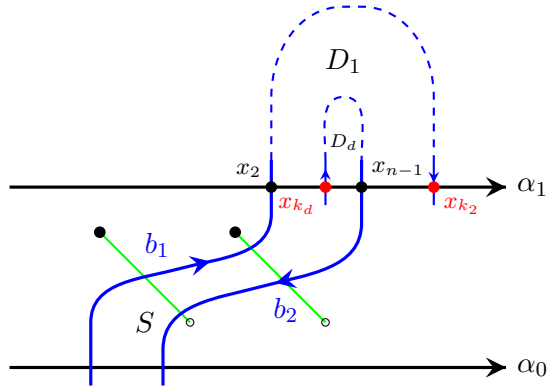
\begin{figure}[htbp]
		\centering
		\begin{tikzpicture}
			\def\x{1.2}
			\draw[very thick, decoration={markings, mark=at position 1 with {\arrow[scale=1.5]{>}}}, postaction={decorate}] (-0.5*\x,0) -- (5*\x,0) node[right] {$ \alpha_0 $};
			\draw[very thick, decoration={markings, mark=at position 1 with {\arrow[scale=1.5]{>}}}, postaction={decorate}] (-0.5*\x,2*\x) -- (5*\x,2*\x) node[right] {$ \alpha_1 $};
			\draw[thick, green] (1.5*\x,0.5*\x) -- (0.5*\x,1.5*\x);
			\draw[thick, green] (3*\x,0.5*\x) -- (2*\x,1.5*\x);
			\draw (1.5*\x,0.5*\x)  circle(0.05);
			\draw (3*\x,0.5*\x)  circle(0.05);
			\draw[fill=black] (0.5*\x,1.5*\x) circle(0.07);
			\draw[fill=black] (2*\x,1.5*\x) circle(0.07);
			\draw[very thick, blue, decoration={markings, mark=at position 0.6 with {\arrow[scale=1.5]{>}}}, postaction={decorate}] (0.4*\x,-0.2*\x) -- (0.4*\x,0.5*\x)
			to[start angle=90,next angle=15] (1.4*\x,1.1*\x) node[above left] {$ b_1 $}
			to[next angle=90] (2.4*\x,1.7*\x) -- (2.4*\x,2.3*\x);
			\draw[thick, blue, dashed] (2.4*\x,2.3*\x) -- (2.4*\x,3*\x)
			to[start angle=90, next angle=0] (3.2*\x,4*\x) -- (3.4*\x,4*\x) 
			to[next angle=-90] (4.2*\x,3*\x) -- (4.2*\x,2.3*\x);
			\draw[thick, blue, decoration={markings, mark=at position 0.5 with \arrow{>}}, postaction={decorate}] (4.2*\x,2.3*\x) -- (4.2*\x, 1.8*\x);
			\draw[thick, blue, decoration={markings, mark=at position 0.8 with \arrow{>}}, postaction={decorate}] (3*\x,1.8*\x) -- (3*\x, 2.3*\x);
			\draw[thick, blue, dashed] (3*\x,2.3*\x)
			to[start angle=90, next angle=0] (3.2*\x,3*\x) 
			to[next angle=-90](3.4*\x,2.3*\x);
			\draw[very thick, blue, decoration={markings, mark=at position 0.5 with {\arrow[scale=1.5]{>}}}, postaction={decorate}] (3.4*\x,2.3*\x) -- (3.4*\x,1.6*\x)
			to[start angle=-90,next angle=-165] (2.3*\x,0.9*\x) node[below right] {$ b_2 $}
			to[next angle=-90] (1.2*\x,0.2*\x) -- (1.2*\x,-0.2*\x);
			
			\draw[fill=black] (2.4*\x,2*\x) circle(0.07) node[above left] {\footnotesize $ x_2 $};
			\draw[fill=black] (3.4*\x,2*\x) circle(0.07) node[above right] {\footnotesize $ x_{n-1} $};
			\draw[color = red, fill=red] (4.2*\x,2*\x) circle(0.07) node[below right] {\footnotesize $ x_{k_2} $};
			\draw[color = red, fill=red] (3*\x,2*\x) circle(0.07) node[below left] {\footnotesize $ x_{k_d} $};
			\draw  (3.2*\x, 3.4*\x) node {$ D_1 $};
			\draw  (3.2*\x, 2.5*\x) node {\tiny $ D_d $};
			\draw  (1*\x, 0.5*\x) node {$ S $};
		\end{tikzpicture}
		\caption{The case where $ D_1 $ and $ D_d $ have the same orientation, but $ S $ does not contain the basepoint. This figure shows that $ x_{k_d} \in S $.}
		\label{fig:b_same_word2}
	\end{figure}

	While if $ P(S) = (0, 0) $, i.e., $ S $ does not contain any basepoints. Suppose $ D_1 $ and $ D_{d} $ have the same orientation as $ D $, then one of $ x_{k_2} $ and $ x_{k_d} $ lies in $ S $. This is impossible. See Figure \ref{fig:b_same_word2}.
\end{proof}

\begin{lemma}\label{lemma:sum_points}
	$ P(D) = P(S) + \sum\limits_{i=1}^{d} P(D_i) $.
\end{lemma}
\begin{proof}
	Let $ a_1 $ be the subarc of the lift $ \alpha_1 $ that connects the points $ x_2 $ and $ x_{n-1} $.
If $ a_1 \subset D $. As illustrated in Figure \ref{fig:sum_points}, 
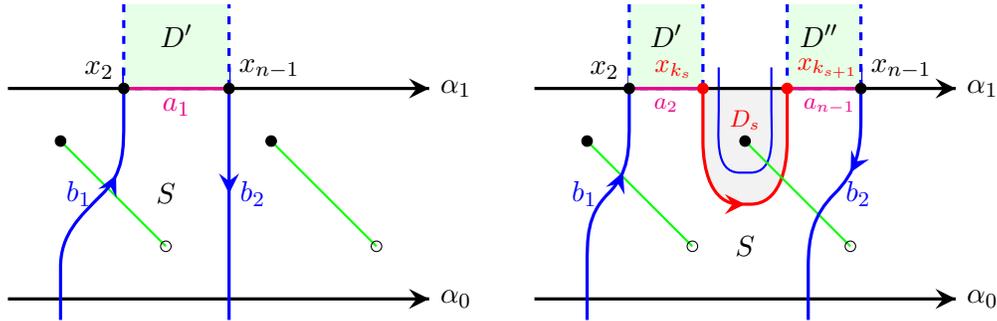
\begin{figure}[htbp]
		\centering
		\begin{tikzpicture}
		\def\x{1.4}
		\begin{scope}[shift = {(0, 0)}]
		\draw[very thick, decoration={markings, mark=at position 1 with {\arrow[scale=1.5]{>}}}, postaction={decorate}] (-0.5*\x,0) -- (3.5*\x,0) node[right] {$ \alpha_0 $};
		\draw[very thick, decoration={markings, mark=at position 1 with {\arrow[scale=1.5]{>}}}, postaction={decorate}] (-0.5*\x,2*\x) -- (3.5*\x,2*\x) node[right] {$ \alpha_1 $};
		\draw[thick, green] (1*\x,0.5*\x) -- (0*\x,1.5*\x);
		\draw[thick, green] (3*\x,0.5*\x) -- (2*\x,1.5*\x);
		\draw (1*\x,0.5*\x)  circle(0.07);
		\draw (3*\x,0.5*\x)  circle(0.07);
		\draw[fill=black] (0*\x,1.5*\x) circle(0.07);
		\draw[fill=black] (2*\x,1.5*\x) circle(0.07);
		\draw[very thick, blue, decoration={markings, mark=at position 0.6 with {\arrow[scale=1.5]{>}}}, postaction={decorate}] (0*\x,-0.2*\x) -- (0*\x,0.3*\x)
		to[start angle=90, next angle=45](0.4*\x,1*\x) node[left] {$ b_1 $}
		to[next angle=90](0.6*\x,1.7*\x) -- (0.6*\x,2.2*\x);
		\draw[very thick, blue, decoration={markings, mark=at position 0.5 with {\arrow[scale=1.5]{>}}}, postaction={decorate}] (1.6*\x,2.2*\x) -- (1.6*\x,\x) node[right] {$ b_2 $} -- (1.6*\x,-0.2*\x);
		\draw[very thick, blue, fill = green!10, dashed] (0.6*\x,2.8*\x) -- (0.6*\x,2*\x) -- (1.6*\x,2*\x) -- (1.6*\x,2.8*\x);
		\draw[very thick, color = magenta] (0.6*\x,2*\x) -- (1.6*\x, 2*\x);
		\draw[magenta] (1.1*\x, 2*\x) node[below] {$ a_1 $};
		\draw[fill=black] (0.6*\x,2*\x) node[above left] {$ x_2 $} circle(0.07);
		\draw[fill=black] (1.6*\x,2*\x) node[above right] {$ x_{n-1} $} circle(0.07);
		\draw (1.1*\x, 2.5*\x) node {$ D' $};
		\draw (1*\x, 1*\x) node {$ S $};
		\end{scope}
	
		\begin{scope}[shift = {(5*\x, 0)}]
		\draw[very thick, red, fill = gray!10, decoration={markings, mark=at position 0.5 with {\arrow[scale=1.5]{>}}}, postaction={decorate}] (1.1*\x,2*\x) -- (1.1*\x,1.5*\x)
		to[start angle=-90, next angle=0](1.45*\x,0.9*\x) -- (1.55*\x,0.9*\x)
		to[next angle=90](1.9*\x,1.5*\x) -- (1.9*\x,2*\x);
		\draw[red] (1.5*\x, 1.7*\x) node {\footnotesize $ D_s $};
		\draw[thick, blue] (1.25*\x,2.2*\x) -- (1.25*\x,1.8*\x)
		to[start angle=-90, next angle=0](1.45*\x,1.2*\x) -- (1.55*\x,1.2*\x)
		to[next angle=90](1.75*\x,1.8*\x) -- (1.75*\x,2.2*\x);
		
		\draw[very thick, decoration={markings, mark=at position 1 with {\arrow[scale=1.5]{>}}}, postaction={decorate}] (-0.5*\x,0) -- (3.5*\x,0) node[right] {$ \alpha_0 $};
		\draw[very thick, decoration={markings, mark=at position 1 with {\arrow[scale=1.5]{>}}}, postaction={decorate}] (-0.5*\x,2*\x) -- (3.5*\x,2*\x) node[right] {$ \alpha_1 $};
		\draw[thick, green] (1*\x,0.5*\x) -- (0*\x,1.5*\x);
		\draw[thick, green] (2.5*\x,0.5*\x) -- (1.5*\x,1.5*\x);
		\draw (1*\x,0.5*\x)  circle(0.07);
		\draw (2.5*\x,0.5*\x)  circle(0.07);
		\draw[fill=black] (0*\x,1.5*\x) circle(0.07);
		\draw[ fill=black] (1.5*\x,1.5*\x) circle(0.07);
		\draw[very thick, blue, decoration={markings, mark=at position 0.6 with {\arrow[scale=1.5]{>}}}, postaction={decorate}] (0*\x,-0.2*\x) -- (0*\x,0.3*\x)
		to[start angle=90,next angle=45](0.2*\x,1*\x) node[left] {$ b_1 $}
		to[next angle=90](0.4*\x,1.7*\x) -- (0.4*\x,2.2*\x);
		\draw[very thick, blue, decoration={markings, mark=at position 0.4 with {\arrow[scale=1.5]{>}}}, postaction={decorate}] (2.6*\x,2.2*\x) -- (2.6*\x,1.7*\x)
		to[start angle=-90, next angle=-135](2.35*\x,1*\x) node[right] {$ b_2 $}
		to[next angle=-90](2.1*\x,0.3*\x) -- (2.1*\x,-0.2*\x);
		\draw[very thick, blue, fill = green!10, dashed] (0.4*\x,2.8*\x) -- (0.4*\x,2*\x) -- (1.1*\x,2*\x) -- (1.1*\x, 2.8*\x);
		\draw[very thick, blue, fill = green!10, dashed] (1.9*\x,2.8*\x) -- (1.9*\x,2*\x) -- (2.6*\x,2*\x) -- (2.6*\x, 2.8*\x);
		\draw[very thick, color = magenta] (0.4*\x, 2*\x) -- (1.1*\x, 2*\x);
		\draw[very thick, color = magenta] (1.9*\x, 2*\x) -- (2.6*\x, 2*\x);

		\draw[fill=black] (0.4*\x,2*\x) circle(0.07) node[above left] {$ x_2 $};
		\draw[red, fill=red] (1.1*\x,2*\x) circle(0.07) node[above left] {\small $ x_{k_s} $};
		\draw[red, fill=red] (1.9*\x,2*\x) circle(0.07) node[above right] {\small $ x_{k_{s+1}} $};
		\draw[fill=black] (2.6*\x,2*\x) circle(0.07) node[above right] {$ x_{n-1} $};
		\draw[magenta] (0.75*\x, 2*\x) node[below] {\footnotesize $ a_2 $};
		\draw[magenta] (2.3*\x, 2*\x) node[below] {\footnotesize $ a_{n-1} $};

		\draw (0.75*\x, 2.5*\x) node {$ D' $};
		\draw (2.2*\x, 2.5*\x) node {$ D'' $};
		\draw (1.5*\x, 0.5*\x) node {$ S $};
		\end{scope}
		\end{tikzpicture}
		\caption{Cutting the primitive bigon $ D $.}
		\label{fig:sum_points}
	\end{figure}
cut $ D $ along $ a_1 $ into two domains: a square domain $ S $ and a real bigon (not necessarily primitive) connecting the points $ x_2 $ and $ x_{n-1} $, say $ D' $. So that $ D = S + D' $. Moreover, $ D' $ can be divided into a combination of primitive bigons $ D_1, \cdots, D_d $. Thus
	$$ P(D) = P(S) + P(D') = P(S) + \sum\limits_{i=1}^{d} P(D_i). $$

	If $ D $ does not contain the entire subarc $ a_1 $, there must be a height one primitive bigon of the form $ \Xm Y^{\pm 1} X $ contained in $ S $. Consider the outer-most one, and suppose it is the $ s $-th height one primitive bigon $ D_s $, whose two endpoints denoted by $ x_{k_s} $ and $ x_{k_{s+1}} $. Then $ D_s $ has the opposite orientation as $ S $ so that its corresponding word is $ \Xm_{k_s} \Ym X_{k_{s+1}} $. Let $ a_2 $ be the subarc of $ \alpha_1 $ that bounded by the points $ x_2 $ and $ x_{k_s} $, and $ a_{n-1} $ be the subarc that bounded by the points $ x_{k_{s+1}} $ and $ x_{n-1} $. Cutting the primitive bigon $ D $ along the line segments $ a_2 $ and $ a_{n-1} $, we obtain three domains: two of them are real bigons (not necessarily primitive) whose corresponding words are $ W_2^{k_s} $ and $ W_{k_{s+1}}^{n-1} $, denoted by $ D' $ and $ D'' $ respectively, and the last one is a hexagon contained in $ S $ that does not contain the basepoint $ z $, denoted by $ S_w $, see Figure \ref{fig:sum_points}. So that
	$$ D = D' + D'' + S_w, \qquad S = S_w - D_s. $$
	On the other hand, the two points $ x_2 $ and $ x_{k_s} $ are connected by a series of primitive bigons $ D_1, \cdots, D_{s-1} $, and points $ x_{k_{s+1}} $ and $ x_{k_{d+1}} $ are connected by primitive bigons $ D_{s+1}, \cdots, D_{d} $. Thus we obtain
	$$ \begin{aligned}
		P(D) & = P(D') + P(D'') + P(S_w) \\
		& = \sum_{i=1}^{s-1} P(D_i) + \sum_{i=s+1}^{d} P(D_i) + P(S_w) \\
		& = \sum_{i=1}^{d} P(D_i) - P(D_s) + P(S_w) \\
		& = \sum_{i=1}^{d} P(D_i) + P(S).
	\end{aligned} $$
\end{proof}

We have discussed in details on upward, positive primitive bigons. The result for the other three types of primitive bigons are similar. In summary, we have the following.

\begin{thm}\label{thm:PD_from_disk_word}
Let $ W_1^n $ be a primitive disk word and $D$ be its corresponding bigon in $ \mathbb{C} $, then the number of basepoints in $ D $ can be computed from the word $W_1^n$.
\end{thm}

\begin{proof} We prove by induction on the length of the word $W_1^n$. The conclusion holds for the four elementary disk words from Figure \ref{fig:elementary_bigons}.

For the case of length $ n $, each height one primitive bigon $ D_i $ has length less than $ n $, so that all of $ P(D_i) $ can be read out from the subword of $ W_1^n $ by induction. On the other hand, Lemmas \ref{cor:b_diff_word} and \ref{lemma:b_same_word} imply that $P(S)$ can be computed from the word $ W_1^n $. Theorem \ref{thm:PD_from_disk_word} then follows by Lemma \ref{lemma:sum_points}.
\end{proof}

\clearpage
\section{The fundamental group and $\widehat{HFK}$}\label{sec:alg}

Now we describe the algorithm for computing $ \widehat{HFK}(S^3, K) $ from a presentation of $\pi_1(S^3\setminus K)$ from a $(1,1)$ Heegaard diagram.

\begin{algo}\label{algo:main}
Let $K$ be a $(1,1)$ knot and $ \langle X, Y \,|\, R(X, Y) \rangle $ be a presentation of $\pi_1(S^3\setminus K)$ from a $(1,1)$ Heegaard diagram. The chain complex $\widehat{CFK}(S^3,K)$ has generators corresponding to the letters $ X $ and $ \Xm $ in the relator $ R $ and trivial differential. The Alexander and Maslov gradings are determined as follows:
\renewcommand{\theenumi}{\Roman{enumi}}
\begin{enumerate}
\item \label{step_bigon} Enumerate all the primitive disk words from the relator $ R $ and determine their orientation by Table \ref{table:elementary_bigons} and Lemma \ref{lemma:orientation}.
\item \label{step_basepoint} For each primitive disk word $W_1^n $, $P(W_1^n) = \left(n_z(W_1^n), n_w(W_1^n)\right) $ is computed iteratively as follows:
\begin{enumerate}[ref=\theenumi(\theenumii)]

\item\label{step_elementary} $n_z$ and $n_w$ for the elementary disk words are given in Table \ref{table:elementary_bigons};
		
\item\label{step_basepoint_key} If $ W_1^n $ is upward and positive with $n>2$. Let $ x_{k_1}, \cdots, x_{k_{d+1}} $ ($ k_1 = 2, k_{d+1} = n-1 $) be all points of height one. Cut $ W_1^n $ into primitive disk words $ W_{k_i}^{k_{i+1}}, \ i = 1, \cdots, d $, and two subwords $ W_1^2 = X_1 Y^l X_2 $ and $ W_{n-1}^n = \Xm_{n-1} \Ym^{l'} \Xm_n $ ($ l,l'\in\bbz $, and $ |l - l'| \leq 1 $ by Lemma \ref{lemma:point_smaller1}). Then by Lemma \ref{lemma:sum_points}
		$$ P(W_1^n) = P(S) + \sum_{i=1}^{d} P(W_{k_i}^{k_{i+1}}), $$
where $ P(S) $ is determined by the following: 
		\begin{enumerate}
			\item If $ |l - l'| = 1 $, $ P(S) $ are shown in Table \ref{table:b_diff_word} (Corollary \ref{cor:b_diff_word}).
			\item If $ l = l' $, then by Lemmas \ref{lemma:one_height_one} and \ref{lemma:b_same_word}
			$$ P(S) = \left\{ \begin{aligned}
				(1,1), & \quad  \text{$ d \geq 2 $ and $ W_{2}^{k_{2}} $ and $ W_{k_d}^{n-1} $ are positive}; \\
				(0, 0), & \quad \text{otherwise}.
			\end{aligned}\right. $$
		\end{enumerate}
		
		\item If $ W_1^n $ is upward and negative. A upward positive disk word $ \widetilde{W}_1^n $ can be obtained from $ W_1^n $ by replacing each $ Y $-letter by its inverse, then
		$$ P(W_1^n) = - P(\widetilde{W}_1^n). $$
		
		\item If $ W_1^n $ is downward and negative. A upward positive disk word $ \widetilde{W}_1^n $ can be obtained from $ W_1^n $ by replacing each $ X $-letter by its inverse, then
		$$ (n_z(W_1^n), n_w(W_1^n)) = - (n_w(\widetilde{W}_1^n), n_z(\widetilde{W}_1^n)). $$
		
		\item If $ W_1^n $ is downward and positive. A upward positive disk word $ \widetilde{W}_1^n $ can be obtained from $ W_1^n $ by replacing each $ X $- and $ Y $-letter by its inverse, then
		$$ (n_z(W_1^n), n_w(W_1^n)) = (n_w(\widetilde{W}_1^n), n_z(\widetilde{W}_1^n)). $$
	\end{enumerate}
	
	\item \label{step_relative} For any two points $ x_1 $ and $ x_n $ that connected by a primitive disk word $ W_1^n $, the relative Alexander is determined by the equation:
	\begin{equation}\label{eqn:alexander_relative}
		F(x_1) - F(x_n) = n_z(W_1^n) - n_w(W_1^n),
	\end{equation}
	and the relative Maslov grading can be computed as:
	\begin{equation}\label{eqn:maslov_relative}
		M(x_1) - M(x_n) = \left\{\begin{aligned}
			& 1 - 2 n_w(W_1^n), &  & \text{if $ W_1^n $ is positive}; \\
			& -1 - 2 n_w(W_1^n), &  & \text{if $ W_1^n $ is negative}.
		\end{aligned} \right.
	\end{equation}
	
	\item \label{step_absolute} The absolute Alexander grading can be obtained by requiring
	\begin{equation}\label{eqn:alexander_grading_absolute}
		\# \{x \,|\, F(x) = i\} \equiv \#\{x \,|\, F(x) = -i \} \pmod{2}, \quad \forall \, i \in \mathbb{Z}. 
	\end{equation}
	To obtain the absolute Maslov grading, we forget the basepoint $ z $, and to find a unique intersection point which is the generator of $ \widehat{HF}(T^2, \alpha, \beta, w) \cong \mathbb{Z} $. This process can only be done on the relator $ R(X, Y) $, as follows:
	\begin{enumerate}
		\item Erasing primitive bigons whose $ n_w = 0 $. (For example, all elementary bigons $ \Xm^k Y^{\pm 1} X^k $ are removed.) The resulting relator is denoted by $ R'(X, Y) $.
		\item The basepoint $ w $ also gives a relative grading by the equation
		$$ w(x_1) - w(x_n) = n_w(W_1^n). $$
		For the relator $ R'(X, Y) $, delete those new primitive bigons whose relative $ w $-grading of two endpoints is zero. (Here the relative grading is given by $ R $, not $ R' $.) The resulting relator is denoted by $ R''(X, Y) $.
		\item There will be exactly one $ X $-letter in $ R''(X, Y) $, assuming it is the $ k $-th one. Thus we can define $ M(x_k) = 0 $ to obtain an absolute $ M $-grading of all points.
	\end{enumerate}
\end{enumerate}
\end{algo}

In Step \ref{step_absolute}, all primitive bigons deleted in the above two procedures do not contain the basepoint $ w $, so it can be realized by isotopies of the $ \beta $ curve in the complement of $ w $, and the resulting diagram $ (T^2, \alpha, \beta'', w) $ specifies $ S^3 $, where $ \beta'' $ is the curve obtained from $ \beta $ by isotopies. We need to show that there is exactly one $ X $-letter in $ R''(X, Y) $. Specifically, we explain that no primitive bigon remains after these two steps. The proof is the same as the proof of Lemma \ref{lemma:one_height_one}. Suppose $ R'' $ contains a primitive bigon $ D_1 $ with $ n_w(D_1) \neq 0 $, and $ x_1 $ and $ x_n $ are two endpoints of $ D_1 $. Assume it is upward without loss of generality. Since the algebraic intersection number $ [\alpha] \cdot [\beta''] = [\alpha] \cdot [\beta] = +1 $, there will be another intersection point $ x' $. Suppose $ x' $ is adjacent to $ x_n $, and they form a primitive bigon $ D_2 $. Since $ D_2 $ remains after the second step, the $ w $-grading implies that $ n_w(D_2) \neq 0 $, i.e., $ D_2 $ also contains the basepoint $ w $. On the other hand, $ D_2 $ is downward. Therefore, there exists a translation $ \Gamma $ such that $ \Gamma(D_2) $ and $ D_1 $ are overlapped, i.e., $ \Gamma(\beta'') \cap \beta'' \neq \emptyset $, which cannot happen.

To summarize, all the steps described above are relevant only to $ R $. So we have the following theorem:

\begin{thm}\label{thm:algo}
	Let $ K $ be a $ (1, 1) $ knot in $ S^3 $, $ \langle X, Y \,|\, R(X, Y) \rangle $ be a cyclically reduced presentation of $ \pi_1(S^3 \setminus K) $ that is coming from a $(1,1)$ Heegaard diagram. The generators of the knot Floer homology of $ K $ are one-to-one correspondence to the $ X $-letters in the relator $ R $, denoted $ x_i $ corresponding to the $ i $-th $ X $-letter; the Alexander and Maslov gradings of each generator can be determined by the above algorithm which involves only the relator $ R(X, Y) $. The knot Floer homology can be expressed as follows:
	$$ \widehat{HFK}_m(S^3, K;s) = \bigoplus_{\{ x_i \,|\, F(x_i) = s, M(x_i) = m\}} \mathbb{Z} \cdot \langle x_i \rangle. $$
\end{thm}

Rewrite the knot Floer homology as the Poincar\'e polynomial
$$ P_R(t, q) = \sum_{s, m \in \bbz} {\rm rank} \widehat{HFK}_m(S^3, K;s) \cdot t^s q^m = \sum_{i} t^{F(x_i)} q^{M(x_i)}. $$


\begin{prop}\label{prop:replace_t_alpha}
Let $ \langle X, Y \,|\, R \rangle $ be a presentation coming from a $ (1,1) $ Heegaard diagram, and $ P_R(t,m) $ be the Poincar\'e polynomial given by the algorithm.
\begin{enumerate}
		\item Let $ R' $ be the relator obtained from $ R $ by doing the transformation $ l_k: X \mapsto Y^kX, Y \mapsto Y $ (resp. $r_k: X \mapsto XY^k, Y \mapsto Y $) for an integer $ k $, and then reduced. Then $ \langle X, Y \,|\, R' \rangle $ is also a presentation coming from a $ (1,1) $ Heegaard diagram, and $ P_{R'}(t, q) = P_R(t, q) $.
		\item Let $ R'' $ be the relator obtained by transforming $ \tau: X \mapsto X, Y \mapsto \Ym $. Then $ \langle X, Y \,|\, R' \rangle $ is also a presentation coming from a $ (1,1) $ Heegaard diagram, and $ P_{R''}(t, q) = P_R(t^{-1}, q^{-1}) $.
	\end{enumerate}
\end{prop}
\begin{proof}
	The difference between two relators $ R $ and $ R' $ is the choice of the arc $ t_{\alpha} $ on $ T^2 - \alpha $. We consider only the case where $ k = 1 $ because $ l_k = l_1^k $ (resp. $ r_k = r_1^k $). Let $ t_{\alpha'} $ be the arc obtained by handlesliding $ t_{\alpha} $ across $ \alpha $ (with orientation), see Figure \ref{knot3_1} for an example on the trefoil knot. Then the intersection of $ t_{\alpha} $ and $ \beta $ is one-to-one corresponding to the intersection of $ t_{\alpha'} $ and $ \beta $; and each intersection point of $ \alpha $ and $ \beta $ corresponds to two intersection points of $ \beta $ with $ \alpha $ and $ t_{\alpha'} $, which are consecutive in $ \beta $. Thus, $ R' $ is the image of $ R $ in $ l_1 $ or $ r_1 $. It is clear that $ R' $ also comes from same Heegaard diagram as $ R $, and $ t_{\alpha} $ and $ t_{\alpha'} $ are isotopic in the solid torus $ V_{\alpha} $. Therefore, $ R $ and $ R' $ represent the same knot.

	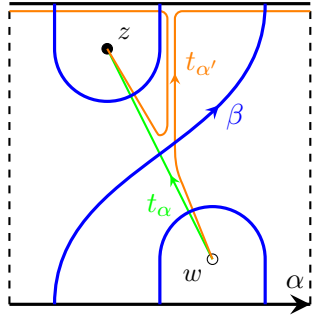
\begin{figure}[htbp]
		\centering
		\begin{tikzpicture}
			\def\y{2}
			\draw[very thick, decoration={markings, mark=at position 1 with {\arrow[scale=1.5]{>}}}, postaction={decorate}] (0,2*\y) -- (2*\y,2*\y)
			(0,0) -- (2*\y,0);
			\draw[thick,dashed] (0,0) -- (0,2*\y)
			(2*\y, 0) -- (2*\y,2*\y);
			\draw (1.9*\y, 0.15*\y) node {$ \alpha $};
			
			\draw[thick, green, decoration={markings, mark=at position 0.4 with \arrow{>}}, postaction={decorate}] (1.35*\y,0.3*\y) -- (0.65*\y,1.7*\y);
			\draw (1.35*\y,0.3*\y)circle(0.07) node[below left] {\small $ w $};
			\draw[fill=black] (0.65*\y,1.7*\y) circle(0.07) node[above right] {\small $ z $};
			\draw[green] (\y, 0.8*\y) node[below] {$ t_{\alpha} $};
			
			\draw[thick,orange,decoration={markings, mark=at position 0.25 with \arrow{>}}, postaction={decorate}] (1.35*\y,0.3*\y) to (1.15*\y, 0.8*\y)
			to[start angle=115, next angle=90] (1.1*\y,1.1*\y)
			-- (1.1*\y, 1.6*\y) node[right] {$ t_{\alpha'} $} -- (1.1*\y, 1.9*\y)
			to[next angle=0] (1.15*\y, 1.95*\y) -- (2*\y, 1.95*\y)
			(0, 1.95*\y) -- (\y, 1.95*\y)
			to[start angle=0, next angle=-90] (1.05*\y, 1.9*\y) -- (1.05*\y, 1.2*\y)
			to[next angle=-180] (\y, 1.125*\y)
			to[next angle=-230] (0.975*\y, 1.15*\y) -- (0.65*\y, 1.7*\y);
			
			\draw[very thick,blue, decoration={markings, mark=at position 0.3 with \arrow{>}}, postaction={decorate}] (0.3*\y, 0)
			to[start angle=90,next angle=90] (1.7*\y, 2*\y)
			(1.7*\y, 0) -- (1.7*\y, 0.3*\y)
			to[start angle=90,next angle=180] (1.35*\y, 0.65*\y)
			to[next angle=270] (\y, 0.3*\y) -- (\y, 0)
			(\y, 2*\y) -- (\y, 1.7*\y)
			to[start angle=-90,next angle=-180] (0.65*\y, 1.35*\y)
			to[next angle=-270] (0.3*\y, 1.7*\y) -- (0.3*\y, 2*\y);
			\draw[blue] (1.5*\y, 1.25*\y) node {$ \beta $};
		\end{tikzpicture}
		\caption{A Heegaard diagram compatible with the trefoil knot $ T_{2,3} $. The relator $ R $ derived from $ t_{\alpha} $ is $ \Xm Y X \Ym X Y $. The relator $ R' $ from $ t_{\alpha'} $ is $ \Xm\Ym Y YX \Ym YX Y $, which is the image of $ R $ under transformation $ l_1 $. It is further reduced to $ \Xm Y X^2 Y $.}
		\label{knot3_1}
	\end{figure}

	Since we require that $ t_{\alpha} $ is oriented from $ w $ to $ z $, the second transformation can be achieved by switching the two basepoints $ z $ and $ w $. Thus, the two relators represent the knot $ K $ and its mirror $ K^* $, respectively. Thereby, $ P_{R''}(t, q) = P_R(t^{-1}, q^{-1}) $ is directly obtained.
\end{proof}

\begin{example}
	A Heegaard diagram of the knot $ 10_{161} $ is illustrated in Figure \ref{fig:knot10_161} (as in \cite{GMM05}). The curve $ \beta $ gives the relator
	$$ R(X, Y) = X \Ym X \Ym \Xm Y \Xm \Ym X \Ym^2 X \Ym \Xm Y \Xm \Ym X \Ym X Y \Xm Y X Y \Xm Y. $$

	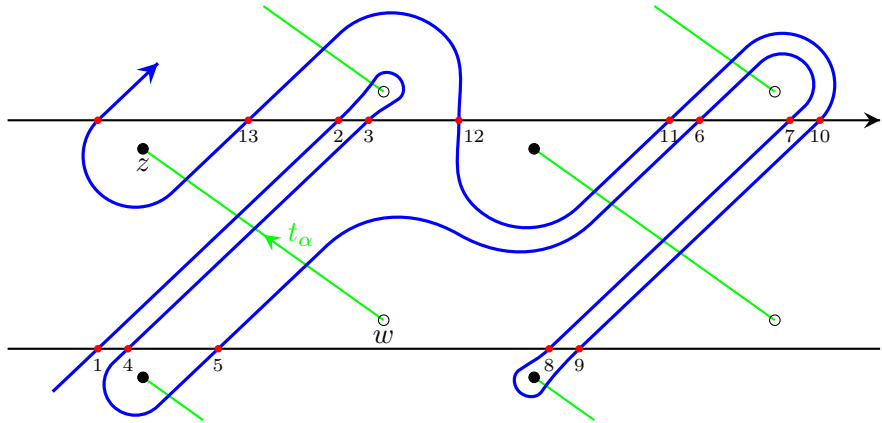
\begin{figure}[htbp]
		\centering
		\begin{tikzpicture}
			\def\x{4}
			\def\y{3.8}
			\draw[thick, decoration={markings, mark=at position 1 with {\arrow[scale=1.5]{>}}}, postaction={decorate},name path = alpha] (-0.3*\x,0) -- (2.6*\x,0)
			(-0.3*\x,0.8*\y) -- (2.6*\x,0.8*\y);
			
			\draw[thick, green, decoration={markings, mark=at position 0.5 with {\arrow[scale=1.5]{>}}}, postaction={decorate}]
			(0.95*\x,0.1*\y) -- (0.15*\x,0.7*\y);
			\draw[green] (0.675*\x,0.4*\y) node {$t_{\alpha}$};
			\draw[fill=black] (0.15*\x,0.7*\y) circle(0.07) node[below] {$z$};
			\draw (0.95*\x,0.1*\y) circle(0.07) node[below] {$w$};
			
			\draw[thick, green] (2.25*\x,0.1*\y) -- (1.45*\x,0.7*\y);
			\draw[fill=black] (1.45*\x,0.7*\y) circle(0.07);
			\draw (2.25*\x,0.1*\y) circle(0.07);
			
			\draw[thick, green] (0.15*\x,-0.1*\y) -- (0.35*\x,-0.25*\y);
			\draw[fill=black] (0.15*\x,-0.1*\y) circle(0.07);
			
			\draw[thick, green] (1.45*\x,-0.1*\y) -- (1.65*\x,-0.25*\y);
			\draw[fill=black] (1.45*\x,-0.1*\y) circle(0.07);
			
			\draw[thick, green] (0.95*\x,0.9*\y) -- (0.55*\x,1.2*\y);
			\draw (0.95*\x,0.9*\y) circle(0.07);
			
			\draw[thick, green] (2.25*\x,0.9*\y) -- (1.85*\x,1.2*\y);
			\draw (2.25*\x,0.9*\y) circle(0.07);
			\draw[very thick, blue, decoration={markings, mark=at position 1 with {\arrow[scale=1.5]{>}}}, postaction={decorate}, name path = beta] (-0.15*\x, -0.15*\y) -- (0.8*\x, 0.8*\y)
			to[start angle=45,next angle=55] (0.925*\x, 0.95*\y)
			to[next angle=-45] (\x, 0.95*\y)
			to[next angle=-145] (\x, 0.875*\y)
			to[next angle=-135] (0.9*\x, 0.8*\y) -- (0.05*\x, -0.05*\y)
			to[start angle=-135,next angle=-45] (0.05*\x, -0.2*\y)
			to[next angle=45] (0.2*\x, -0.2*\y) -- (0.75*\x, 0.35*\y)
			to[next angle=-30] (1.2*\x, 0.4*\y)
			to[next angle=45] (1.65*\x, 0.45*\y) -- (2.2*\x, \y)
			to[next angle=-45] (2.35*\x, \y)
			to[next angle=-135] (2.35*\x, 0.85*\y) -- (2.3*\x, 0.8*\y) -- (1.5*\x, 0*\y) 
			to[start angle=-135,next angle=-145] (1.4*\x, -0.075*\y)
			to[next angle=-45] (1.4*\x, -0.15*\y)
			to[next angle=55] (1.475*\x, -0.15*\y)
			to[next angle=45] (1.6*\x, 0*\y) -- (2.4*\x, 0.8*\y)
			to[start angle=45,next angle=135] (2.4*\x, 1.05*\y)
			to[next angle=225] (2.15*\x, 1.05*\y) -- (1.9*\x, 0.8*\y) -- (1.6*\x, 0.5*\y)
			to[start angle=-135,next angle=-225] (1.25*\x, 0.5*\y)
			to[next angle=90] (1.2*\x, 0.8*\y)
			to[next angle=135] (1.15*\x, 1.1*\y)
			to[next angle=225] (0.8*\x, 1.1*\y)
			-- (0.5*\x, 0.8*\y) -- (0.25*\x, 0.55*\y)
			to[start angle=-135,next angle=-225] (0, 0.55*\y)
			to[next angle=45] (0, 0.8*\y) -- (0.2*\x, \y);
			
			\fill [name intersections={of=beta and alpha, name=i, total=\t}] [red, every node/.style={above left, black, opacity=1}] \foreach \s in {1,...,\t}{(i-\s) circle (0.05)};
			\draw (0, 0) node[below] {\tiny $ 1 $};
			\draw (0.1*\x, 0) node[below] {\tiny $ 4 $};
			\draw (0.4*\x, 0) node[below] {\tiny $ 5 $};
			\draw (1.5*\x, 0) node[below] {\tiny $ 8 $};
			\draw (1.6*\x, 0) node[below] {\tiny $ 9 $};
			
			\draw (0.5*\x, 0.8*\y) node[below] {\tiny $ 13 $};
			\draw (0.8*\x, 0.8*\y) node[below] {\tiny $ 2 $};
			\draw (0.9*\x, 0.8*\y) node[below] {\tiny $ 3 $};
			\draw (1.25*\x, 0.8*\y) node[below] {\tiny $ 12 $};
			\draw (1.9*\x, 0.8*\y) node[below] {\tiny $ 11 $};
			\draw (2*\x, 0.8*\y) node[below] {\tiny $ 6 $};
			\draw (2.3*\x, 0.8*\y) node[below] {\tiny $ 7 $};
			\draw (2.4*\x, 0.8*\y) node[below] {\tiny $ 10 $};
		\end{tikzpicture}
		\caption{A Heegaard diagram of the knot $ 10_{161} $. This diagram is taken from \cite{GMM05}.}
		\label{fig:knot10_161}
	\end{figure}

	Label each $ X $-letter in order. Consider the primitive bigon $ D: X \Ym^2 X \Ym \Xm Y \Xm $ from $ x_5 $ to $ x_8 $ for instance. It contains only one primitive bigon $ D_1: X_6 \Ym \Xm_7 $ of height one, which is the second type in Table \ref{table:elementary_bigons}. The square domain $ S $ is the third type in Table \ref{table:b_diff_word}. Thus
	$$ P(D) = P(S) + P(D_1) = (0, 1) + (0, 1) = (0, 2). $$
	Removing primitive bigons with $ n_w = 0 $ yields
	$$ X_2 \Ym \ \Ym \ Y \ Y \ \Ym. $$
	So that we obtain $ M(x_2) = 0 $. Moreover, the Alexander and Maslov gradings of the generators are listed in Table \ref{table:knot_10_161}, and the Poincar\'e polynomial is
	$$ P_R(t,q) = t^{-3} + (1 + q)t^{-2} + 2qt^{-1} + 3q^2 + 2q^3t + (q^4 + q^5)t^2 + q^6t^3. $$
	\begin{table}[htbp]
	\centering
	\begin{tabular}{|c|c|c|c|c|c|c|c|c|c|c|c|c|c|}
		\hline
		& $ x_1 $ & $ x_2 $ & $ x_3 $ & $ x_4 $ & $ x_5 $ & $ x_6 $ & $ x_7 $ & $ x_8 $ & $ x_9 $ & $ x_{10} $ & $ x_{11} $ & $ x_{12} $ & $ x_{13} $ \\
		\hline
		$ F $ & -2 & -3 & -2 & -1 & 0 & 0 & 1 & 2 & 3 & 2 & 1 & 0 & -1  \\
		\hline
		$ M $ & 0 & 0 & 1 & 1 & 2 & 2 & 3 & 5 & 6 & 4 & 3 & 2 & 1  \\
		\hline
	\end{tabular}
	\vskip\baselineskip
	\caption{The Alexander and Maslov gradings of the knot Floer homology of knot $ 10_{161} $.}
	\label{table:knot_10_161}
	\end{table}
\end{example}

\begin{example}
	Figure \ref{fig:knotD6} shows a Heegaard diagram of the knot $ D_+(T_{2,3}, 6) $ (the $ 6 $-twisted positive Whitehead double of the right-handed trefoil), as in \cite{HO05}. The presentation is
	\begin{equation}\label{eq:pre_knotD6}
		\langle X, Y \,|\, X \Ym \Xm^3 Y X^3 Y \Xm^3 \Ym X^2 \Ym \Xm^3 Y X^3 Y \Xm^3 \Ym X^4 \rangle.
	\end{equation}

	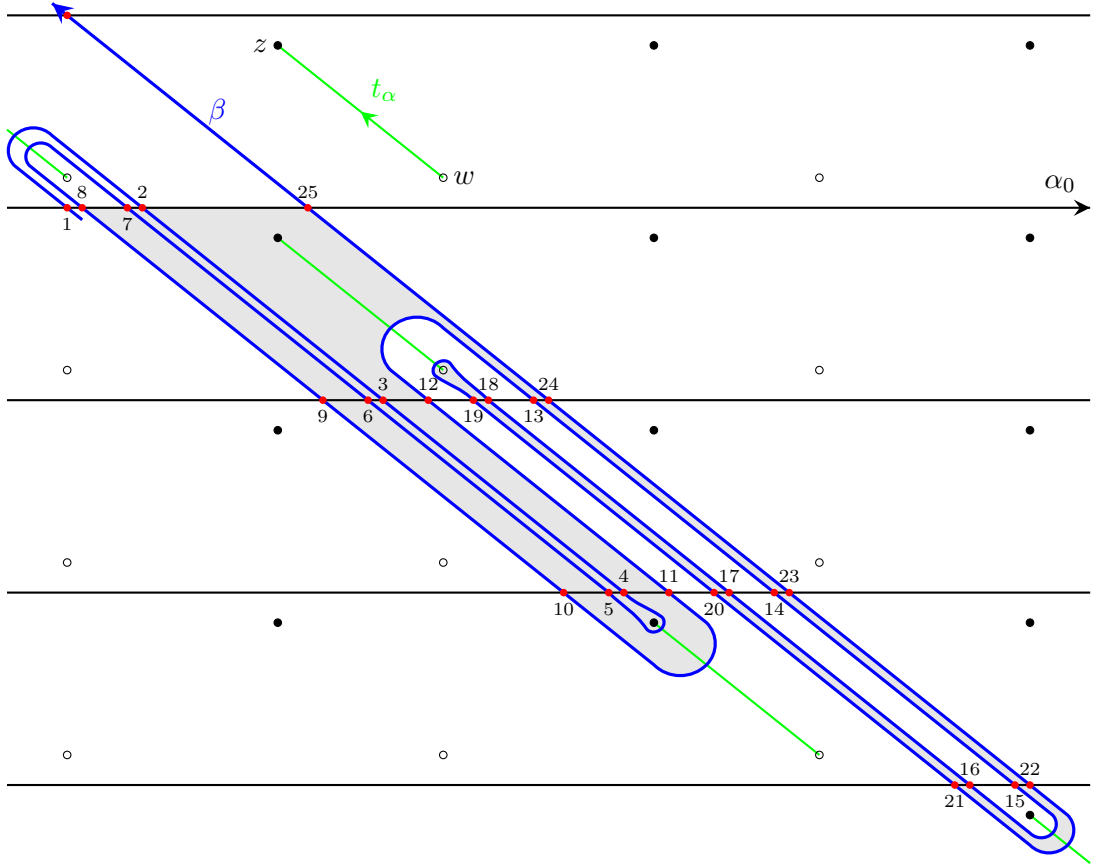
\begin{figure}[htbp]
		\centering
		\begin{tikzpicture}
			\def\x{2}
			\def\y{1.6}
			\draw[fill = gray!20] (0.1*\x, 0) -- (3.3*\x, -3.2*\y)
			-- (3.9*\x, -3.8*\y)
			to[start angle=-45, next angle=45] (4.25*\x, -3.8*\y)
			to[next angle=135] (4.25*\x, -3.45*\y) -- (4*\x, -3.2*\y)
			-- (2.4*\x, -1.6*\y) -- (2.15*\x, -1.35*\y)
			to[start angle=135, next angle=45] (2.15*\x, -\y) 
			to[next angle=-45] (2.5*\x, -\y) -- (3.1*\x, -1.6*\y)
			-- (6.3*\x, -4.8*\y) -- (6.55*\x, -5.05*\y)
			to[start angle=-45, next angle=-135] (6.55*\x,-5.2*\y) 
			to[next angle=-225] (6.4*\x, -5.2*\y) -- (6*\x, -4.8*\y)
			-- (2.8*\x, -1.6*\y) 
			to[start angle=135,next angle=125] (2.55*\x, -1.3*\y)
			to[next angle=225] (2.45*\x, -1.3*\y)
			to[next angle=325] (2.45*\x, -1.4*\y)
			to[next angle=315] (2.7*\x, -1.6*\y)
			-- (5.9*\x, -4.8*\y) -- (6.4*\x, -5.3*\y)
			to[start angle=-45, next angle=45] (6.65*\x, -5.3*\y)
			to[next angle=135] (6.65*\x, -5.05*\y) -- (6.4*\x, -4.8*\y)
			-- (1.6*\x, 0) -- (0.1*\x, 0);
			\draw[thick, decoration={markings, mark=at position 0.4 with {\arrow[scale=1.5]{>}}}, postaction={decorate},name path = alpha] (-0.4*\x,1.6*\y) -- (6.8*\x,1.6*\y)
			(-0.4*\x,0) -- (6.8*\x,0)
			(-0.4*\x,-1.6*\y) -- (6.8*\x,-1.6*\y)
			(-0.4*\x,-3.2*\y) -- (6.8*\x,-3.2*\y)
			(-0.4*\x,-4.8*\y) -- (6.8*\x,-4.8*\y);
			\draw (6.6*\x, 0.2*\y) node{$ \alpha_0 $};
			
			\begin{scope}
				\draw[thick, green, decoration={markings, mark=at position 0.5 with {\arrow[scale=1.5]{>}}}, postaction={decorate},] (2.5*\x,0.25*\y) -- (1.95*\x,0.8*\y) node[above right] {$ t_{\alpha} $} -- (1.4*\x,1.35*\y);
				\draw (2.5*\x,0.25*\y) circle(0.05) node[right] {$ w $};
				\draw[fill=black] (1.4*\x,1.35*\y) circle(0.05) node[left] {$ z $};
				
				\draw[thick, green] (2.5*\x,-1.35*\y) -- (1.4*\x,-0.25*\y);
				\draw (2.5*\x,-1.35*\y) circle(0.05);
				\draw[fill=black] (1.4*\x,-0.25*\y) circle(0.05);
				
				\draw (2.5*\x,-2.95*\y) circle(0.05);
				\draw[fill=black] (1.4*\x,-1.85*\y) circle(0.05);
				
				\draw (2.5*\x,-4.55*\y) circle(0.05);
				\draw[fill=black] (1.4*\x,-3.45*\y) circle(0.05);
			\end{scope}
			\begin{scope}[shift={(2.5*\x, 0)}]
				\draw (2.5*\x,0.25*\y) circle(0.05);
				\draw[fill=black] (1.4*\x,1.35*\y) circle(0.05);
				
				\draw (2.5*\x,-1.35*\y) circle(0.05);
				\draw[fill=black] (1.4*\x,-0.25*\y) circle(0.05);
				
				\draw (2.5*\x,-2.95*\y) circle(0.05);
				\draw[fill=black] (1.4*\x,-1.85*\y) circle(0.05);
				
				\draw[thick, green] (2.5*\x,-4.55*\y) -- (1.4*\x,-3.45*\y);
				\draw (2.5*\x,-4.55*\y) circle(0.05);
				\draw[fill=black] (1.4*\x,-3.45*\y) circle(0.05);
			\end{scope}
			
			\draw[thick, green] (0, 0.25*\y) -- (-0.4*\x,0.65*\y);
			\draw (0, 0.25*\y)circle(0.05);
			\draw (0, -1.35*\y)circle(0.05);
			\draw (0, -2.95*\y)circle(0.05);
			\draw (0, -4.55*\y)circle(0.05);
			
			\draw[thick, green] (6.4*\x,-5.05*\y) -- (6.8*\x,-5.45*\y);
			\draw[fill=black] (6.4*\x,-5.05*\y) circle(0.05);
			\draw[fill=black] (6.4*\x,-3.45*\y) circle(0.05);
			\draw[fill=black] (6.4*\x,-1.85*\y) circle(0.05);
			\draw[fill=black] (6.4*\x,-0.25*\y) circle(0.05);
			\draw[fill=black] (6.4*\x,1.35*\y) circle(0.05);
			
			\draw[very thick, blue, decoration={markings, mark=at position 1 with {\arrow[scale=1.5]{>}}}, postaction={decorate},name path = beta] (0.1*\x, -0.1*\y) -- (-0.35*\x, 0.35*\y)
			to[start angle=135,next angle=45] (-0.35*\x, 0.6*\y) 
			to[next angle=-45] (-0.1*\x, 0.6*\y) -- (0.5*\x, 0)
			-- (3.7*\x, -3.2*\y)
			to[start angle=-45,next angle=-35] (3.95*\x, -3.4*\y)
			to[next angle=-135] (3.95*\x,-3.5*\y)
			to[next angle=-235] (3.85*\x,-3.5*\y)
			to[next angle=-225] (3.6*\x, -3.2*\y)
			-- (0.4*\x, 0) -- (-0.1*\x, 0.5*\y) 
			to[start angle=135,next angle=225] (-0.25*\x,0.5*\y) 
			to[next angle=315] (-0.25*\x, 0.35*\y) -- (0.1*\x, 0)
			-- (3.3*\x, -3.2*\y) -- (3.9*\x, -3.8*\y)
			to[start angle=-45, next angle=45] (4.25*\x, -3.8*\y)
			to[next angle=135] (4.25*\x, -3.45*\y) -- (4*\x, -3.2*\y)
			-- (2.4*\x, -1.6*\y) -- (2.15*\x, -1.35*\y)
			to[start angle=135, next angle=45] (2.15*\x, -\y) 
			to[next angle=-45] (2.5*\x, -\y) -- (3.1*\x, -1.6*\y)
			-- (6.3*\x, -4.8*\y) -- (6.55*\x, -5.05*\y)
			to[start angle=-45, next angle=-135] (6.55*\x,-5.2*\y)
			to[next angle=-225] (6.4*\x, -5.2*\y) -- (6*\x, -4.8*\y)
			-- (2.8*\x, -1.6*\y) 
			to[start angle=135,next angle=125] (2.55*\x, -1.3*\y)
			to[next angle=225] (2.45*\x, -1.3*\y)
			to[next angle=325] (2.45*\x, -1.4*\y)
			to[next angle=315] (2.7*\x, -1.6*\y)
			-- (5.9*\x, -4.8*\y) -- (6.4*\x, -5.3*\y)
			to[start angle=-45, next angle=45] (6.65*\x, -5.3*\y)
			to[next angle=135] (6.65*\x, -5.05*\y) -- (6.4*\x, -4.8*\y)
			-- (-0.1*\x, 1.7*\y);
			\draw[blue] (\x, 0.8*\y) node{$ \beta $};
			
			\fill [name intersections={of=beta and alpha, name=i, total=\t}] [red, every node/.style={above left, black, opacity=1}] \foreach \s in {1,...,\t}{(i-\s) circle (0.05)};
			\draw (0, 0) node[below] {\tiny $ 1 $};
			\draw (0.5*\x, 0) node[above] {\tiny $ 2 $};
			\draw (0.4*\x, 0) node[below] {\tiny $ 7 $};
			\draw (0.1*\x, 0) node[above] {\tiny $ 8 $};
			\draw (1.6*\x, 0) node[above] {\tiny $ 25 $};
			
			\draw (2.1*\x, -1.6*\y) node[above] {\tiny $ 3 $};
			\draw (2*\x, -1.6*\y) node[below] {\tiny $ 6 $};
			\draw (1.7*\x, -1.6*\y) node[below] {\tiny $ 9 $};
			\draw (2.4*\x, -1.6*\y) node[above] {\tiny $ 12 $};
			\draw (3.1*\x, -1.6*\y) node[below] {\tiny $ 13 $};
			\draw (2.8*\x, -1.6*\y) node[above] {\tiny $ 18 $};
			\draw (2.7*\x, -1.6*\y) node[below] {\tiny $ 19 $};
			\draw (3.2*\x, -1.6*\y) node[above] {\tiny $ 24 $};
			
			\draw (3.7*\x, -3.2*\y) node[above] {\tiny $ 4 $};
			\draw (3.6*\x, -3.2*\y) node[below] {\tiny $ 5 $};
			\draw (3.3*\x, -3.2*\y) node[below] {\tiny $ 10 $};
			\draw (4*\x, -3.2*\y) node[above] {\tiny $ 11 $};
			\draw (4.7*\x, -3.2*\y) node[below] {\tiny $ 14 $};
			\draw (4.4*\x, -3.2*\y) node[above] {\tiny $ 17 $};
			\draw (4.3*\x, -3.2*\y) node[below] {\tiny $ 20 $};
			\draw (4.8*\x, -3.2*\y) node[above] {\tiny $ 23 $};
			
			\draw (6.3*\x, -4.8*\y) node[below] {\tiny $ 15 $};
			\draw (6*\x, -4.8*\y) node[above] {\tiny $ 16 $};
			\draw (5.9*\x, -4.8*\y) node[below] {\tiny $ 21 $};
			\draw (6.4*\x, -4.8*\y) node[above] {\tiny $ 22 $};
		\end{tikzpicture}
		\caption{A Heegaard diagram of the knot $ D_+(T_{2,3}, 6) $ (taken from \cite{HO05}).}
		\label{fig:knotD6}
	\end{figure}

	\noindent Label each $ X $-letter in order. The corresponding word of the shaded domain $ D $ in Figure \ref{fig:knotD6} is
	$$ \Xm^3 \Ym X^2 \Ym \Xm^3 Y X^3 Y \Xm^3 \Ym X^4, $$
	which is a primitive bigon from $ x_{8} $ to $ x_{25} $. Determine $ P(D) $ directly from the word as follows: two height $ -1 $ points $ x_{9} $ and $ x_{24} $ are connected by elementary bigons:
	$$ \begin{aligned}
		D_1: ~~ & \Xm^2 \Ym X^2, \quad & (x_{9} \rightarrow x_{12}); \\
		D_2: ~~ & X \Ym \Xm, & (x_{12} \rightarrow x_{13} ); \\
		D_3: ~~ & \Xm^3 Y X^3, & (x_{13} \rightarrow x_{18} ); \\
		D_4: ~~ & X Y \Xm, & (x_{18} \rightarrow x_{19} );  \\
		D_5: ~~ & \Xm^3 \Ym X^3, & (x_{19} \rightarrow x_{24} ).
	\end{aligned} $$
	Note that $ D_1 $ and $ D_5 $ have the same orientation as $ D $, thus
	$$ \begin{aligned}
		P(D) & = P(S) + \sum_{i=1}^{5} P(D_i) \\
			& = (-1, -1) + (-1, 0) + (0, 1) + (1, 0) + (0, -1) + (-1, 0) \\
			& = (-2, -1),
		\end{aligned} $$
	which is consistent with the diagram shown in Figure \ref{fig:knotD6}.

	The following steps are to find the unique $ X $-letter that generates $ \widehat{HF}(S^3, w) $:
	\begin{enumerate}
		\item Remove primitive bigons with $ n_w = 0 $, the resulting relator is
		$$ X_1 \Ym \ Y \Xm_8 \ \Ym \ Y \ X_{25}. $$
		\item Note that the points $ x_1 $ and $ x_8 $ are connected by three primitive bigons in the original relator $ R $:
		$$ \begin{aligned}
				E_1: ~~ & X \Ym \Xm, \quad & (x_{1} \rightarrow x_{2}); \\
				E_2: ~~ & \Xm^3 Y X^3, & (x_{2} \rightarrow x_{7} ); \\
				E_3: ~~ & X Y \Xm, & (x_{7} \rightarrow x_{8} ).
		\end{aligned} $$
		So that the $ w $-grading shift is
		$$ w(x_1) - w(x_8) = w(x_1) - w(x_2) + w(x_2) - w(x_7) + w(x_7) - w(x_8) = 1 + 0 - 1 = 0. $$
		On the other hand, $ w(x_8) - w(x_{25}) = -1 $, since the two points $ x_8 $ and $ x_{25} $ are connected by the bigon $ D $ in above. Therefore, the subword $ X_1 \Ym \ Y \Xm_8 $ is removed in the second step, thus giving
		$$ \Ym \ Y \ X_{25}. $$
		\item We obtain $ M(x_{25}) = 0 $.
		\end{enumerate}
	The Poincar\'e polynomial is
	\begin{equation}\label{eq:pp_knotD6}
		P_R(t, q) = (4q^{-1} + 2q)t^{-1} + (4q^2 + 9) + (2q^3 + 4q)t.
	\end{equation}
\end{example}

\clearpage
\section{$ (1,1) $ knots in lens spaces}\label{sec:lens}
In this section, we extend the algorithm for $ (1,1) $ knots in lens spaces, which are automatically rationally null-homologous. See \cite{OS11} for details. 

Let $ K $ be a $ (1,1) $ knot in the lens space $ L(p, q) (p>0) $ and $ \class{X, Y\,|\,R(X, Y)}$ be a presentation of $\pi_1(L(p,q)\setminus K)$ obtained from a $ (1,1) $ Heegaard diagram $ (T^2, \alpha, \beta, z, w) $ of $ (L(p,q),K) $, satisfying
\begin{enumerate}
\item The curves $ \alpha $ and $ \beta $ are oriented so that
$$ \# \{X \,|\, X \in R(X,Y)\} - \#\{\Xm \,|\, \Xm \in R(X,Y)\} = p;$$
\item The relator $R(X,Y)$ is cyclically reduced.
\end{enumerate}

For a knot $K$ in $L(p,q)$, the knot Floer homology $\widehat{HFK}(L(p,q),K)$ can be decomposed as a directe sum
$$ \widehat{HFK}(L(p,q),K) = \bigoplus_{\mathfrak{s} \in {\rm Spin}^c(L(p,q))} \widehat{HFK}(L(p.q),K;\mathfrak{s}) $$
with respect to ${\rm Spin}^c$ structures on $L(p,q)$. The Alexander grading and the Maslov grading are defined for each direct summand $\widehat{HFK}(L(p.q),K;\mathfrak{s})$.
Note that there is an affine isomorphism between the space of $ {\rm Spin}^c $ structures over a 3-manifold and its second cohomology. In particular, 
$$ {\rm Spin}^c(L(p,q)) \cong H^2(L(p,q), \mathbb{Z}) \cong \mathbb{Z}/p\mathbb{Z}. $$
Let $ {\rm Spin}^c(L(p,q)) = \{ \mathfrak{s}_0, \cdots, \mathfrak{s}_{p-1} \} $. The set of relative $ {\rm Spin}^c $ structures for $ (L(p, q), K) $ is denoted by
$$  \underline{{\rm Spin}^c} (L(p, q), K) \cong {\rm Spin}^c(L(p,q))  \times \mathbb{Z} = \{ \mathfrak{(s}_i,s) \,|\,  i = 0, \cdots, p-1, \ s \in \mathbb{Z} \}. $$
For each $ \mathfrak{s}_i \in {\rm Spin}^c(L(p,q)) $,
$$ \begin{aligned}
	\widehat{HFK}(L(p,q), K; \mathfrak{s}_i) & = \bigoplus_{\small \{ \underline{\mathfrak{t}} \in \underline{{\rm Spin}^c} (L(p, q), K) \,|\, \underline{\mathfrak{t}}  \text{ extends } \mathfrak{s}_i \}} \widehat{HFK}(L(p,q), K; \underline{\mathfrak{t}}) \\
	& = \bigoplus_{s \in \mathbb{Z}} \widehat{HFK}(L(p,q), K; (\mathfrak{s}_i,s)). 
\end{aligned}$$
Note that $ \widehat{HF}(L(p,q); \mathfrak{s}_i) \cong \mathbb{Z} $. More generally, a rational homology sphere $Y$ is called an \emph{L-space} if $ \widehat{HF}(Y; \mathfrak{s}) \cong \mathbb{Z} $ for all $ \mathfrak{s} \in {\rm Spin}^c(Y) $. 
Ignoring the basepoint $z$, the Heegaard diagram $ (T^2, \alpha, \beta, w) $ specifies the lens space $ L(p,q) $, so that there exist exactly $ p $ intersection points that generates $\widehat{HF}(L(p,q);\mathfrak{s}_i) $, one for each $\mathfrak{s}_i$. Denoted by $ x_{\mathfrak{s}_i}$ the intersection point generating $\widehat{HF}(L(p,q);\mathfrak{s}_i) $ and by $ X_{\mathfrak{s}_i} $ the corresponding $X$-letter in the relator $ R $. With $w$ fixed, each intersection point $x$ of $\alpha$ and $\beta$ specify a $ {\rm Spin}^c$ structure $\mathfrak{s}_w(x)$ of $L(p,q)$ \cite{OS01a}. Let $ \mathfrak{S} $ be the set of intersection points of $ \alpha $ and $ \beta $ (i.e., the set of $X$-letters in $ R $), and
$$ \mathfrak{S}_i := \{ x \in \mathfrak{S} \,|\, \mathfrak{s}_w (x) = \mathfrak{s}_i  \}, \ i = 0, \cdots, p-1. $$
For each $ \mathfrak{s}_i$, $\widehat{CF}(L(p, q); \mathfrak{s}_i) $ and $ \widehat{CFK}(L(p, q), K; \mathfrak{s}_i) $ are both generated by $ \mathfrak{S}_i $.

The key of Algorithm \ref{algo:main} is to determine the number of basepoints in each primitive bigon, and the first three steps works for (1,1) knots in lens spaces as well. 
In Step \ref{step_absolute}, we got the unique generator of $ \widehat{HF}(S^3) $ by reducing the relator for (1,1) knots in $S^3$. For $(1,1)$ knots in $L(p,q)$, the same process give $ p $ generators $X_{\mathfrak{s}_0},\cdots,X_{\mathfrak{s}_{p-1}}$ in different $ {\rm Spin}^c $ structures. Therefore, we can apply Step \ref{step_absolute} to each $ {\rm Spin}^c $ structure separately. We only need to determine the $ {\rm Spin}^c $ structure for each intersection point, that is, determine $ \mathfrak{S}_i $ from the relator $ R(X, Y) $.  

\begin{lemma}\label{lemma:bigon_to_spinc}
	If two $X$-letters are connected by a sequence of primitive bigons, then they correspond to the same $ {\rm Spin}^c $ structure of $ L(p, q) $.
\end{lemma}
\begin{proof}
	Suppose there is a sequence $ X_{d_1}, \cdots, X_{d_n} $ of $X$-letters such that 
	any two adjacent $ X_{d_i} $ and $ X_{d_{i+1}} $ are connected by a primitive bigon $ D_i, i = 1, \cdots, n-1 $. The difference $ \varepsilon (x_{d_i}, x_{d_{i+1}}) = 0 $ (see \cite[Section 2.4]{OS01a}) since $ D_i $ is a primitive bigon. By \cite[Lemma 2.19]{OS01a},
	$$ \mathfrak{s}_w (x_{d_{i+1}}) - \mathfrak{s}_w (x_{d_{i}}) = {\rm PD} [\varepsilon (x_{d_i}, x_{d_{i+1}})] = 0. $$
	Therefore, all $ X_{d_i}, i = 1, \cdots, n $ correspond to the same $ {\rm Spin}^c $ structure.
\end{proof}

\begin{lemma}
Each $ X $-letter in $ R(X, Y) $ is connected to some $ X_{\mathfrak{s}_i}$ ($0\leqslant i<p$) via a sequence of primitive bigons.
\end{lemma}
\begin{proof}
	Without loss of generality, we write the relator $ R $ so that our consideration $ X $-letter is the first letter $ X_1 $ in $ R $.  For each $ X_{\mathfrak{s}_i} \in R $, denote its height (relative to $ X_1 $) by $ h_i $. Claim that $ h_i \not \equiv h_j \pmod{p} $ if $ i \neq j $.
	
	Suppose there exist $ i $ and $ j $ such that $ h_i  \equiv h_j \pmod{p} $, i.e., $ h_i - h_j = pu $ for some integer $ u $ (assume without loss of generality that $ u \geq 0 $). Pick the word $ R^{u+1} $, the $ u+1 $ copies of $ R $, which corresponds to a subarc of a lift of the $ \beta $ curve. Let $ X_k^{(l)} $ denote the $ k $-th $ X $ letter in the $ l $-th copy of $ R $. By assumption, the height of $ X_{\mathfrak{s}_j}^{(u+1)} $ is
	$$ h(X_{\mathfrak{s}_j}^{(u+1)}) = pu + h(X_{\mathfrak{s}_j}) = h_i = h(X_{\mathfrak{s}_i}) . $$
	Thus $ X_{\mathfrak{s}_i} $ and $ X_{\mathfrak{s}_j}^{(u+1)} $ lie in the same lift of the $ \alpha $ curve in the universal cover $\mathbb{C} $. Denoted by
	$ X_{d_1}, X_{d_2}, \cdots, X_{d_n} $ all the $ X $-letters between $ X_{\mathfrak{s}_i} $ and $ X_{\mathfrak{s}_j}^{(u+1)} $ in $ R^{u+1} $ that has height $ h_i $, then they give a sequence of primitive bigons from $ X_{\mathfrak{s}_i} $ to $ X_{\mathfrak{s}_j}^{(u+1)} $. Thus $ s_w (X_{\mathfrak{s}_i}) = s_w(X_{\mathfrak{s}_j}) $ by Lemma \ref{lemma:bigon_to_spinc}, so $ i = j $.
	
	In particular, there is an $ h_i \equiv 0 \pmod{p} $. Suppose $ h_i = p v $ for some integer $ v $. Note that $ h(X_1) = 0 $. The same argument as above gives a sequence of primitive bigons connecting $ X_1 $ and $ X_{\mathfrak{s}_i}^{(v+1)} $.
\end{proof}

It follows that we can divide the set $\mathfrak{S}$ of generators into $p$ disjoint subsets directly from the relator $ R $:
\begin{cor}
	$ X_k \in \mathfrak{S}_i $ if and only if $ X_k $ and $ X_{\mathfrak{s}_i} $ are connected by a sequence of primitive bigons, that is,
\begin{equation}\label{eq:dividegenerators}
		\mathfrak{S}_i = \{ X_k \in R(X, Y) \,|\, \text{ $ X_k $ and $ X_{\mathfrak{s}_i} $ are connected by primitive bigons} \}. 
	\end{equation}
\end{cor}

Now we modify Algorithm \ref{algo:main} for $ (1,1) $ knots in lens spaces.

\begin{algo}\label{algo:lensspace}
	Let $K$ be a $(1,1)$ knot in the lens space $ L(p, q) $ and $ \langle X, Y \,|\, R(X, Y) \rangle $ be a presentation of $\pi_1(L(p, q) \setminus K)$ from a $(1,1)$ Heegaard diagram. Let $ \mathfrak{S} $ be the set of  letters $ X $ and $ \Xm $ in the relator $ R $. The chain complex $\widehat{CFK}(L(p, q),K)$ has generators corresponding to elements in $ \mathfrak{S} $ and vanishing differentials. The first two steps (Step \ref{step_bigon} and \ref{step_basepoint}) of determining Alexander and Maslov gradings are the same as Algorithm \ref{algo:main}.
	\renewcommand{\theenumi}{\Roman{enumi}}
	\begin{enumerate}
		\item[(III')] We apply the process of reducing the relator (Step \ref{step_absolute}) first. The resulting relator $ R''(X, Y) $ contains exactly $ p $ $ X $-letters $ X_{\mathfrak{s}_0}, \cdots,X_{\mathfrak{s}_{p-1}}$. We divide $ \mathfrak{S} $ into $ p $ disjoint subsets $ \mathfrak{S}_0, \cdots, \mathfrak{S}_{p-1} $ by \eqref{eq:dividegenerators}. $\widehat{HFK}(L(p, q), K; \mathfrak{s}_i) $ is then generated by $ \mathfrak{S}_i $.
		
		\item[(IV')] Apply Step \ref{step_relative} and \ref{step_absolute}  of Algorithm \ref{algo:main} for each set $ \mathfrak{S}_i $ to obtain the Alexander grading and the Maslov grading of the homology $ \widehat{HFK}(L(p,q), K; \mathfrak{s}_i) $, respectively.
	\end{enumerate}
\end{algo}

\clearpage
\section{On general presentation of the fundamental group}\label{sec:alexander}

In previous sections, we always assume that the group presentation comes from a $ (1,1) $ Heegaard diagram. This ensures that the curve $\beta$ and its image in the covering transformation do not intersect, a fact that we use several times when determining the number of basepoints contained in each primitive bigon. However, Algorithm \ref{algo:main} can compute the knot Floer homology of a (1,1) knot $K$ from a presentation of $\pi_1(S^3\setminus K)$ which does not come from a (1,1) Heegaard diagram (Example \ref{example:not_geometric_but_homology}). Indeed, Algorithm \ref{algo:main} itself does not a priori require the group to be the fundamental group of a knot complement.


\begin{mydef}
Let $ P=\langle X, Y \,|\, R(X, Y) \rangle $ be a presentation of a group $G$ with $R(X,Y)$ cyclically reduced. $P$ is \emph{quasi-geometric} if
\begin{enumerate}
\item $ \# \{X \,|\, X \in R(X,Y)\} - \#\{\Xm \,|\, \Xm \in R(X,Y)\} = +1 $;
\item all subwords of the form $ X Y^k \Xm $ (or $\Xm Y^k X$) must have $ |k| = 1 $;
\item there are no two subwords of the form $ X Y^l X $ (or $ \Xm \Ym^l \Xm $) and $ X Y^{l'} X $ (or $ \Xm \Ym^{l'} \Xm $) that satisfy $\abs{l-l'} > 1 $, where $ l, l' \in \bbz $.
\end{enumerate}
\end{mydef}

For a quasi-geometric presentation, Algorithm \ref{algo:main} (Step 1-3) can be applied: the primitive disk words can be enumerated; their orientations can be determined by counting the positive and negative elementary words that it contains; and the functions $ n_z $ and $ n_w $ on primitive disk words can be formally computed. Hence we get a relatively bigraded chain complex with trivial differential.

For a quasi-geometric presentation of the fundamental group of a $(1,1)$ knot $K$, the homology of the above chain complex may not be isomorphic to the knot Floer homology $\widehat{HFK}(S^3,K)$. For example, the following presentation
\begin{equation}\label{eqn:trefoil_presentation}
\langle X, Y\,|\, YX^3 Y \Xm \Ym \Xm \rangle
\end{equation}
of the fundamental group of the trefoil knot $T_{2,3}$ is quasi-geometric. However, the homology of \eqref{eqn:trefoil_presentation} from Algorithm \ref{algo:main} (Step 1-3) has rank 5, while the rank of its knot Floer homology is $3$. However, we will show that its Euler characteristic agrees with the (unnormalized) Alexander polynomial.

The Alexander polynomial of a knot $K$ can be computed from a presentation of $ \pi_1(S^3 \setminus K) $ as follows \cite{Rol76}:
\begin{enumerate}
	\item For a knot $ K $ in $ S^3 $, the knot group $ G \triangleq \pi_1(S^3 \setminus K) $ has a presentation of the form
	$$ \langle Y, g_1, \cdots, g_p \,|\, r_1, \cdots, r_p \rangle, $$
	where $ Y \mapsto 1 $ and $ g_i \mapsto 0 $ under the abelianization $ G \rightarrow G/[G: G] \cong \mathbb{Z} $.
	
	\item The commutator subgroup $ C = [G: G] $ of $ G $ is generated by all words of the form:
	$$ Y^{-k} g_i^{\pm 1} Y^k. $$
	Each relator $ r_i $ is conjugated to a word $ r'_i $, which is a product of words of this form.
	
	\item Let $ \widetilde{X} $ be the infinite cyclic cover of the knot complement $ X = S^3 \setminus K $. Then the covering map $ p: \widetilde{X} \rightarrow X $ gives a $ \Lambda \triangleq \mathbb{Z}[t,t^{-1}] $-module isomorphism:
	$$ p_*: H_1(\widetilde{X};\mathbb{Z}) \rightarrow C/[C:C]. $$
	So we can obtain a $ \Lambda $-module presentation of $ C/[C:C] $ by taking the image of $ g_i $ under abelianization as the generator, say $ \alpha_i $, and replacing $ Y^{-k} g_i^{\pm 1} Y^k $ by $ \pm t^k \alpha_i $ in the relator $ r'_i $ (multiplication becomes addition). Therefore, the Alexander polynomial is given by
	$$ H_1(\widetilde{X};\mathbb{Z}) \cong \Lambda/\left(\Delta_K(t)\right). $$
\end{enumerate}

When $ p = 1 $, $ H_1 (\widetilde{X}) $ is a cyclic $ \Lambda $-module generated by a generator $ \alpha $. Rewrite the generator $ g $ of $ G $ by $ a $ and label each $ g $-letter in the relator $ r' $ by $ a_i $ (signed) in sequence. An alternative description of the last step is as follows: Since the relator $ r' $ is a product of words form like $ Y^{-k} g^{\pm 1} Y^{k} $, we define two gradings for each $ a_i $ as
$$ A(a_i) = k, \quad (-1)^{S(a_i)} = \pm 1 \quad {\rm for } \quad Y^{-k} a_i^{\pm 1} Y^k. $$
Then the (unnormalized) Alexander polynomial is
$$ \sum_i (-1)^{S(a_i)} t^{A(a_i)}. $$

\begin{prop}\label{prop:Euler2Alexander}
Let $ \langle X, Y\,|\,R(X, Y) \rangle $ be a quasi-geometric presentation of a $ (1, 1) $ knot $ K $ in $ S^3 $. Algorithm \ref{algo:main} (Step 1-3) gives a chain complex with two relative gradings $ F $ and $ M $, whose Euler characteristic satisfies
	\begin{equation}\label{eq:Euler2Alexander}
		\sum_{i} (-1)^{M(x_i)} t^{F(x_i)} \overset{\circ}{=} \Delta_K(t).
	\end{equation}
Where $f(t)\overset{\circ}{=} g(t)$ if $f(t)=\pm t^c g(t)$ for some integer $c$.
\end{prop}

\begin{proof}
	By assumption, the relator $ R(X, Y) $ maps to $ X + k Y $ for some integer $ k $ under abelianization. Let $ a = Y^kX $ and $ R'(a, Y) $ be the reduced relator of $ R(Y^{-k}a, Y) $. Consider the presentation
	$$ \langle Y, a \,|\, R'(a, Y) \rangle. $$
	It is easy to see that $ Y \mapsto 1 $ and $ a \mapsto 0 $ under abelianization, so that we can do the above procedure. Label each $ a $-letter and the origin $ X $-letter in sequence, and require that $ a_i $ corresponds to the $ i $-th $ X $-letter. Claim that
	\begin{equation}\label{eq:2grdings_are_agree}
		F(x_i) - F(x_j) = A(a_i) - A(a_j), \quad M(x_i) - M(x_j) \equiv S(a_i) - S(a_j) \pmod{2}.
	\end{equation}
	In fact, it is sufficient to show that the two equations hold in the case where $ x_i, x_j $ are two endpoints of a primitive disk word.
	
	Let $ W_1^n $ be a primitive disk word.  By definition, we have $ S(a_1) - S(a_n) \equiv 1 \pmod{2} $. On the other hand,
	$$ M(x_1) - M(x_n) \equiv 1 - 2n_w(W_1^n) \equiv 1 \pmod{2}. $$
	So the second equation is automatically satisfied.
	
For the first equation in \eqref{eq:2grdings_are_agree}, we prove by induction on the length of the primitive disk word $ W_1^n $. For the four elementary disk words, see $ X_1 Y \Xm_2 $ as example, we have $ F(x_1) - F(x_2) = 0 - (-1) = 1 $; on the other hand, the corresponding subword in $ R'(a, Y) $ is $ Y^{-k} a_1 Y a^{-1}_2 Y^{k} $, thus $ A(a_1) - A(a_2) = 1 $ by definition. In general, we see a upward positive disk word $ W_1^n = X_1 Y^l X_2 \cdots \Xm_{n-1} \Ym^{l'} \Xm_n $ as example, where $ |l - l'| \leq 1 $. By Step \ref{step_basepoint} and \eqref{eqn:alexander_relative}, we have
	$$ \begin{aligned}
		F(x_1) - F(x_n) & = n_z(S) + \sum_{i=1}^d n_z(W_{k_i}^{k_{i+1}}) - n_w(S) - \sum_{i=1}^d n_w(W_{k_i}^{k_{i+1}}) \\
		& = n_z(S) - n_w(S) + \sum_{i=1}^d F(x_{k_i}) - F(x_{k_{i+1}}) \\
		& = n_z(S) - n_w(S) + F(x_2) - F(x_{n-1}).
	\end{aligned} $$ 
	On the other hand, the image of $ W_1^n $ in the relator $ R'(a, Y) $ is
	$$ Y^{-k}a_1 Y^{l-k} a_2 \cdots a^{-1}_{n-1} Y^{k-l'} a^{-1}_n Y^{k}. $$
	Thus
	$$ \begin{aligned}
		A(a_1) - A(a_n) & = A(a_1) - A(a_2) + A(a_{n-1}) - A(a_n) + A(a_2) - A(a_{n-1}) \\
		& = l - l' + A(a_2) - A(a_{n-1}).
	\end{aligned} $$
	By the inductive hypothesis, we only need to show that
	$$ n_z(S) - n_w(S) = l - l'. $$
	This is the direct from the definition of $ P(S) $ in Step \ref{step_basepoint_key}. 
	
	Therefore, from the equations \eqref{eq:2grdings_are_agree}, we obtain
	$$ \sum_i (-1)^{M(x_i)} t^{F(x_i)} = \sum_i (-1)^{S(a_i)} t^{A(a_i) + c}, $$
	for some integer $ c $. 
\end{proof}

The presentation \eqref{eqn:trefoil_presentation} does not come from a $(1,1)$ Heegaard diagram of the trefoil knot. Algorithm \ref{algo:main} fails on Step \ref{step_absolute} for \eqref{eqn:trefoil_presentation}: label the $X$-letters in order, and the resulting relator after removing primitive disk words with $ n_w = 0 $ is $ X_2 X_3 Y \Xm_4 \Ym $, and then we cannot remove any letter since
$$ w(x_4) - w(x_3) = w(x_4) - w(x_2) = 1. $$

\begin{mydef}
	Let $ P=\langle X, Y \,|\, R(X, Y) \rangle $ be a quasi-geometric presentation of a group $G$. $P$ is \emph{pseudo-geometric} if
	\begin{enumerate}
		\item The relative Alexander grading can be normalized to satisfy \eqref{eqn:alexander_grading_absolute}.
		\item The relator obtained by the process of deleting subwords in Algorithm \ref{algo:main} Step \ref{step_absolute} has exactly one $X$-letter.
	\end{enumerate}
\end{mydef}

We remark that \eqref{eqn:alexander_grading_absolute} automatically holds for knot groups by Proposition \ref{prop:Euler2Alexander}, but not for general groups. For example, the presentation $\langle X, Y \,|\, YX^2 \Ym XY \Xm^2 \rangle$ is quasi-geometric, but \eqref{eqn:alexander_grading_absolute} does not hold.

For a pseudo-geometric group presentation $G = \langle X, Y\,|\,R(X, Y) \rangle $, Algorithm \ref{algo:main} applies to give a chain complex with trivial differential and two absolute gradings $F$ and $M$. Denote its homology by $H(R)$, and its Poincar\'e polynomial is
\begin{equation}\label{eqn:poincare_poly_pseudo_geometric}
P_R(t, q) = \sum_{i} t^{F(x_i)} q^{M(x_i)}.
\end{equation}

\begin{cor}\label{cor:algorithm_alex_polynomial_pseudo_geometric}
Let $\class{X, Y\,|\,R(X, Y)}$ be a pseudo-geometric presentation of a $ (1,1) $ knot $ K $ in $ S^3 $, then its Poincar\'e polynomial \eqref{eqn:poincare_poly_pseudo_geometric} satisfies
	$$ P_R(t, -1) = \Delta_K(t). $$
\end{cor}

However, a pseudo-geometric presentation of $\pi_1(S^3 \setminus K)$ for a $(1,1)$ knot $K$ does not always come from a $(1,1)$ Heegaard diagram. Algorithm \ref{algo:main} can compute the knot Floer homology of $K$ from a pseudo-geometric presentation correctly in some cases (Example \ref{example:not_geometric_but_homology}) and incorrectly in other cases (Example \ref{example:pseudo_but_not_homology}).

\begin{example}\label{example:not_geometric_but_homology}
 The fundamental group of the torus knot $ T_{p,q} $ has the following presentation (by \cite{MP71}):
 $$  \langle W, Z \,|\, (W^p Z^{t})^s Z \rangle, $$
 where $ q = ts + 1 $. Consider the case that $ p = 2, q = 7, s = 2, t = 3 $. By setting $ Z \mapsto  \Xm, W \mapsto X^2 Y $, we get the following presentation of $\pi_1\bracket{S^3\setminus T_{2,7}}$:
 \begin{equation}\label{pre:bad_T[27]}
 \langle X, Y \,|\,  Y X^2Y \Xm Y X^2Y \Xm^2  \rangle.
 \end{equation}
Note that all primitive disk words are elementary, so this presentation is pseudo-geometric. Algorithm \ref{algo:main} applied to \eqref{pre:bad_T[27]} is computed as $ \widehat{HFK}(S^3, T_{2,7}) $. However, the presentation \eqref{pre:bad_T[27]} does not come from a $(1,1)$ Heegaard diagram.
\end{example}

\begin{example}\label{example:pseudo_but_not_homology}
	The following is a pseudo-geometric presentation of $ \pi_1(S^3 \setminus D_+(T_{2,3}, 6)) $:
	\begin{equation}\label{eq:pre_knotD6_bad}
		\langle X, Y \,|\, X^4 \Ym \Xm^{3} Y X^4 Y \Xm^{3} \Ym X \Ym \Xm^{3} Y X^4 Y \Xm^{3} \Ym \rangle,
	\end{equation}
	which is obtained from the presentation \eqref{eq:pre_knotD6} by performing the transformation $ Y \mapsto XY $. Label each $ X $-letter in order. To see that it is pseudo-geometric, consider the disk word $ D_1 $ from the 5-th to the 10-th $ X $-letter, and the disk word $ D_2 $ from the 12-th to the 4-th $ X $-letter, both primitive. Applying the algorithm gives
	$$ P(D_1) = (1, 0), \, P(D_2) = (-1, 0), $$
	so that both are deleted in Step \ref{step_absolute} and the resulting relator is $ \Ym X_{11} Y $, which has a unique $ X $-letter.

Algorithm \ref{algo:main} yields $ F = M $ for all generators, and the resulting Poincar\'e polynomial is $ P(t,q) = 6q^{-1}t^{-1} + 13 + 6qt $, which is different from \eqref{eq:pp_knotD6}. It follows that \eqref{eq:pre_knotD6_bad} does not come from a $(1,1)$ Heegaard diagram of $D_+(T_{2,3}, 6)$.
\end{example}


Algorithm \ref{algo:main} does not a priori require the group to be the fundamental group of a $(1,1)$ knot. We can apply Algorithm \ref{algo:main} for any pseudo-geometric two-generator one-relator group presentations. But we do not know whether it yields any invariant of the group, and what properties it captures.



\begin{thebibliography}{10}

\bibitem{ACS21}
Antonio~Alfieri,  Daniele~Celoria and Andr\'as~Stipsicz.
\newblock Upsilon invariants from cyclic branched covers.
\newblock {\em Studia Scientiarum Mathematicarum Hungarica}, 58(4), 457--488, 2021.
	
\bibitem{AFW15}
Matthias Aschenbrenner, Stefan Friedl, and Henry Wilton.
\newblock {\em {$3$}-manifold groups}.
\newblock EMS Series of Lectures in Mathematics. European Mathematical Society (EMS), Z\"{u}rich, 2015.

\bibitem{BZ23}
Fraser Binns and Hugo Zhou.
\newblock $(1,1)$ almost {L}-space knots.
\newblock {\em arXiv:2304.11475v1 [math.GT]}, 2023.

\bibitem{CGH12b}
Vincent~Colin, Paolo~Ghiggini and Ko~Honda.
\newblock The equivalence of {H}eegaard {F}loer homology and embedded contact homology via open book decompositions {I}.
\newblock {\em arXiv:1208.1074. [math.GT]}, 2012.

\bibitem{CGH12c}
Vincent~Colin, Paolo~Ghiggini and Ko~Honda.
\newblock The equivalence of {H}eegaard {F}loer homology and embedded contact homology via open book decompositions {II}.
\newblock {\em arXiv:1208.1077. [math.GT]}, 2012.

\bibitem{CGH12a}
Vincent~Colin, Paolo~Ghiggini and Ko~Honda.
\newblock The equivalence of {H}eegaard {F}loer homology and embedded contact homology {III}: from hat to plus.
\newblock {\em arXiv:1208.1526. [math.GT]}, 2012.

\bibitem{Dol92}
Helmut Doll.
\newblock A generalized bridge number for links in {$3$}-manifolds.
\newblock {\em Mathematische Annalen}, 294(4):701--717, 1992.


\bibitem{Flo88a}
Andreas Floer.
\newblock An instanton-invariant for 3-manifolds.
\newblock {\em Communications in Mathematical Physics}, 118(2):215--240, 1988.

\bibitem{Fuj96}
Hirozumi Fujii.
\newblock Geometric indices and the {A}lexander polynomial of a knot.
\newblock {\em Proceedings of the American Mathematical Society}, 124(9):2923--2933, 1996.

\bibitem{GLV18}
Joshua Evan Greene, Sam Lewallen, and Faramarz Vafaee.
\newblock {$(1,1)$} {L}-space knots.
\newblock {\em Compositio Mathematica}, 154(5): 918--933, 2018.

\bibitem{GMM05}
Hiroshi Goda, Hiroshi Matsuda, and Takayuki Morifuji.
\newblock Knot {F}loer homology of {$(1,1)$}-knots.
\newblock {\em Geometriae Dedicata}, 112:197--214, 2005.

\bibitem{Him24}
Keisuke Himeno.
\newblock New family of hyperbolic knots whose {U}psilon invariants are convex.
\newblock {\em 2403.13342v1 [math.GT]}, 2024.

\bibitem{HO05}
Matthew Hedden and Philip Ording.
\newblock The {O}zsv\'{a}th-{S}zab\'{o} and {R}asmussen concordance invariants are not equal.
\newblock {\em American Journal of Mathematics}, 130(2):441--453, 2008.

\bibitem{Hut02}
Michael Hutchings.
\newblock An index inequality for embedded pseudoholomorphic curves in symplectizations.
\newblock {\em Journal of the European Mathematical Society (JEMS)}, 4(4):313--361, 2002.

\bibitem{HT07}
Michael Hutchings and Clifford~Henry Taubes.
\newblock Gluing pseudoholomorphic curves along branched covered cylinders. {I}.
\newblock {\em The Journal of Symplectic Geometry}, 5(1):43--137, 2007.

\bibitem{HT09}
Michael Hutchings and Clifford~Henry Taubes.
\newblock Gluing pseudoholomorphic curves along branched covered cylinders. {II}.
\newblock {\em The Journal of Symplectic Geometry}, 7(1):29--133, 2009.



\bibitem{KM11}
Peter~B~Kronheimer and Tomasz~S~Mrowka.
\newblock Knot homology groups from instantons.
\newblock {\em Journal of Topology}, 4(4):835--918, 2011.

\bibitem{KLT10a}
\c{C}a\u{g}atay Kutluhan, Yi-Jen Lee, and Clifford~Henry Taubes.
\newblock {${\rm HF}={\rm HM}$}, {I}: {H}eegaard {F}loer homology and {S}eiberg-{W}itten {F}loer homology.
\newblock {\em Geometry \& Topology}, 24(6):2829--2854, 2020.

\bibitem{KLT10b}
\c{C}a\u{g}atay Kutluhan, Yi-Jen Lee, and Clifford~Henry Taubes.
\newblock {${\rm HF}={\rm HM}$}, {II}: {R}eeb orbits and holomorphic curves for the ech/{H}eegaard {F}loer correspondence.
\newblock {\em Geometry \& Topology}, 24(6):2855--3012, 2020.

\bibitem{KLT10c}
\c{C}a\u{g}atay Kutluhan, Yi-Jen Lee, and Clifford~Henry Taubes.
\newblock {$\rm HF{=}HM$}, {III}: holomorphic curves and the differential for the ech/{H}eegaard {F}loer correspondence.
\newblock {\em Geometry \& Topology}, 24(6):3013--3218, 2020.

\bibitem{KLT11}
\c{C}a\u{g}atay Kutluhan, Yi-Jen Lee, and Clifford~Henry Taubes.
\newblock {${\rm HF}={\rm HM}$}, {IV}: {T}he {S}eiberg-{W}itten {F}loer homology and ech correspondence.
\newblock {\em Geometry \& Topology}, 24(7):3219--3469, 2020.

\bibitem{KLT12}
\c{C}a\u{g}atay Kutluhan, Yi-Jen Lee, and Clifford~Henry Taubes.
\newblock {${\rm HF}={\rm HM}$}, {V}: {S}eiberg-{W}itten {F}loer homology and handle additions.
\newblock {\em Geometry \& Topology}, 24(7):3471--3748, 2020.



\bibitem{MP71}
James~Mccool and Alfred~Pietrowski. 
\newblock On free products with amalgamation of two infinite cyclic groups. 
\newblock {\em Journal of Algebra}, 18(3):377–383, 1971.

\bibitem{Nie19}
Zipei~Nie. 
\newblock Topologically slice (1,1)-knots which are not smoothly slice.
\newblock {\em arXiv:1901.07774 [math.GT]}, 2019.

\bibitem{Ord06}
Philip Ording.
\newblock {\em On knot {F}loer homology of satellite $ (1,1) $ knots}.
\newblock ProQuest LLC, Ann Arbor, MI, 2006.
\newblock Thesis (Ph.D.)--Columbia University.

\bibitem{Ord13}
Philip Ording.
\newblock Constructing doubly-pointed {H}eegaard diagrams compatible with {$(1,1)$} knots.
\newblock {\em Journal of Knot Theory and its Ramifications}, 22(11):1350071, 26 pp, 2013.

\bibitem{OS03}
Peter Ozsv{\'a}th and Zolt{\'a}n Szab{\'o}.
\newblock Absolutely graded {F}loer homologies and intersection forms for four-manifolds with boundary.
\newblock {\em Advances in Mathematics}, 173(2):179--261, 2003.

\bibitem{OS04c}
Peter Ozsv{\'a}th and Zolt{\'a}n Szab{\'o}.
\newblock Heegaard diagrams and holomorphic disks.
\newblock In {\em Different Faces of Geometry}, volume~3 of {\em International Series of Numerical Mathematics}, pages 301--348. Kluwer/Plenum, New York, 2004.

\bibitem{OS04a}
Peter Ozsv{\'a}th and Zolt{\'a}n Szab{\'o}.
\newblock Holomorphic disks and knot invariants.
\newblock {\em Advances in Mathematics}, 186(1):58--116, 2004.

\bibitem{OS01a}
Peter Ozsv\'{a}th and Zolt\'{a}n Szab\'{o}.
\newblock Holomorphic disks and topological invariants for closed three-manifolds.
\newblock {\em Annals of Mathematics. Second Series}, 159(3):1027--1158, 2004.

\bibitem{OS01b}
Peter Ozsv\'{a}th and Zolt\'{a}n Szab\'{o}.
\newblock Holomorphic disks and three-manifold invariants: properties and applications.
\newblock {\em Annals of Mathematics. Second Series}, 159(3):1159--1245, 2004.

\bibitem{OS11}
Peter Ozsv\'{a}th and Zolt\'{a}n Szab\'{o}.
\newblock Knot Floer homology and rational surgeries.
\newblock {\em Algebraic \& Geometric Topology}, 11(1):1--68, 2011.

\bibitem{Per02}
Grisha Perelman.
\newblock The entropy formula for the Ricci flow and its geometric applications.
\newblock {\em arXiv:math/0211159v1 [math.DG]}, 2002.

\bibitem{Per03a}
Grisha Perelman.
\newblock Finite extinction time for the solutions to the Ricci flow on certain three-manifolds.
\newblock {\em arXiv:math/0307245v1 [math.DG]}, 2003.

\bibitem{Per03b}
Grisha Perelman.
\newblock Ricci flow with surgery on three-manifolds.
\newblock {\em arXiv:math/0303109v1 [math.DG]}, 2003.

\bibitem{Per08}
Timothy~Perutz.
\newblock Hamiltonian handleslides for Heegaard Floer homology.
\newblock {\em Proceedings of G\"okova Geometry-Topology Conference 2007}, 15--35.
\newblock G\"okova Geometry/Topology Conference (GGT), G\"okova, 2008.

\bibitem{Ras03}
Jacob Rasmussen.
\newblock {\em Floer homology and knot complements}.
\newblock ProQuest LLC, Ann Arbor, MI, 2003.
\newblock Thesis (Ph.D.)--Harvard University.


\bibitem{Rol76}
Dale Rolfsen.
\newblock {\em Knots and Links}.
\newblock Mathematics Lecture Series, No. 7. Publish or Perish Inc., Houston, TX, 1990. Corrected reprint of the 1976 original.

\bibitem{Tau10a}
Clifford~Henry Taubes.
\newblock Embedded contact homology and {S}eiberg-{W}itten {F}loer cohomology {I}.
\newblock {\em Geometry \& Topology}, 14(5):2497--2581, 2010.

\bibitem{Tau10b}
Clifford~Henry Taubes.
\newblock Embedded contact homology and {S}eiberg-{W}itten {F}loer cohomology {II}.
\newblock {\em Geometry \& Topology}, 14(5):2583--2720, 2010.

\bibitem{Tau10c}
Clifford~Henry Taubes.
\newblock Embedded contact homology and {S}eiberg-{W}itten {F}loer cohomology {III}.
\newblock {\em Geometry \& Topology}, 14(5):2721--2817, 2010.

\bibitem{Tau10d}
Clifford~Henry Taubes.
\newblock Embedded contact homology and {S}eiberg-{W}itten {F}loer cohomology {IV}.
\newblock {\em Geometry \& Topology}, 14(5):2819--2960, 2010.

\bibitem{Tau10e}
Clifford~Henry Taubes.
\newblock Embedded contact homology and {S}eiberg-{W}itten {F}loer cohomology {V}.
\newblock {\em Geometry \& Topology}, 14(5):2961--3000, 2010.

\bibitem{Thu82}
William~P. Thurston.
\newblock Three-dimensional manifolds, {K}leinian groups and hyperbolic geometry.
\newblock {\em American Mathematical Society. Bulletin. New Series},
6(3):357--381, 1982.

\bibitem{Wit94}
Edward~Witten.
\newblock Monopoles and four-manifolds.
\newblock {\em Mathematical Research Letters}, 1(6), 769--796, 1994.

\end{thebibliography}
\end{document}